\documentclass[12pt]{article}
\usepackage{amsmath, amsthm, xspace, amssymb}
\usepackage[english]{babel}

\usepackage{graphicx}

\usepackage{ifpdf} 
\ifpdf
\usepackage{epstopdf}
\usepackage{hyperref}
\else
\usepackage[hypertex]{hyperref}
\fi
\usepackage[arrow,matrix, curve]{xy}

\newtheorem{thm}{Theorem}
\newtheorem{lemma}[thm]{Lemma}
\newtheorem{prop}[thm]{Proposition}
\newtheorem{cor}[thm]{Corollary}
\theoremstyle{definition}
\newtheorem{rem}[thm]{Remark}
\newtheorem{Def}[thm]{Definition}

\DeclareMathOperator{\Aut}{Aut}
\DeclareMathOperator{\tft}{tft}
\DeclareMathOperator{\id}{id}

\newcommand\erf[1]      {\mbox{(\ref{#1})}}
\newcommand\void[1] {}

\renewcommand\hom  {\mbox{Hom}}
\renewcommand\phi  {\varphi}

\def\be            {\begin{equation}}
\def\bearl         {\begin{array}{l}} 
\def\bearll        {\begin{array}{ll}}
\def\bearlll       {\begin{array}{lll}}
\def\catdim                {\mathcal{D}}
\def\C                     {\ensuremath{\mathcal{C}}\xspace}
\def\CC                    {\ensuremath{\mathcal{C}\boxtimes\mathcal{C}}\xspace}
\def\CG                    {\ensuremath{\C^G}\xspace}
\def\CX                                    {\ensuremath{\C^{\mathcal{X}}}\xspace}
\def\CZ                    {\CX}
\def\cobd                  {\ensuremath{{\bf cob(d)}}\xspace}
\def\cobtwo                {\ensuremath{{\bf cob(2)}}\xspace}
\def\congto                {\stackrel{\cong}{\to}}
\def\cover         {\mathcal{F}}
\def\defn         {\stackrel{\mathrm{def}}=}
\def\ext                   {\ensuremath{{\bf Ext}}\xspace}
\def\ee            {\end{equation}}
\def\eear          {\end{array}}
\def\gdcob                 {\ensuremath{{\bf Gcob(d)}}\xspace}
\def\Gext                  {\ensuremath{{\bf GExt}}\xspace}
\def\gtwocob       {\ensuremath{{\bf Gcob(2)}}\xspace}
\def\gtft                  {\tft^G}
\def\homc                  {\hom_{\C}}
\def\homcc                 {\hom_{\C\boxtimes\C}}
\def\homcz         {\hom_{\CZ}}
\def\I                     {\mathcal{I}}
\def\la                    {\langle}
\def\oti                   {\otimes}
\def\otik                  {\otimes_{\Bbbk}}
\def\unit          {\mathbf{1}}  
\def\R                     {\mathcal{R}}
\def\ra                    {\rangle}
\def\rag                   {\rangle_{\X}}
\def\tauzz                 {\tau^{\X}}
\def\tftR                  {\tft_R}
\def\vect          {{\mathcal V}\mbox{\sl ect}}
\def\vectk         {{\mathcal V}\mbox{\sl ect}_\Bbbk}
\def\X                     {\ensuremath{\mathcal{X}}\xspace}
\def\ZZ                    {\ensuremath{\mathbb{Z}/2}\xspace}

\def\B                          {{\bf B}}
\def\Z                          {{\bf Z}}
\def\S                          {{\bf S}}
\def\F                          {{\bf F}}

\def\bp            {\begin{picture}}
\newcommand\pic[1] {\includegraphics{#1}}
\def\ep            {\end{picture}}

\newcommand\rib[1] {\includegraphics{#1}}

\newcommand{\BIGOP}[1]{\mathop{\mathchoice%
{\raise-0.22em\hbox{\LARGE $#1$}}%
{\raise-0.05em\hbox{\Large $#1$}}{\hbox{\large $#1$}}{#1}}}
\newcommand{\bigboxtimes}{\BIGOP{\boxtimes}}

\begin{document}

\thispagestyle{empty}
\def\thefootnote{\fnsymbol{footnote}}
\begin{flushright}
   {\sf ZMP-HH/10-9}\\
   {\sf Hamburger$\;$Beitr\"age$\;$zur$\;$Mathematik$\;$Nr.$\;$ 367}\\[2mm]
   April 2010
\end{flushright}
\vskip 2.0em
\begin{center}\Large 
   A GEOMETRIC CONSTRUCTION FOR PERMUTATION EQUIVARIANT CATEGORIES FROM MODULAR FUNCTORS

\end{center}\vskip 1.4em
\begin{center}
  Till Barmeier\,$^{\,a}$ and
  ~Christoph Schweigert\,$^{\,a}$\footnote{\scriptsize 
  ~Email addresses: \\
  $~$\hspace*{2.4em}barmeier@math.uni-hamburg.de, 
  Christoph.Schweigert@uni-hamburg.de}
\end{center}

\vskip 3mm

\begin{center}\it$^a$
  Organisationseinheit Mathematik, \ Universit\"at Hamburg\\
  Bereich Algebra und Zahlentheorie\\
  Bundesstra\ss e 55, \ D\,--\,20\,146\, Hamburg
\end{center}

\vskip 2.5em
\begin{abstract} \noindent
Let $G$ be a finite group. Given a finite $G$-set $\X$ 
and a modular tensor category $\C$, we construct a
weak $G$-equivariant fusion category \CX, called the permutation
equivariant tensor category. The construction is geometric
and uses the formalism of modular functors.
As an application, we concretely work out a complete set of 
structure morphisms for $\ZZ$-permutation equivariant categories,
finishing thereby a program we initiated in 
\cite{bfrs}.  
\end{abstract}

\setcounter{footnote}{0} \def\thefootnote{\arabic{footnote}} 

\newpage

\tableofcontents

\newpage

\section{Introduction}

The correspondence between commutative Frobenius algebras
over a field $k$ and two-dimensional topological field theories
with values in $k$-vector spaces is by now a classic result of mathematical physics
\cite{Atiyah:1989vu,Dijkgraaf-thesis}. Its most precise form is the 
assertion \cite{kock} 
that  the  tensor categories of two-dimensional topological field 
theories and of commutative Frobenius algebras are 
equivalent as symmetric tensor categories. 

Various generalizations of this theorem have been
addressed. On the one hand side, given any (finite) group $G$,  
$G$-equivariant two-dimensional topological field theories
have lead to the notion of a $G$-Frobenius algebra
\cite{mose,turaevhomotopy}. For three-dimensional topological
field theories on the other hand, one is lead to the
algebraic structure of a
modular tensor category. A structure related to
three-dimensional topological
field theories that is more appropriate for our purposes is given by the
notion of a modular functor. For
any abelian category \C satisfying suitable finiteness 
conditions, the following correspondences have been
established \cite{BK}: 
\begin{itemize}
\item 
\C-extended genus $0$ modular functors correspond 
to structures of (weakly) ribbon categories on  \C.

\item 
Higher genus modular functors correspond to structures 
of a modular category on \C. 
\end{itemize}

We now fix a finite group $G$.
Equivariant versions of three-dimensional topological field 
theories have been constructed from crossed group categories 
in \cite{turaevhomotopy3d}; for the related notion of
a $G$-modular functor see \cite{kpGMF}.
A partial generalization of the preceding statements  asserts
\cite{kpGMF} that
the structure of a (weakly) $G$-equivariant fusion category 
on a given $G$-equivariant abelian category  $\CG$ is 
equivalent to a $\CG$-extended $G$-equivariant genus $0$ 
modular functor.

The present paper is devoted to the construction of
a $G$-equivariant fusion category $\CX$ from 
a finite $G$-set $\X$ and a modular tensor category
$\C$. We address the problem by constructing 
from these data a $G$-modular functor.

Let us pause to explain the importance of
this construction: a general construction \cite{kirillov}
allows to associate to any $G$-equivariant modular category \CG a
modular category $\CG//G$, the orbifold category.
In the special case of the permutation group $G=S_N$
acting on the set $\X=\underline N:=\{1,2,\ldots,N\}$ 
of $N$ elements, one obtains
from permutation equivariant categories 
permutation orbifolds. These categories \cite{Fuchs:2003yk}
enter in the construction of boundary conditions
for tensor product theories that break permutation symmetries,
so-called permutation branes \cite{Recknagel:2002qq}.
Moreover, they conveniently encode refined
aspects of the family of representations of mapping
class groups that is associated to the modular tensor category
$\C$. This explains the role of permutation orbifolds 
in Bantay's approach to 
the congruence subgroup conjecture (\cite{BantayPO,BantayKernel}).
(For a different proof of the congruence subgroup conjecture
that is based on generalized Frobenius-Schur indicators,
see \cite{Ng-Schauenburg}.)
We expect that the module categories over
$\C^{\boxtimes N}$ constructed in this work describe
permutation modular invariants
on $\C^{\boxtimes N}$;  our construction can thus also be seen as
a first step towards showing that these modular invariants
are physical.

We now summarize the content of this paper:
given a finite $G$-set $\X$, we construct a symmetric monoidal 
functor $\cover_\X$ from the 
category $\gdcob$ of $G$-cobordisms to the category $\cobd$ 
of cobordisms. This functor assigns to a principal $G$-cover $(P\to M)$
the total space of the associated bundle
\be
\cover_\X(P\to M):=\X\times_GP=\X\times P/((g^{-1}x,p)\sim (x,gp)) \,\, .
\ee
Pulling back topological field theories along this functor
$\cover_\X$, we find $G$-equivariant theories.
This functor is introduced in subsection \ref{sec:GTFT} and used in
subsections \ref{sec:2dgeom} and \ref{sec:2dalg} to study 
geometric and algebraic properties of two-dimensional 
$G$-equivariant field theories.

These sections also
contain necessary preparation for the construction
of a $\CG$-extended modular functor for which we need
to know the corresponding category $\CG$ as a 
$G$-equivariant abelian category. We gain the necessary
insight in the structure for the correct ansatz by first
representing a given two-dimensional TFT by the corresponding commutative 
Frobenius algebra $(R,m,\eta,\Delta, \epsilon)$ and then
reading  off the full $G$-Frobenius algebra $A$ of the 
induced $G$-TFT.

In section \ref{sec:GMF} we present theorem \ref{thm:GMF}, one of
our main results: we use the cover functor  $\cover_\X$ to obtain a
$G$-equivariant modular functor for every $G$-set $\X$
and modular tensor category $\C$.

We illustrate the situation by the following diagram:
\be
\raisebox{-15pt}{
\xymatrix{
\mbox{modular category $\C$}\ar[rr]&&\mbox{$\C$-extended modular functor $\tau$}\ar[dd]^{\mbox{$\cover_\X$-construction}}\ar@/_2pc/@{-->}[ll]\\
\\
\mbox{$G$-modular category $\CX$}\ar[rr]&&\mbox{$\CX$-extended modular functor $\tau^G$}\ar@/^2pc/@{-->}[ll]
}
}\\[2em]
\ee
The upper arrow that points to the left is dashed, since a
$\C$-extended modular functor endows an abelian category $\C$ only
with a weak duality \cite{BK} while on a modular tensor category one
has a strong duality. The lower arrow pointing to the left
is dashed not only for this reason;
moreover the algebraic structure corresponding to higher 
genus $G$-equivariant modular functors has, so far,
not yet fully been worked out.

There is no genus zero version of our results and the
the requirement on $\C$ to be modular cannot be weakened:
the modular functor corresponding to $\C$ has to be 
applied to total spaces of covers of manifolds which can
have higher genus, even if the base manifold is of genus 
$0$.

In section \ref{sec:z2} we use these results to derive in full detail the structure of a $\ZZ$-equivariant fusion category obtained from the permutation action of the group $\ZZ$
on the set of two elements, completing thus the program initiated 
in \cite{bfrs}. In contrast to the ad hoc ansatz used 
in \cite{bfrs}, the geometric structure unraveled in this paper
provides clear guiding principles to write down a consistent
set of constraint morphisms.

We fix the following conventions for this paper:
$G$\ is a finite group and $\X$ is a finite $G$-set. 
$\Bbbk$ is an algebraically closed field of characteristic 
zero. All manifolds are  smooth oriented manifolds and all 
maps are smooth and orientation preserving. 
We freely use the graphical calculus for morphisms in ribbon categories for which we refer to \cite{joSt5,TFT1}.

\subsubsection*{Acknowledgements}
We would like to thank Thomas Nikolaus for helpful discussions.
The authors are partially supported by the DFG Priority Program SPP 1388
``Representation theory''.


\section{Cover functors and two-dimensional topological field
theories}
\subsection{$G$-cobordisms and cover functors}\label{sec:GTFT}

In this subsection we  recall the notion of a two-dimensional  $G$-equivariant 
topological field theory.
We first define a category \gdcob\ of cobordisms with $G$-covers.
Throughout this paper,
cobordism categories will be true categories rather than higher
categories; correspondingly, topological field theories will not
be extended or $n$-tier topological field theories.

An object $(P\to\Sigma,e)$ of \gdcob consists of a $(d-1)$-dimensional closed oriented manifold $\Sigma$,
together with a principal left $G$-bundle $P$ on $\Sigma$. For technical
reasons, we also fix a marked point $e$ on each connected
component of $P$. 
For any oriented manifold $M$, the manifold with the opposite orientation will be denoted by $\overline M$. It is sufficient to choose
orientations of the base manifolds:

\begin{lemma}\label{lemma:cover-oriented}
Let $\pi:P\to M$ be a discrete cover of an oriented manifold $M$. Then the orientation of $M$
induces a canonical orientation on $P$. 
\end{lemma}

\begin{proof}
The global orientation of $M$ can be represented by global section of the orientation bundle on $M$, which can be pulled back along the local diffeomorphism $\pi$ to a global section of the orientation bundle of $P$.
\end{proof}

The morphisms in \gdcob  from $(P_1\to\Sigma_1,e_1)$ to  $(P_2\to\Sigma_2,e_2)$
are diffeomorphism classes of cobordisms of the base and the total space of 
the principal bundles. More precisely, consider pairs $(M,E)$, consisting of a left principal
$G$-bundle $E$ over a $d$-dimensional oriented manifold $M$
and an orientation preserving diffeomorphism of $G$-bundles
from the restriction $\partial E\to\partial M$ to the
$G$-bundles on the boundary, i.e.\ diffeomorphisms
$$\partial M\stackrel\cong\to \Sigma_1\sqcup\overline{\Sigma_2}\,\, ,
\qquad
E|_{\Sigma_1}\stackrel\cong\to P_1
\, \text{ and }\,\, 
E|_{\overline{\Sigma_2}}\stackrel\cong\to\overline{P_2}
\,\, . $$
The morphisms in \gdcob are now obtained by modding out diffeomorphisms of $E$.
We write $(E\to M)$ for a representative.

The composition of morphisms is gluing along the boundary of both the base and the total space:
Let $(P_i\to\Sigma_i,e_i)$ for $i=1,2,3$ be objects in \gdcob and $(E\to M):(P_1\to\Sigma_1,e_1)\to (P_2\to\Sigma_2,e_2)$ and 
$(E'\to M'):(P_2\to\Sigma_2,e_2)\to (P_3\to\Sigma_3,e_3)$ be cobordisms. Then the manifold 
$M\sqcup_{\Sigma_2}M'$ obtained by gluing the base spaces is naturally equipped with 
the $G$-principal bundle $E\sqcup_{P_2}E'$. 
The identity on $(P\to\Sigma,e)$ is the diffeomorphism class of the cylinder over $\Sigma$ with trivial $G$-cover. 
The category \gdcob has a natural structure of a symmetric 
monoidal category, with tensor product given by disjoint union of manifolds and bundles. The  empty set with empty $G$-bundle
is the tensor unit.

We comment on the role of the marked point $e$ on the $G$-cover $P$ of an object
$(P\to\Sigma,e)$ which we have chosen as an auxiliary datum. Its projection on $\Sigma$ determines a base point $x\in \Sigma$. 
Moreover, it determines an identification of the fibre $P_x$ over $x$ with $G$. We did not choose marked points on the covers of
morphisms.
Thus different choices for $e\in P$ give objects that are isomorphic in \gdcob with the isomorphism
being the cylinder with the trivial $G$-cover.
In the case of the trivial group $G=1$, one can forget the marked point and obtains an
equivalence of categories of \gdcob to the usual category \cobd of manifolds and cobordisms.

\begin{Def}\label{def:gtft}
The category of \emph{$d$-dimensional $G$-equivariant topological field theories} (or $G$-TFTs for short) is the category of symmetric monoidal functors
\be
\gtft:\gdcob\to\vectk.
\ee
with monoidal natural transformations as morphisms.
\end{Def}

One ingredient in our construction of $G$-TFTs is a finite left $G$-set $\X$.
In the construction we will need to make choices for all partitions of \X. To this end,
we fix an order on $\X$ as an auxiliary datum.

For any object $(P\to\Sigma,e)$ of \gdcob, we consider the smooth
$(d-1)$-manifold $\X\times_GP$, where
$\X\times_GP=(\X\times P)/((g^{-1}x,p)\sim (x,gp))$. Similarly, we obtain
smooth $d$-manifolds 
for morphisms of \gdcob. For any $G$-bundle $P\to M$, the manifold $\X\times_GP$ is the total space of 
an $|\X|$-fold cover of $M$, we call this cover the associated $G$-cover. We agree to write $[x,p]\in\X\times_GM$ for the equivalence class of $(x,p)\in \X\times M$ for any $G$-manifold $M$.

\begin{prop}\label{coverprop}
Let $\X$ be a finite ordered left $G$-set. The assignment
$(P\to\Sigma,e)\mapsto \X\times_G P$
defines a symmetric monoidal functor
\be
\cover_\X: \,\,\gdcob\to \cobd
\ee
\end{prop}

\begin{proof}
The proof is straightforward, including its most intricate aspect, the fact that gluing of cobordisms is respected.
\end{proof}

\begin{Def}\label{def:coverfunctor}
The functor $\cover_\X$ in proposition \ref{coverprop} is called 
the \emph{$d$-dimensional cover functor} for the $G$-set $\X$.
\end{Def}

\begin{cor}\label{cor:covertft}
Let $\tft:\cobd\to \vectk$ be a topological field theory and let $\X$ be a $G$-set. Then the composite functor
\be
\tft^\X:\,\,
\gdcob\stackrel{\cover_\X}{\to}\cobd\stackrel{\tft}{\to}\vectk
\ee
is a $d$-dimensional $G$-equivariant topological field theory in the sense of definition \ref{def:gtft}.
\end{cor}

\subsection{Covers of two-dimensional cobordisms}
\label{sec:2dgeom}

From now on, we specialize to dimension $d=2$. We recall the 
following definition \cite{BKlego, BK, kpGMF}:

\begin{Def}\label{def:extendedsurf}
\begin{enumerate}
\item
An \emph{extended surface} is a compact oriented smooth two-dimensional manifold $M$, possibly with boundary, 
together with a choice of a marked point  on each connected component of the boundary $\partial M$. The set of boundary components of $M$ is denoted by $A(M)$ and we write extended surfaces as $(M,\{e_a\}_{a\in A(M)})$. A morphism of extended surfaces is a smooth map that preserves marked points.

\item
A \emph{$G$-cover of an extended surface} $(M,\{e_a\}_{a\in A(M)})$ is a pair $(P\to M, \{p_a\}_{a\in A(M)})$, where $P\to M$ is a principal $G$-cover of $M$ and the $p_a$ are marked points in the fibre over $e_a$. A morphism of $G$-covers of extended surfaces is a smooth bundle map that preserves marked points.
\end{enumerate}
\end{Def}

The morphisms in \gtwocob are thus diffeomorphism classes of 
$G$-covers of extended surfaces. We also consider a category
\ext whose objects are extended surfaces and whose
morphisms are orientation preserving 
diffeomorphisms of extended surfaces. Similarly, we have
a $G$-equivariant version, a category \Gext with
$G$-covers of extended surfaces as objects and
orientation preserving
diffeomorphisms of $G$-covers of extended surfaces
as morphisms. 

The cover functor $\cover_\X$ induces a functor 
$\Gext\to\ext$ which we will also denote by $\cover_\X$. 
To see this, let $(E\to M, \{p_a\}_{a\in A(M)})$ be an object in \Gext, i.e. a $G$-cover of an extended surface. 
By definition, $E$ is endowed with 
a marked point $p_a$ for each connected component of the boundary of $M$.
This yields on the boundary of $\cover_\X(E\to M)$ the marked points $[x,p_a]$ 
for every element $x\in\X$ of the $G$-set and every boundary component $a$ of $M$. To turn $\cover_\X(E\to M)$ into an extended surface, we need to
choose just a single point for every connected component of the boundary of $\cover_\X(E\to M)$. At this point, we use the auxiliary
structure of an ordering on the $G$-set $\X$ to choose the point
$[x,p_a]$  with the smallest value of $x\in\X$ in that boundary component.

We will now analyze covers of the following basic manifolds:
\begin{enumerate}
\item 

The vector spaces relevant for a $G$-equivariant topological field theory are given by evaluations of
the TFT functor on $G$-covers of the circle $S^1$.

\item
The multiplicative structure on the vector spaces underlying a topological field theory comes
from the 3-punctured sphere, the so-called pair-of-pants.
To set the stage for the discussion in section 4, we consider covers of the $n$-punctured
sphere.

\end{enumerate}

\subsubsection{Covers of the circle}\label{sssect:covcirc}
For any element $g\in G$, we introduce the
principal $G$-bundle $P_g$ of $S^1$ with total space
$P_g:=\mathbb{R}\times G/(t+2\pi,h)\sim(t,hg)$ and distinguished point $[0_\mathbb{R},1_G]$.
With respect to this point, the monodromy of the bundle is given by $g$. One easily checks that
every principal $G$-bundle over $S^1$ is isomorphic to $P_g$ for some $g\in G$. 

Given a finite $G$-set $\X$, we define for every $g\in G$ a 
$|\X|$-fold cover of $S^1$ with total space
$E_g:=\mathbb{R}\times\X/(t+2\pi,x)\sim (t,g^{-1}x)$. The following lemma is straightforward: 

\begin{lemma}\label{lemma:EgPg}
For any $g\in G$, the covers $E_g$ and $\X\times_GP_g$ over $S^1$ are isomorphic.
\end{lemma}

As a closed one-dimensional manifold, the total space $E_g$ is a disjoint union of circles. We describe the connected components of
$E_g$:

\begin{lemma}\label{lem:orbits}
For any element $g\in G$ there is a one-to-one correspondence 
between the connected components of the manifold $E_g$ and the 
orbits of $\X$ under the action of the cyclic group 
$\langle g\rangle$.
\end{lemma}
\void{
\begin{proof}
Any connected component of $E_g$ contains a point of the form $[0,x]$ for some $x\in\X$. Let $\gamma:[0,1]\to S^1$ be $\gamma(t)=[2\pi t]\in \mathbb{R}/2\pi\mathbb{Z}\cong S^1$ the path that winds once around $S^1$. Viewing $E_g$ as a cover of $S^1$ we can lift $\gamma$ to a path $\hat\gamma_x:[0,1]\to E_g$ with initial point $\hat\gamma_x(0)=[0,x]$. But then $\hat\gamma_x(1)=[2\pi,x]=[0,g^{-1}x]$. Lifting $\gamma$ to $\hat\gamma_{g^{-1}x}$ allows us to consider a path $\hat\gamma_{g^{-1}x}\star\hat\gamma_x$ that connects the points $[0,x], [0,g^{-1}x]$ and $[0,g^{-2}x]$ of $E_g$. Repeating this procedure yields all points over zero that are contained in the connected component of $[0,x]$, all of these have the form $[0,g^kx]$. Hence the connected component of $E_g$ containing $[0,x]$ gives the $\langle g\rangle$-orbit of $\X$ containing $x$. The choice of another point $[0,x']$ in the same connected component gives the same orbit as then $x'=g^kx$.

Conversely let $o\in O_g$ be an orbit of $\X$. Choose any $x\in o$. This defines a point $[0,x]$ and hence a connected component of $E_g$. By the above arguments this is independent of the precise choice of $x$, as $g^kx$ gives a point in the same connected component.
\end{proof}
}
We denote the set of orbits of $\X$ under the action of the cyclic group $\langle g\rangle\subset G$  by 
$O_g$;  let $b_g:=|O_g|$ be the number of orbits.

The following lemma follows from an easy calculation, as well.
\begin{lemma}\label{lem:orbitmap}
The map
\be
\bearll
E_g&\to E_{hgh^{-1}}\\[0pt]
[t,x]&\mapsto [t,hx]
\eear
\ee
is an isomorphism of covers of $S^1$. The induced map on the 
sets of connected components is given by the map
\be
\bearll
O_g&\to O_{hgh^{-1}}\\
o&\mapsto ho.
\eear
\ee
between the sets of orbits of cyclic groups.
\end{lemma}

\subsubsection{Covers of the $n$-punctured sphere}
\label{sssect:nsphere}

Next, we investigate covers of the $n$-punctured sphere;
to this end, we fix a standard model \cite{BKlego,BK} of this manifold:

\begin{Def}\label{def:stdsphere}
For every $n\in{\mathbb N}$  the \emph{standard sphere} 
$S_{n}$ is the complex sphere
$\overline{\mathbb{C}}$ with standard orientation and with discs of radius $\frac{1}{3}$ centered around the first $n$ positive integers
removed. As marked points on the boundary components of $S_n$, we choose $k-\frac{i}{3}$ for $k=1,\dots,n$.
\end{Def}

We  need $G$-covers of the standard sphere $S_n$; 
as standard models for these
covers, we use the so called \emph{standard blocks} \cite{prince}. 
To construct the standard blocks, we remove from
$S_n$ the straight lines connecting the points $k+\frac{i}{3}$ to the point $\infty$.
The resulting manifold $S_n\!\setminus\!\mbox{cuts}$ is contractible, hence it only has the trivial $G$-cover $((S_n\!\setminus\!\mbox{cuts})\times G\to S_n)$. For any $n$-tuple $g_1,\dots g_n$ of elements in $G$
whose product is the neutral element, we obtain a $G$-cover of $S_n$ by gluing the
$j$-th cut in $(S_n\!\setminus\!\mbox{cuts})\times G$ with the action of $g_j\in G$.
The following picture shows the situation with a view in the direction of the negative imaginary axis:

\be
\raisebox{-20pt}{
\bp(200,150)
\put(0,40){\pic{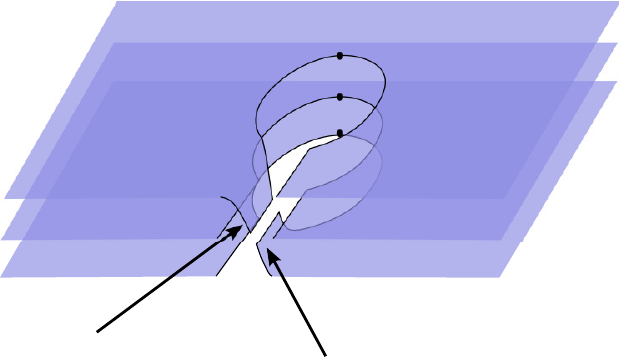}}
\put(0,30){$(z,xg_j)$}
\put(85,25){$(z,x)$}
\ep
}
\ee
We finally have to specify marked points on the boundary components of $((S_n\setminus\mbox{cuts})\times G)/\mbox{gluing}$.
To this end, we choose another $n$-tuple $h_1,\dots,h_n$ of elements in $G$
and take as marked points $[k-\frac i3,h_k]$ for $k=1,\ldots n$.
We write $S_n(g_1,\dots,g_n;h_1,\dots,h_n)$ for these marked $G$-covers over 
$S_n$. In fact, any $G$-cover over an $n$-punctured sphere is diffeomorphic
to one $G$-cover of the form
$(S_n(g_1,\dots,g_n;h_1,\dots,h_n)\to S_n)$. These covers have monodromies $h_ig_i^{-1}h_i^{-1}$ around the $i$-th boundary circle of $S_n$.

In the definition of standard blocks, the orientation of a boundary component depends on whether we consider the component
as ingoing or outgoing. For example, for the pair-of-pants with one outgoing circle the
third circle is given a clockwise orientation.
The cover of $S_3$ that is most important in the following discussion is $S_3(g_1,g_2,(g_1g_2)^{-1};1,1,1)$. When orienting the third boundary component as outgoing, this cover has monodromies $g_1^{-1}$ and $g_2^{-1}$ at the ingoing components and $(g_1g_2)^{-1}$ at the outgoing component.
We sometimes abbreviate $E_{g_1;g_2}:=\cover_\X(S_3(g_1,g_2,(g_1g_2)^{-1};1,1,1)\to S_3)$.

To analyze how the structure of $\cover_\X(S_3(g_1,g_2,(g_1g_2)^{-1};1,1,1)\to S_3)$ depends on the group elements $g_1$ and $g_2$, we introduce 
the following paths in the base manifold, the pair-of-pants $S_3$:
\begin{itemize}
\item $\gamma_{g_1}$, the path with winding number one around the ingoing boundary circle which has monodromy $g_1^{-1}$.
\item $\gamma_{g_2}$ the path winding once around the ingoing boundary circle which has monodromy $g_2^{-1}$.
\item $\gamma_{g_1g_2}$ the path winding once around the outgoing boundary circle which has monodromy $(g_1g_2)^{-1}$.
\item $\alpha, \beta$ open paths connecting the base points of the ingoing circles with the base point of the outgoing circle.
\end{itemize}
\be
\raisebox{15pt}{
\bp(180,145)
\pic{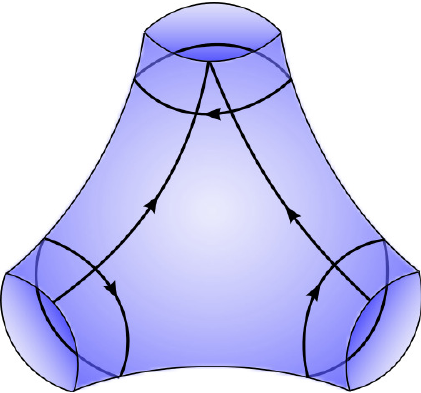}
\put(-80,120) {outgoing}
\put(5,10){ingoing, monodromy $g_1^{-1}$}
\put(-175,-12){ingoing, monodromy $g_2^{-1}$}
\put(-35,90){$\gamma_{g_1g_2}$}
\put(-3,46){$\gamma_{g_1}$}
\put(-129,46){$\gamma_{g_2}$}
\put(-85,60){$\alpha$}
\put(-40,60){$\beta$}
\ep
}
\ee
\\
The following lemma describes the connected components of 
$E_{g_1;g_2}$:

\begin{lemma}\label{lem:poporbits}\mbox{}\\[-2em]
\begin{enumerate}
\item[(i)]
There is a natural bijection between the connected 
components of 
$$\cover_\X(S_3(g_1,g_2,(g_1g_2)^{-1};1,1,1)\to S_3)=E_{g_1;g_2}$$ 
and orbits of the $G$-set $\X$ under the action of the subgroup 
$\langle g_1,g_2\rangle\subset G$ of $G$ generated by 
the elements $g_1$ and $g_2$.
\item[(ii)]
By lemma \ref{lemma:EgPg} the restriction of $E_{g_1;g_2}$ to the boundary with monodromy $g_1^{-1}$ is diffeomorphic to $E_{g_1^{-1}}$
and similarly for the other boundaries.
Let $o$ be a $\langle g_1,g_2\rangle$-orbit of $\X$ and write $E_{g_1;g_2}^o$ for the connected component of
$E_{g_1;g_2}$ corresponding to the orbit $o$. The boundary components of $E_{g_1;g_2}^o$ correspond to precisely those orbits of the
cyclic subgroups $\langle g_1\rangle$, $\langle g_2\rangle$ and $\langle g_1g_2\rangle$ that are contained in the orbit $o$
of the group $\langle g_1,g_2\rangle$.

\item[(iii)] In particular, the number of sheets
of the cover $E_{g_1;g_2}^o\to S_3$ is $|o|$.

\end{enumerate}
\end{lemma}

This is seen by choosing appropriate lifts of the paths $\gamma_{g_1},\gamma_{g_2},\gamma_{g_1g_2},\alpha$ and $\beta$.
We leave the details to the reader as an exercise.
We write $b_{g_1}^o$ for the number of $\langle g_1\rangle$-orbits that are contained in $o$ and similarly for $g_2$ and $g_1g_2$. 
We can now describe the topology of the connected
components of the cover:

\begin{lemma}\label{lem:genus} 
Let $o$ be a $\langle g_1,g_2\rangle$-orbit on $\X$. Then the component $E_{g_1,g_2}^o$ of 
$$\cover_\X(S_3(g_1,g_2,(g_1g_2)^{-1};1,1,1)\to S_3)$$
is a surface of genus
\be\label{eq:genus}
\frac{2-b_{g_1}^o-b_{g_2}^o-b_{g_1g_2}^o+|o|}{2}
\ee
\end{lemma}
\begin{proof}
We have an unramified $|o|$-fold cover of the pair-of-pants $S_3$ with
Euler characteristic $\chi(S_3)=2-3=-1$. Hence, by the theorem of Riemann Hurwitz,
$\chi(E_{g_1,g_2}^o) = -|o|$.

 As described in lemma \ref{lem:poporbits}, the number of boundary components of $E_{g_1;g_2}^o$ is $b_{g_1}^o+b_{g_2}^o+b_{g_1g_2}^o$. This implies formula \erf{eq:genus} for the genus of $E_{g_1,g_2}^o$.
\end{proof}

Finally, we consider the special case of 
the cylinder $S_2$ with principal bundle $S_2(g,g^{-1};1,h)$,
where the first boundary is oriented as ingoing and the second 
boundary as outgoing. Then $S_2(g,g^{-1};1,h)$ has monodromies $g^{-1}$ around the ingoing boundary and $hg^{-1}h^{-1}$ around the outgoing boundary. We identify the ingoing boundary of $(S_2(g,g^{-1};1,h)\to S_2)$ with $(P_{g^{-1}}\to S^1,[0,1_G])$ and the outgoing boundary with $(P_{g^{-1}}\to S^1,[0,h])$. The latter is isomorphic to $(P_{hg^{-1}h^{-1}}\to S^1, [0,1_G])$ under the  map $[t,k]\mapsto[t,kh^{-1}]$ of bundles. Hence the boundaries of $\cover_\X(S_2(g,g^{-1};1,h)\to S_2)$ are isomorphic to $E_{g^{-1}}$ and $E_{hg^{-1}h^{-1}}$ respectively.

\begin{lemma}\label{lem:actioncomponents}
The manifold $\cover_\X(S_2(g,g^{-1};1,h)\to S_2)$ is 
a disjoint union of cylinders. These cylinders interpolate
between the connected component of $E_{g^{-1}}$  corresponding to 
the $\langle g\rangle$-orbit $o$ of $\X$  and  the connected component of $E_{hg^{-1}h^{-1}}$ that corresponds to the $\la hg^{-1}h^{-1}\ra=\langle hgh^{-1}\rangle$-orbit $ho$.
\end{lemma}

\begin{proof}
Consider the following paths in the base cylinder $S_2$:
\begin{itemize}
\item $\gamma_g$ the path winding once around the ingoing boundary circle which has monodromy $g^{-1}$.
\item $\gamma_{hg^{-1}h^{-1}}$ the path winding once around the outgoing boundary circle which has monodromy $hg^{-1}h^{-1}$.
\item $\alpha$ a path connecting the base points of both boundary circles.
\end{itemize}

Let $x\in o$ be an element of the $\langle g\rangle$-orbit $o$. The point $[0,x]$ in $E_{g^{-1}}$ is connected to the points $[0,g^{i}x]$ of $E_{g^{-1}}$ by repeated lifts of $\gamma_g$, similar to the proof of lemma \ref{lem:poporbits}. 
By lifting $\alpha$ to a path $\hat\alpha_{g^{i}x}$ in $\cover_\X(S_2(g,g^{-1};1,h)\to S_2)$ with initial point $\hat\alpha_{g^{i}x}(0)=[0,g^{i}x]\in E_{g^{-1}}$, these points are connected to the points $[0,hg^{i}x]\in E_{hg^{-1}h^{-1}}$, where the map from lemma \ref{lem:orbitmap} is used to identify the outgoing boundary of $\cover_\X(S_2(g,g^{-1};1,h)\to S_2)$ with $E_{hg^{-1}h^{-1}}$. 
These again are connected by lifts of $\gamma_{hg^{-1}h^{-1}}$ only connected to points of the same form. 
By lemma \ref{lem:orbits}, the connected component of $E_{hg^{-1}h^{-1}}$ containing these points, corresponds to the $\langle hg^{-1}h^{-1}\rangle$-orbit $ho$.
\end{proof}


\subsection{Equivariant Frobenius algebras from equivariant topological field
theories}\label{sec:2dalg}

In this subsection, we review the equivariant generalization of the correspondence \cite{kock}
between two-dimensional topological field theories and commutative Frobenius algebras to set the stage for the discussion of $G$-equivariant modular functors.
We then present a decategorified version of the main
construction of this paper.

\subsubsection{From $G$-equivariant TFTs to $G$-Frobenius algebras}
\label{sec:tfttofrob}

We start by recalling definitions from 
\cite{mose, turaevhomotopy}

\begin{Def}\label{def:gfrob}
A \emph{$G$-Frobenius algebra} (or crossed $G$ Frobenius algebra or Turaev algebra) is a $G$-graded associative unital algebra $A=\bigoplus_{g\in G}A_g$ together with a group homomorphism $\alpha:G\to \Aut(A)$ such that
\begin{enumerate}
\item The $G$ action is compatible with the $G$-grading via the adjoint action of $G$ on itself,
$\alpha_h:A_g\to A_{hgh^{-1}}$.

\item The restriction of $\alpha_h$ to $A_h$ is the identity.
\item $A$ is twisted commutative: For all $a\in A_g, b\in A_h$ we have $\alpha_h(a)b=ba$.
\item There is a $G$-invariant trace $\epsilon: A_1\to \Bbbk$ such that the induced pairing $A_g\otimes A_{g^{-1}}\stackrel{m}{\to}A_1\stackrel{\epsilon}{\to}  \Bbbk$ is non-degenerate.
\item For all $g,h\in G$ we have
\be
\sum\alpha_h(\xi_i)\xi^i=\sum\eta_i\alpha_g(\eta^i)\in A_{hgh^{-1}g^{-1}}
\ee
where $(\xi_i,\xi^i)$ and $(\eta_i,\eta^i)$ are pairs of dual bases of $A_g,A_{g^{-1}}$ and $A_h,A_{h^{-1}}$ respectively.
\end{enumerate}
We call $A_g$ the $g$-graded component and $A_1$ the neutral
component.
A morphism of $G$-Frobenius algebras is a morphism of unital algebras that 
respects the trace, the $G$-action and the grading. One verifies that all morphisms
of $G$-Frobenius algebras are isomorphisms.
\end{Def}

The following theorem \cite{mose, turaevhomotopy}  holds:

\begin{thm}\label{thm:gtft-gfrob}
The symmetric tensor categories of $G$-TFTs and $G$-Frobenius algebras 
are equivalent.
\end{thm}

Instead of reviewing the complete proof (see \cite{turaevhomotopy} or \cite{mose}
with slightly different conventions), 
we recall how to extract the data of a $G$-Frobenius algebra from a $G$-TFT
$\gtft$.

\begin{itemize}
\item 
For $g\in G$, the $g$-graded component is defined as
$$A_g:=\gtft(P_{g^{-1}}\to S^1,[0,1_G]) \,\,\, , $$
where $P_g$ is the principal $G$-bundle on $S^1$ introduced 
subsection \ref{sssect:covcirc}.

\item
For any pair of group elements $g,h\in G$ consider the 
standard three-point block 
$$(S_3(g,h,(gh)^{-1};1,1,1)\to S_3)$$ 
as a cobordism
\be
(P_{g^{-1}}\to S^1,[0,1_G])\sqcup (P_{h^{-1}}\to S^1,[0,1_G])\longrightarrow (P_{(gh)^{-1}}\to S^1,[0,1_G])
\ee
Its image under the functor $\gtft$ is a morphism 
$m_{g,h}:A_g\otimes A_h\to A_{gh}$ which yields an 
associative product 
$\sum_{g,h\in G}m_{g,h}$ on the $G$-graded vector space 
$A=\bigoplus_{g\in G}A_g$.

\item 
Unit $\eta$  and counit $\epsilon$ of $A$ are obtained from cobordisms with the topology of a
disc $D$. Since the disc is contractible, it only admits the trivial cover
$D\times G\to D$ which restricts unit and counit to be trivial on $A_g$ 
for $g\neq 1$.

The unit $\eta$ of $A$ is obtained as $\gtft(D\times G\to D)$
with the disc viewed as a cobordism from the empty set to 
$(P_1\to S^1,[0,1_G])$. Hence $\eta:\Bbbk\to A_1$. 
Similarly we define the counit as $\gtft(D\times G\to D)$, where this time 
the disc is seen as a cobordism from $(P_1\to S^1, [0,1_G])$ to the empty set.

\item
We finally obtain the action of $h\in G$ on $A_g$ from the
cover 
$$S_2(g,g^{-1};1,h)\to S_2$$ 
of the cylinder where the marked point over the outgoing boundary
has been shifted by $h\in G$. The discussion in subsection \ref{sssect:nsphere}
shows that
the boundaries are isomorphic to the bundles $(P_{g^{-1}}\to S^1,[0,1_G])$ and 
$(P_{hg^{-1}h^{-1}}\to S^1, [0,1_G])$ respectively, hence $\alpha_h$ has the 
correct domain and target.
\end{itemize}

We refer to \cite{mose} for a detailed proof that this endows
$A=\bigoplus_{g\in G} A_g$ with the structure of a $G$-Frobenius algebra.

\subsubsection{Permutation equivariant Frobenius algebras}
\label{sssect:permequivFrob}

We now describe a construction that can be seen as the
decategorified version of the main construction of
this article. Suppose we are given a finite ordered $G$-set $\X$ and a 
commutative Frobenius algebra $(R,\eta, m, \epsilon, \Delta)$, 
playing the 
role of a decategorification of a modular tensor category. 
A Frobenius algebra structure on an associative unital algebra
$(R,m,\eta)$ can be equivalently described by a linear form $\epsilon$ 
such that the induced 
bilinear pairing on $R$ is non-degenerate
or by a coalgebra structure
$(R,\Delta,\epsilon)$ such that $\Delta$ is a morphism of
$R$-bimodules. Here, we prefer the latter description.
We want to
to construct a $G$-Frobenius algebra $A$ with neutral component
$A_1=\bigotimes_{\X}R$, where the $G$-action on
$A_1$ is induced by the $G$-action on  $\X$. 

We start by constructing the underlying $G$-graded vector
space:

\begin{itemize}
\item
Composing the $2$-dimensional topological field theory
associated to $R$
\be
\tftR:\cobtwo\to \vectk.
\ee
with the tensor functor $\cover_\X$, we obtain by corollary 
\ref{cor:covertft} a $2$-dimensional $G$-equivariant TFT,
$$ \tftR^\X:= \tftR\circ\cover_\X\ \,\,\, . $$

\item
To describe the $G$-Frobenius algebra $A=\bigoplus_{g\in G}A_g$ that corresponds 
to this $G$-TFT, we first describe the vector spaces $A_g$ for $g\in G$.
By lemma \ref{lemma:EgPg}, we have 
$\cover_\X(P_{g^{-1}}\to S^1,[0,1_G])\cong E_{g^{-1}}$; since the
functor $\tftR$ is monoidal, we only need to know the number of connected 
components of $E_{g^{-1}}$ which by lemma \ref{lem:orbits} is
the number $b_g$ of orbits of the cyclic group $\langle g\rangle$
on \X. Thus, as a vector space,
\be
A_g\cong \tftR(E_{g^{-1}})\cong\tftR(\sqcup_{o\in O_g}S^1)\cong R^{\otimes b_g} \cong
\bigotimes_{o\in O_{g}}R_o
\,\, ,
\ee
with $R_o\cong R$ as a vector space. Hence an element of $A_g$ is a linear 
combination of elements of the form $r_{o_1}\otimes\dots\otimes r_{o_{b_{g}}}$ 
with $r_{o_i}\in R$. 
\end{itemize}

The product morphisms $m_{g_1,g_2}:A_{g_1}\otimes A_{g_2}\to A_{g_1g_2}$
are induced by the covers 
$$(S_3(g_1,g_2,(g_1g_2)^{-1};1,1,1)\to S_3)$$ 
of the three-punctured sphere. 
On these covers the cover functor $\cover_\X$ gives
the surfaces $E_{g_1;g_2}$.

By lemma \ref{lem:poporbits} each $\langle g_1,g_2\rangle$-orbit 
$o$ on $\X$ gives a connected component $E_{g_1;g_2}^o$ of 
$E_{g_1;g_2}$ and thus a contribution to the product 
morphism.
We describe these contributions separately and write the elements 
in $A_{g_1}$ as products of elements $r_{o_i'}\in R$ and the elements in 
$A_{g_2}$ as products of elements $s_{o_i''}\in R$.

\begin{itemize}
\item First multiply all elements $r_{o'}$ and $s_{o''}$ for all $\langle g_1\rangle$-orbits $o'$ and the $\langle g_2\rangle$-orbits $o''$ that are contained in the $\langle g_1,g_2\rangle$-orbit $o$.
No choices are involved, 
because the Frobenius algebra $R$ is commutative.

\item By lemma \ref{lem:genus}, the genus of $E_{g_1;g_2}^o$ equals
$p:=\frac{2-b_{g_1}^o-b_{g_2}^o-b_{g_1g_2}^o+|o|}{2}$. In a second step,
apply the endomorphism $(m\circ\Delta)^p$ of $R$ to the product of the previous step.

\item Let $o_1,\dots, o_k$ be those orbits of the cyclic group
$\langle g_1\cdot g_2\rangle$ that are contained in the 
$\langle g_1,g_2\rangle$-orbit $o$. To the element of $R$ obtained in the 
previous step, apply the $k$-fold coproduct of $R$, so we get an element in 
$R_{o_1}\otimes\dots\otimes R_{o_k}$. This element is well defined by 
coassociativity of $R$.

\item Map the factors $R_{o_i}$ of the previous step to the corresponding 
factors of $A_{g_1g_2}$. No choices are involved,  
since the coproduct on the Frobenius algebra $R$ is
cocommutative.

\end{itemize}

This provides the prescription for the product on the $G$-Frobenius
algebra. As explained in subsection \ref{sec:tfttofrob},
the unit of $A$ is obtained as the evaluation of the functor 
$\tftR^\X$ on the 
disc with the trivial $G$-cover. The cover functor maps $(D\times G\to D)$ to the disjoint union
$\bigsqcup_{x\in\X}D$ of $|\X|$-many discs, hence the unit 
of the $G$-Frobenius algebra $A$ is just the tensor product
of the units 
$\bigotimes_{x\in\X}\eta:\Bbbk\to A_1\cong\bigotimes_{x\in\X}R$
of $R$. Similarly 
we find that the counit of $A$ is given  by tensor product
of the counits of $R$.

From lemma \ref{lem:actioncomponents} we deduce that the $G$-action 
$\alpha_h:A_g\to A_{hgh^{-1}}$ is given by the permutation of factors: 
The connected components of the cover $\cover_\X(S_2(g,g^{-1};1,h))$ of the
cylinder are again cylinders. For any $\langle g\rangle$-orbit $o$ of $\X$ 
the factor $R_o$ of $A_g$ is thus mapped to the factor $R_{ho}$ of 
$A_{hgh^{-1}}$.


\section{$G$-modular functors and $G$-equivariant ribbon categories}
\label{sec:GMF}

Let $\C$ be a modular category over a field $\Bbbk$; we assume
$\C$ to be strict. Let $\X$ be a finite ordered $G$-set.
The goal of this section is to construct for any pair $(\C,\X)$ 
a $G$-equivariant modular functor. Our construction is based
on the decategorified version of the construction in section \ref{sec:2dalg}.

\subsection{Definitions and notation}

We start by recalling some definitions \cite{kpGMF}:

\begin{Def}
A \emph{$G$-equivariant category} is an abelian category \CG with the following structure:
\begin{itemize}
\item A decomposition $\CG\cong\bigoplus_{g\in G}\CG_g$ into full abelian subcategories. 
\item A $G$-action covering the adjoint action of $G$ on itself. \\
In more detail, we have for any group element $g\in G$ a 
functor $R_g:\CG\to \CG$ and for any pair $g,h\in G$ of
group elements isomorphisms $\alpha_{g,h}:R_g\circ R_h\Rightarrow R_{gh}$ such that $R_1=\mbox{Id}_{\CG}$, $R_g(\CG_h)\subset\CG_{hgh^{-1}}$.
The isomorphisms $\alpha_{g,h}$ are required to satisfy an 
associativity condition.
\end{itemize}
\end{Def}

As a shorthand, we introduce the notation ${}^gV\equiv R_g(V)$ for $g\in G$ and $V$ in $\CG$.

\begin{Def}
Let \CG be a $G$-equivariant category. 
We assume from now on that $\CG$ is enriched over the category of 
finite-dimensional $\Bbbk$-vector spaces. We denote by $\boxtimes$ the 
Deligne tensor product of $\Bbbk$-linear categories.
An object $\R\in\CG\boxtimes\CG$ is called a {\em gluing object} if
\begin{itemize}
\item $\R$ is of vanishing total degree, 
$\R\in\bigoplus_h\CG_h\boxtimes\CG_{h^{-1}}$.
\item $\R$ is symmetric, i.e. $\R\cong\R^{op}$. Here $\R^{op}$ is obtained by the permutation action on the
       two factors.
\item $\R$ is $G$-invariant: For every group element $g\in G$ there is an isomorphism $(R_g\boxtimes R_g)(\R)\cong \R$.
\item These isomorphisms are compatible with each other.
\end{itemize}

We write $\R_h$ for the component of $\R$ in 
$\CG_h\boxtimes\CG_{h^{-1}}$; sometimes we
use the Sweedler-like notation $\R=\R^{(1)}\boxtimes\R^{(2)}$.
\end{Def}

We are now in a position to give the definition \cite{kpGMF} of a $G$-modular functor. To make the
notation less cumbersome, we sometimes use the abbreviation
$P\equiv(P\to M,\{p_a\}_{a\in A(M)})$ for $G$-covers.

\begin{Def}
Let \CG a $G$-equivariant category enriched over the category of finite-dimensional
$\Bbbk$-vector spaces and $\R$ be a gluing object for $\CG$. A \emph{\CG-extended $G$-equivariant modular functor} consists of the following data:
\begin{enumerate}
\item Functors for $G$-covers: \\
For every $G$-cover $(P\to M,\{p_a\}_{a\in A(M)})$ of an extended surface a functor
\be
\tau^G(P\to M,\{p_a\}_{a\in A(M)}):\bigboxtimes_{a\in A(M)}\CG_{m_a^{-1}(M)}\to\vect \,\, ,
\ee
where $m_a$ is the monodromy of $P$ around the $a$-th boundary component of $M$. 
We will often write $\tau^G(P;\{V_a\})$ for the value of the
functor on a family $\{V_a\}$ of suitable objects.

\item Functorial isomorphisms for morphisms of $G$-covers: \\
For every isomorphism $f:(P\to M,\{p_a\}_{a\in A(M)})\to (P'\to M',\{p'_a\}_{a\in A(M')})$ 
of extended surfaces a functorial isomorphism 
$$f_\ast: \tau^G(P\to M,\{p_a\}_{a\in A(M)})\to\tau^G(P'\to 
M',\{p'_a\}_{a\in A(M')})$$ 
that depends only on the isotopy class of $f$.
\item Isomorphisms $\tau^G(\emptyset)\cong\Bbbk$ and $\tau^G (P\sqcup P')\cong \tau^G(P)\otimes_\Bbbk\tau^G(P')$.
\item Functorial gluing isomorphisms: \\
Let $(P\to M,\{p_a\}_{a\in A(M)})$ be a $G$-cover of an extended surface and let $\alpha,\beta\in A(M),\alpha\neq\beta$ such that
the monodromies are inverse,
$m_\alpha=m_\beta^{-1}$. Then gluing of $P$ along the boundary components over $\alpha$ and $\beta$ is well defined. We have functorial \emph{gluing isomorphisms}
\be
G_{\alpha,\beta}:\tau^G(P;\{V_a\},\R_{\alpha,\beta})\stackrel{\sim}{\to}\tau^G(\sqcup_{\alpha,\beta}P;\{V_a\})
\ee
where $\R_{\alpha,\beta}$ indicates that the summand $\R_{m_\alpha}$ of $\R$ is assigned to the boundary components $\alpha$ and $\beta$ respectively. This is well defined by symmetry of $\R$.
Here $\sqcup_{\alpha,\beta}M$ is the surface with the 
boundary components $\alpha$ and $\beta$ glued.
$\sqcup_{\alpha,\beta}P$ denotes the $G$-cover of 
$\sqcup_{\alpha, \beta}M$ that is obtained by gluing the 
corresponding boundary components over $\alpha$ and $\beta$.
\item Equivariance under the $G$-action: \\
For any $G$-cover $(P\to M,\{p_a\}_{a\in A(M)})$ and any tuple of group elements
${\bf g}=(g_a)_{a\in A(M)}\in G^{A(M)}$ we have functorial isomorphisms
\be
T_{\bf g}:\tau^G(P\to M,\{p_a\}_{a\in A(M)};\{V_a\})\stackrel{\sim}{\to}\tau^G(P\to M,\{g_ap_a\}_{a\in A(M)};\{{}^{g_a}V_a\})\,\,.
\label{17}
\ee
\end{enumerate}
These data are subject to the following conditions:
\begin{itemize}
\item $(f\circ g)_\ast=f_\ast\circ g_\ast$ and $\mbox{id}_\ast=\mbox{id}$.
\item All morphisms are functorial in $(P\to M,\{p_a\}_{a\in A(M)})$ and compatible with each other.
\item When identifying $\R\cong\R^{op}$ we have $G_{\alpha,\beta}=G_{\beta,\alpha}$.
\item Normalization: $\tau^G(S^2\times G\to S^2)\cong\Bbbk$
\end{itemize}
\end{Def}

\begin{rem}
Specializing to the trivial group, $G=\{1\}$ and suppressing the morphisms $T_{\bf g}$ from \erf{17} implementing
equivariance, we recover the usual definition of a modular functor, see e.g.\ \cite{BK}. 
\end{rem}

For our purposes, the notion of a $G$-equivariant monoidal structure
\cite{turaevhomotopy3d,kirillov,kpGMF} 
will be important. 

\begin{Def}\label{def:Gtensor}\mbox{}\\
\begin{enumerate}
\item
A \emph{$G$-equivariant monoidal category} is a semisimple $G$-equivariant category \CG with a monoidal 
structure that is compatible with the grading, i.e.\ $X\otimes Y\in\CG_{gh}$ for $X\in \CG_g, Y\in \CG_h$
and for which the functors $R_g$ implementing equivariance 
are endowed with the structure of tensor functors.

\item
A $G$-equivariant monoidal category is called \emph{braided}, if for any pair of objects
$X\in \CG_g, Y\in \CG_h$ there are isomorphisms 
$$C_{X,Y}:X\otimes Y\to {}^gY\otimes X$$ 
that satisfy two $G$-equivariant hexagon axioms.

\item
An object $V$ in a $G$-equivariant monoidal category $\CG$ has a 
\emph{weak dual} if there is an object $V^\ast\in\CG$ representing the
functor $\hom_{\CG}(\unit,V\otimes ?)$. This amounts to the existence of
functorial isomorphisms $\hom_{\CG}(\unit,V\otimes T)\cong\hom_{\CG}(V^\ast,T)$ 
for all $T\in\CG$.

\item
A $G$-equivariant monoidal category is called \emph{weakly rigid} if every object has a weak dual. 
It is called \emph{rigid}, if there are compatible duality morphisms.

\item
A $G$-equivariant monoidal category is called \emph{weakly ribbon} if it is weakly rigid, braided and for every object $V\in\CG_g$ there is a functorial isomorphism $\Theta_V:V\to {}^gV$, satisfying certain 
coherence conditions spelled out in \cite[Section 2]{kirillov}. 
A weakly ribbon category is called a \emph{ribbon category}, if it is rigid rather than only weakly rigid.
\end{enumerate}
\end{Def}

The discussion in the following subsection \ref{ssect:Gmodcov} strongly uses the following theorem from \cite{kpGMF}:

\begin{thm}
A genus $0$ \CG-extended $G$-modular functor is equivalent to the structure of a $G$-equivariant weakly ribbon category on \CG. 
\end{thm}

In \cite{kpGMF} no explicit prescription is given how to obtain from a 
\CG-extended $G$-modular functor $\tau^G$ the 
structure morphisms endowing the equivariance functors $R_g:\CG\to\CG$ with the structure 
of  tensor functors. We will need the explicit form of these structure morphisms
and therefore explain them in some detail.

In the definition in subsection \ref{sssect:nsphere} of standard blocks 
on  the $n$-punctured sphere $S_n$ as a quotient of $S_n\!\setminus\!\mbox{cuts}\times G$
a point of the total space $S_n(g_1,\dots,g_n;h_1,\dots, h_n)$ is an equivalence class $[z,x]$ 
with $z\in S_n\!\!\setminus\!\!\mbox{cuts}$ and $x\in G$. For every group element $k\in G$ 
the map $[z,x]\mapsto [z,xk]$ induces an isomorphism of $G$-covers
\be
\tilde k: \,\,\, S_n(g_1,\dots,g_n;h_1,\dots, h_n)\to S_n(k^{-1}g_1k,\dots, k^{-1}g_nk;h_1k,\dots,h_nk)
\label{18}\ee

The corresponding natural transformations enter in the
construction of the tensoriality constraints.
To construct these morphisms,
let $h\in G$ and $A\in\CG_{g_1}$ and $B\in\CG_{g_2}$ be objects of \CG. 
The main step is to construct a natural isomorphism between the functors
$\C\to\vect_{\Bbbk}$ given by
\be
X\mapsto \tau^G(S_2(hg_2^{-1}g_1^{-1}h^{-1},hg_1g_2h^{-1};1,1);X,{}^h(A\otimes B))
\label{eq:funct1}
\ee
and
\be X\mapsto
 \tau^G(S_2(hg_2^{-1}g_1^{-1}h^{-1},hg_1g_2h^{-1};1,1);X,{}^hA\otimes{}^hB) \,\, .
\label{eq:funct2}
\ee
Since our categories are, by assumption, semi-simple with
finitely many isomorphism classes of simple objects, every
functor is representable. In \cite{kpGMF} the objects 
${}^h(A\otimes B)$ and ${}^hA\otimes{}^hB$ have been introduced as the
objects representing the functors (\ref{eq:funct1}) and (\ref{eq:funct2})
respectively.
Thus, by the Yoneda lemma the 
image of the identity for $X={}^h(A\otimes B)$ under the
natural isomorphism
gives an isomorphism
$$\phi_{A,B}^h:\,\,\, {}^h(A\otimes B)\to {}^hA\otimes{}^hB \,\, .$$

The natural transformation we use is given by
\be\label{eq:tensorf}
\begin{array}{ll}
&\tau^G(S_2(hg_2^{-1}g_1^{-1}h^{-1},hg_1g_2h^{-1};1,1);X,{}^h(A\otimes B))\\
\stackrel{T_{(1,h^{-1})}}{\to}&\tau^G(S_2(hg_2^{-1}g_1^{-1}h^{-1},hg_1g_2h^{-1};1,h^{-1});X,A\otimes B)\\
\stackrel{(\tilde h)_\ast}{\to}&\tau^G(S_2(g_2^{-1}g_1^{-1},g_1g_2;h,1);X,A\otimes B)\\
\stackrel{G^{-1}}{\to}&\tau^G(S_2(g_2^{-1}g_1^{-1},g_1g_2;h,1);X,\R^{(1)})\otik\tau^G(S_2(g_2^{-1}g_1^{-1},g_1g_2;1,1);\R^{(2)},A\otimes B)\\
\defn&\tau^G(S_2(g_2^{-1}g_1^{-1},g_1g_2;h,1);X,\R^{(1)})\otik\tau^G(S_3(g_2^{-1}g_1^{-1},g_1,g_2;1,1,1);\R^{(2)},A, B)\\
\stackrel{G}{\to}&\tau^G(S_3(g_2^{-1}g_1^{-1},g_1,g_2;h,1,1);X,A,B)\\
\stackrel{(\tilde h^{-1})_\ast}{\to}&\tau^G(S_3(hg_2^{-1}g_1^{-1}h^{-1},hg_1h^{-1},hg_2h^{-1};1,h^{-1},h^{-1});X,A,B)\\
\stackrel{T_{(1,h,h)}}{\to}&\tau^G(S_3(hg_2^{-1}g_1^{-1}h^{-1},hg_1h^{-1},hg_2h^{-1};1,1,1);X,{}^hA,{}^hB)\\
\defn&\tau^G(S_2(hg_2^{-1}g_1^{-1}h^{-1},hg_1g_2h^{-1};1,1);X,{}^hA\otimes{}^hB)
\end{array}
\ee
for any $X$ in $\CG_{hg_2^{-1}g_1^{-1}h^{-1}}$. The idea
in the definition is to use the equivariance isomorphisms
defined in equation (\ref{17}) to shift the $G$-action
from objects in the category to geometric quantities.
Then a factorization is applied to be able to use the
definition of the tensor product and finally, the $G$-action
is shifted again to objects.

The following observation follows from the compatibility of 
all occurring morphisms and the definition of the 
associativity constraints in $\C$.

\begin{lemma}\label{prop:tensorf} \mbox{}\\
The morphisms $\phi_{A,B}^h$ endow the functor $R_h$ with the
structure of a monoidal functor.
\end{lemma}

Next, we adapt the discussion of \cite[prop. 5.3.13]{BK} of conditions ensuring that a weakly
ribbon category is a ribbon category to the $G$-equivariant case.
Suppose $\CG$ is a weakly rigid $G$-equivariant category in which biduals of objects can be
functorially identified with the objects, $(V^\ast)^\ast\cong V$ for all $V$ in \CG.
This condition is fulfilled for all categories with tensor
structure obtained from a modular functor.
The image of the identity on $V^\ast$ under the functorial isomorphisms
of definition \ref{def:Gtensor}.3 provides us
with a morphism $i_V:\unit\to V\otimes V^\ast$ for any object $V\in \CG$.
Consider for a simple object $V$ of \CG the morphism
\be
\alpha_{V^\ast,V,V^\ast}^{-1}\circ (\id_{V^\ast}\otimes i_{V}):
V^\ast\to(V^\ast\otimes V)\otimes V^\ast
\ee
constructed from $i_V$ and the associativity constraint $\alpha$ of
$\CG$. Since $\CG$ is semisimple,
we can decompose $V^\ast\otimes V\cong\bigoplus_jV_j$ into a direct sum of
simple objects $V_j$. Multiplicities can occur, but since  
$\dim_\Bbbk\hom_{\CG}(\unit,V^\ast\otimes V)=1$
we can decompose the morphism 
\be
\alpha_{V^\ast,V,V^\ast}^{-1}\circ (\id_{V^\ast}\otimes i_V)
=a_V\otimes\id_{V^\ast}+\sum\psi_{j}
\ee
into a morphism $a_V:\unit\to V^\ast\otimes V$ in the one-dimensional vector space
$\hom(1,V^\ast\otimes V)$ and certain morphisms $\psi_{j}$. By the same arguments as in
\cite{BK}, we have 

\begin{prop}\label{prop:conditionrigid}
The category \CG is rigid, if and only if $a_V\neq 0$ for 
all simple objects $V$.
\end{prop}

\subsection{$G$-equivariant modular functors from the cover functor}
\label{ssect:Gmodcov}

We fix a finite ordered $G$-set $\X$ and wish to use the 2-dimensional cover functor
\be
\cover_X:\Gext\to\ext
\ee
on extended surfaces to associate to every modular tensor category $\C$ a $G$-equivariant
modular functor.

We fix a modular tensor category $\C$ and choose representatives $(U_i)_{i\in\I}$ 
for the isomorphism classes of simple objects of $\C$ and denote by 
$\theta_i\in\Bbbk$ the eigenvalue of the twist on $U_i$ and by $d_i$ 
the dimension of $U_i$. As usual, we introduce the scalars 
$p^\pm:=\sum_{i\in\I}\theta_i^\pm d_i^2$ and assume that $p^+=p^-$.
Finally, we introduce 
\be \catdim=\sqrt{p^+p^-}=  \sqrt{\sum_i (d_i)^2} \,\, . 
\label{catdim}\ee
With this assumption, the modular category $\C$ defines
(see e.g. \cite{turaevquantum, BK}) a $\C$-extended modular functor $\tau$ 
which acts on objects as
\be
\tau(S_n;V_1,\dots,V_n):=\homc(\unit,V_1\otimes\dots\otimes V_n) \,\, .
\ee

To define a $G$-modular functor, we need a $G$-equivariant category as input.
Given $(\C,\X)$, we define a $G$-equivariant category by transferring the results
of subsection \ref{sec:GTFT}
 to categorical structures.

The first part of data required to define a $G$-modular functor is a 
$G$-equivariant category.
\begin{Def}\label{def:CG}
Let $G$ be a finite group, $\X$ a finite $G$-set and $\C$ a modular tensor category. Then the following data define a
$G$-equivariant category $\CX$:
\begin{enumerate}
\item
For $g\in G$ set 
$$\CX_g:=\C^{\boxtimes b_g} \equiv\C_{o_1}\boxtimes\dots\boxtimes\C_{o_{b_g}}\,\, , $$ 
where $b_g$ is the number of $\langle g\rangle$-orbits $o_i$ of $\X$.
Moreover, $\C_{o_i}\cong\C$ as abelian categories.

\item
For $g,h\in G$ define a functor $R_h:\CX_g\to \CX_{hgh^{-1}}$ by permutation of factors mapping
the factor $\C_o$ of $\CX_g$ to the factor $\C_{ho}$ of $\CX_{hgh^{-1}}$.

The equivariance functors  obey the strict 
identities $R_h\circ R_k=R_{hk}$ for all $h,k\in G$ which allows us to choose trivial composition constraints,
$\alpha_{h,k}=\mbox{id}$ for all $h,k\in G$. 
\end{enumerate}
\end{Def}

The $G$-equivariant category $\CX$ is semisimple and has only finitely many 
isomorphism classes of simple objects. Representatives of
the isomorphism classes of simple objects in $\CX_g$ 
are $U_{i_1}\times\dots\times U_{i_{b_g}}$, where the $U_i$ are representatives 
of the isomorphism classes of simple objects of $\C$. 

This ansatz agrees with the description of simple objects in $\ZZ$-permutation
orbifolds found in \cite{bohs} and the analysis of more general permutation
orbifolds in \cite{Ba2}. We briefly explain  the simplest case of
$\ZZ$-permutation orbifolds: in the orbifold theory, simple
objects in the twisted sector come in pairs which are in correspondence
to simple objects in $\C$. The operation inverse to orbifolding is, in the
situation at hand, a simple current extension with a simple current of
order two. In the twisted sector, simple current ``fixed points'' do not
occur; hence the pair of simple objects gives rise to a single simple object
in the equivariant theory. Indeed, in this specific situation, the generating
element $g$ of $\ZZ$ has a single orbit and thus our ansatz for  the twisted
sector $\CX_g$ as an abelian category is $\C$. The analysis of the untwisted
sector is similar, and the whole analysis can be extended to general
permutation orbifolds.

We next define the gluing object 
\be
\R\in\bigoplus_{g\in G}\CX_g\boxtimes\CX_{g^{-1}} \,\, ;
\ee
its component $\R_g$ in $\CX_g\boxtimes\CX_{g^{-1}}$ is
\be\label{def:R}
\R_g:=\bigoplus_{i_1,i_2\ldots i_{b_g}}(U_{i_{o_1}}\times\dots\times U_{i_{o_{b_g}}})\times (U_{i_{o_1}}^\vee\times\dots\times U_{i_{o_{b_g}}}^\vee)
\ee
where the direct sum is taken over all isomorphism classes 
of simple objects of $\C$. Thus the direct sum in \erf{def:R} 
is taken over representatives of all simple objects of $\C^{\boxtimes b_g}$. 
With this definition, $\R$ is clearly symmetric and $G$-invariant.

Now we are able to define a $\CX$-extended $G$-equivariant modular functor $\tau^\X$. Let $(P\to M,\{p_a\}_{a\in A(M)})$ be a $G$-cover of an extended surface $M$. Consider the surface 
$$\cover_\X(P\to M,\{p_a\}_{a\in A(M)})=\X\times_GP \,\,. $$ 
It can be viewed as the total space of an $|\X|$-fold cover of $M$ and
has, by the discussion in subsection \ref{sec:2dgeom} the
structure of an extended surface.
 Let $a\in A(M)$ be a boundary component of $M$. By lemmas \ref{lemma:EgPg} and \ref{lem:orbits} the restriction of $\X\times_GP$ to $a$ has a connected component for every $\langle m_a\rangle$-orbit of $\X$, where $m_a$ is the monodromy of $P$ around $a$.

\begin{Def}
Let $(P\to M,\{p_a\}_{a\in A(M)})$ be a $G$-cover of an extended surface $M$. Define a functor
\be
\tau^\X(P\to M,\{p_a\}_{a\in A(M)}):\bigboxtimes_{a\in A(M)}\CX_{m_a^{-1}}\to \vectk
\ee
by
\be
\tau^\X(P\to M,\{p_a\}_{a\in A(M)}):=\tau(\cover_\X(P\to M,\{p_a\}_{a\in A(M)}))=\tau(\X\times_GP).
\ee
\end{Def}

This is indeed well-defined: from
$\CX_{m_a^{-1}}=\bigboxtimes_{O_{m_a}}\C$ we conclude that
$$\bigboxtimes_{a\in A(M)}\CX_{m_a^{-1}}
=\bigboxtimes_{a\in A(\X\times_GP)}\C \,\, .$$
Note that the definition of $\tau^\X$ also depends on the 
choice of marked points in $\X\times_GP$.

We will now describe all the additional data that is needed 
to turn $\tau^\X$ into a $\CX$-extended $G$-equivariant 
modular functor and prove that all axioms are satisfied.

\begin{itemize}
\item
Let $f:(P\to M,\{p_a\}_{a\in A(M)})\congto(P'\to M',\{p'_a\}_{a\in A(M')})$ be an isomorphism in $\Gext$. This gives a morphism 
$\bar f:=\cover_\X(f):\X\times_GP\congto\X\times_GP'$ in \ext. By definition, the $\C$-extended modular functor $\tau$ gives an isomorphism $\bar f_\ast:\tau(\X\times_GP)\Rightarrow\tau(\X\times_GP')$ of functors, hence an isomorphism $\bar f_\ast:\tau^\X(P\to M,\{p_a\}_{a\in A(M)})\Rightarrow\tau^\X(P'\to M',\{p'_a\}_{a\in A(M')})$. This isomorphism only depends on the isotopy class of $f$ and it obeys $\overline{(f\circ g)}_\ast=\bar f_\ast\circ\bar g_\ast$.

\item
The $\C$-extended modular functor has enough structure to
provide an isomorphism of functors
\be
\tau^\X(\emptyset)=\tau(\X\times_G\emptyset)=\tau(\emptyset)\cong\Bbbk
\ee
and for $G$-covers $(P\to M,\{p_a\}_{a\in A(M)})$ and $(P'\to M',\{p'_a\}_{a\in A(M')})$
\be
\bearl
\tau^\X(P\sqcup P')=\tau(\X\times_G(P\sqcup P'))\cong \tau((\X\times_GP)\sqcup (\X\times_GP'))\\
\Rightarrow \tau(\X\times_GP)\otimes_\Bbbk\tau(\X\times_GP'))
=\tau^\X(P)\otimes_\Bbbk\tau^\X(P').
\eear
\ee

\item
Next we describe the gluing isomorphisms. Let $(P\to M,\{p_a\}_{a\in A(M)})$ be a $G$-cover of $M$ and let $\alpha,\beta\in A(M)$ be two different boundary components of $M$ with inverse monodromies, $m_\alpha=m_\beta^{-1}$. 
This condition ensures that we can glue $P$ along $\alpha$ 
and $\beta$. For all $a\in A(M)\setminus\{\alpha,\beta\}$, let $W_a$ be an object of $\CX_{m_a^{-1}}$. Then
\be
\tau^\X(P,\{W_a\},\R_{\alpha,\beta})\cong\bigoplus\tau(\X\times_GP, \{W_a\}, U_{i_{o_1}}\times\dots\times U_{i_{o_k}}, 
U_{i_{o_1}}^\vee\times\dots\times U_{i_{o_k}}^\vee).
\ee
Here the $o_i$ are the $\langle m_\alpha^{-1}\rangle$-orbits of $\X$. 
Since the monodromies are inverses, $m_\beta=m_\alpha^{-1}$, 
these orbits are equal to the $\langle m_\beta\rangle$-orbits so that the assignment of the simple objects in the first component of $\R_{m_\alpha^{-1}}$ to the boundary components of $\X\times_GP$ over $\alpha$ and the simple objects in the second component of $\R_{m_\alpha^{-1}}$ to the boundary components over $\beta$ is compatible. The direct sum is taken over all simple objects as in  definition \ref{def:R}, where $\R$ has been introduced.

Since $\tau$ is a modular functor, we get gluing isomorphisms for all boundary components of the cover $\X\times_GP$ over $\alpha$ and $\beta$ by gluing $\X\times_GP$ along each boundary component over $\alpha$ and $\beta$ separately. The gluing isomorphisms of $\tau$ satisfy an associativity 
condition; thus the procedure does not depend on the order 
in which the boundary components are glued. Hence we get gluing isomorphisms
\be
G_{\alpha,\beta}:\tau^\X(P,\{W_a\},\R_{\alpha,\beta})\congto\tau^\X(\sqcup_{\alpha,\beta}P, \{W_a\})
\ee
that are functorial in the objects $W_a$.
The procedure in the definition of $G_{\alpha, \beta}$ relies on the
fact that the cover functor $\cover_\X$ respects gluing of covers.

\item
Next we implement $G$-equivariance.
Let $(P\to M,\{p_a\}_{a\in A(M)})$ be a $G$-cover of an extended surface $M$ with marked points $\{p_a\}$ and let ${\bf{g}}=(g_a)_{a\in A(M)}\in G^{A(M)}$ be a tuple of elements of $G$ for each boundary component of $M$. For all $a\in A(M)$ let $W_a$ be an object in $\CG_{m_a^{-1}}$. 
We need to give functorial isomorphisms
\be\label{eq:Tg}
T_{\bf{g}}:\tau^\X(P\to M,\{p_a\}, \{W_a\})\congto\tau^\X(P\to M, \{g_ap_a\}, \{{}^{g_a}W_a\}).
\ee
implementing the action of ${\bf{g}}$ on the boundary
components.

The boundary component of $(P\to M,\{p_a\})$ 
corresponding to the $a$-th boundary component of $M$
is isomorphic to the cover $(P_{m_a}\to S^1,[0,1_G])$. Now we analyse the situation on the right hand side of \erf{eq:Tg}. The boundary of $(P\to M,\{g_ap_a\})$ is isomorphic to $(P_{m_a}\to S^1,[0,g_a])\cong (P_{g_am_ag_a^{-1}}\to S^1, [0,1_G])$. This
induces precisely the map $E_{m_a}\to E_{g_am_ag_a^{-1}}$
in lemma \ref{lem:orbitmap}. After identifying boundary
components and orbits,
a $\langle m_a\rangle$-orbit $o$ is mapped under this map
to the $\langle g_am_ag_a^{-1}\rangle$-orbit $g_ao$. 
On the other hand, the action of $R_{g_a}$ on 
$\CX_{m_a^{-1}}$ permutes the objects in $\CX_{m_a^{-1}}$ 
in exactly the same way. Hence we can choose the 
equivariance isomorphisms $T_{\bf g}$ to be  identity morphisms.

\item
It follows from functoriality of the cover functor $\cover_\X$ and of $\tau$ that all isomorphisms constructed above are functorial in $(P\to M, \{p_a\})$. 
Similarly, one concludes that all isomorphisms are 
compatible.

\item The $G$-equivariant modular functor $\tau^\X$ is normalized:
\be
\tau^\X(S^2\times G\to S^2)=\tau(\X\times_G(S^2\times G))\cong \tau(\X\times S^2)\cong\Bbbk\otimes\dots\otimes\Bbbk\cong\Bbbk
\ee
\end{itemize}

We summarize these findings in the following theorem, which is one of the 
main results of this paper:

\begin{thm}\label{thm:GMF}
Let $G$ be a finite group, $\X$ a finite ordered $G$-set and \C a $\Bbbk$-linear modular category. 
Denote by $\CX$ the $G$-equivariant category
introduced in definition \ref{def:CG}. Then the functor $\tau^\X$
defined by 
\be
\tau^\X(P\to M, \{p_a\}_{a\in A(M)}):=\tau(\cover_\X(P\to M, \{p_a\}_{a\in A(M)}))
\ee
is a $\CX$-extended $G$-equivariant modular functor.
\end{thm}

This implies in particular:
\begin{cor}
\begin{enumerate}
\item The $G$-equivariant modular functor $\tau^\X$ induces
a $G$-equivariant monoidal structure on 
the $G$-equivariant category $\CX$.

\item The equivariant modular functor $\tau^\X$ induces
on the $G$-equivariant monoidal category $\CX$ the
structure of a weakly fusion category.
Since the $G$-modular functor in theorem \ref{thm:GMF} is 
defined for arbitrary genus, we expect this structure to 
be $G$-modular, so that orbifold theories exist.
Unfortunately, $G$-equivariant modular functors for higher 
genus are not yet well understood.

\item By restriction, the $G$-equivariant functor
$\tau^\X$ endows for any group element
$g\in G$ the abelian category   $\C^{\boxtimes b_g}$
with the structure of a module category 
over the tensor category $\C^{\boxtimes\X}$.
\end{enumerate}
\end{cor}

In the next section, we make this structure explicit
in the case of $G=\ZZ$ acting non-trivially on a two-element 
set $\X$. In this case, we can show that the equivariant category is
rigid rather than only weakly rigid. It is also known in this
case \cite{bfrs} that the module category structure on $\C$ over
$\C\boxtimes\C$ describes the permutation modular invariant.

For our construction, it is indispensable to require $\C$
to be modular. A genus $0$ $\CX$-extended $G$-equivariant 
modular functor has to be defined on $G$-covers of extended
surfaces of genus $0$. According to lemma \ref{lem:genus}, the 
total space of the associated $|\X|$-fold cover,
however,  can have a non-zero genus.
Hence the $\C$-extended
modular functor $\tau$ has to be defined for any genus,
and thus we can define, as in theorem \ref{thm:GMF}, the 
equivariant modular functor $\tau^\X$ for any genus as well.

\section{$\ZZ$-permutation equivariant fusion categories}\label{sec:z2}

\subsection{Notation and conventions}\label{ssec:notconv}
From now on, we restrict ourselves on the case where the cyclic group $G=\ZZ$ acts on the ordered two-element set $\X=\{1,2\}$ by permutation of elements. We denote the generator of \ZZ by $g$ and hence write $\ZZ=\{1,g\}$ with $g^2=1$. 

We note that in this special case,
the \ZZ-cover $P\to M$ and the cover $\X\times_{\ZZ}P\to M$ are 
isomorphic as covers over the manifold $M$. 
However, whereas $P$ has only one marked point for every 
connected component of $\partial M$, the cover $\X\times_{\ZZ}P$ 
has one marked point in every connected component of the
boundary of the cover. We assume that one distinguished
point has been chosen over each connected component of $\partial P$
over $\partial M$.

Our construction will involve choices: we define
the value of functors like the duality functor or the
tensor product functor by objects representing
functors constructed from the equivariant modular functor
$\tau^\X$. This is possible, since
by \cite[lemma 5.3.1]{BK} any additive functor 
$F:\mathcal{D}\to\vect_{\Bbbk}$
from a semisimple abelian category $\mathcal D$ with finitely many objects is representable.
By definition of the dual object this shows that any object $V$ is determined by the functor 
$\tau^\X(S_2,?,V)$.

The representing objects from the Yoneda lemma
will only
be unique up to canonical isomorphism.
Different choices of representing objects lead to different structures of ribbon categories on $\CZ$, which are equivalent via 
a tensor functor which is the identity functor 
on \CZ and whose structure morphisms 
are provided the Yoneda lemma. We will indicate such choices
in our discussion.

For the value of the equivariant modular functor 
$\tau^{\X}$ on the standard blocks of subsection \ref{sec:GTFT},
we introduce the shorthand notation
\be
\langle V_1,\dots,V_n\rangle_\X:=\tau^{\X}(S_n(g_1,\dots,g_n;1,\dots,1); V_1,\dots,V_n).
\ee
with objects $V_i\in \CZ_{g_i}$. We introduce a similar
shorthand for the $\C$-extended modular functor $\tau$
as well,
\be
\langle W_1,\dots,W_n\rangle:=\tau (S_n;W_1,\dots,W_n) \,\, ,
\ee
where $W_1,\dots, W_n$ are objects in $\C$.

As an abelian category, the $\ZZ$-equivariant
category \CZ is of the form $(\C\boxtimes\C)\oplus(\C)$, where 
$\C\boxtimes\C$ is the \emph{neutral} and \C the 
\emph{twisted} component or sector of \CZ. We denote objects 
in the neutral component by $A_1\times A_2, B_1\times B_2,\dots$ and objects in the twisted component by $M,N,\dots$.

The dual of an object $V$ in \C will be denoted by $V^\vee$. We agree to drop 
the tensor product symbol for objects of \C when using the monoidal structure 
of \C, so we write $AB\equiv A\otimes_\C B$. The braiding of two objects in 
\C will be denoted by $c_{A,B}$ and the equivariant braiding in \CZ by 
$C_{A,B}$ in capital letters. Similarly, $\theta_U$ denotes the twist,
$b_U$ the coevaluation and $d_U$ the evaluation in \C, while for the 
corresponding morphisms $\Theta_A$, $B_A$ and $D_A$ of $\CZ$ capital
letters are used.

When dealing with modular functors we use a graphical notation. First note that $\ZZ$ covers over $S_n$ can have the following two types of boundary components:

\be
\raisebox{-50pt}{
\bp(200,50)
\pic{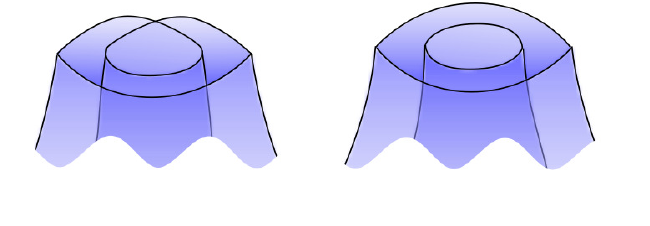}
\ep
}
\ee

The first is a boundary component of non-trivial monodromy $g$, the second with trivial monodromy $1\in\ZZ$.

A decorated surface $\Sigma$ also represents the corresponding vector
space $\tau(\Sigma; V_1,\dots,V_n)$: we then write objects of the
appropriate categories to the boundary components:
\be
\raisebox{-50pt}{
\bp(200,150)
\pic{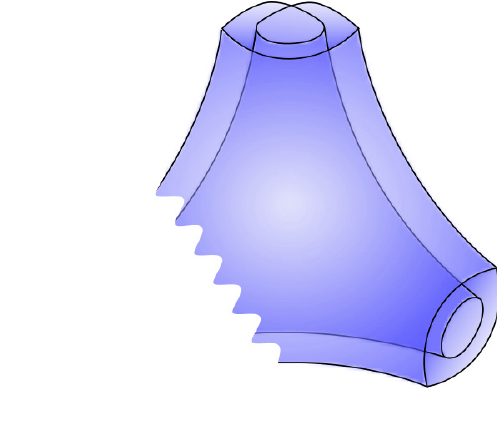}
\put(-65,135) {$M$}
\put(0,30){$A_1\times A_2$}
\ep
}
\ee
When we write an object $A_1\times A_2$ next to a boundary embedded
in $\mathbb{R}^3$, our convention is such that the first object is assigned
to the outer circle and the second object to the inner circle.

To keep track of diffeomorphisms of a
surface $\Sigma$, we use techniques from the Lego-Teichm\"uller-Game (\cite{BKlego}), or LTG for short and find
a graphical representation of these diffeomorphisms.  
For the definition of the LTG, we refer to \cite{BKlego}; we use the  
notations of this paper. The LTG requires to draw
marking graphs with a base point on the surfaces under consideration:
\be
\raisebox{-50pt}{
\bp(200,150)
\pic{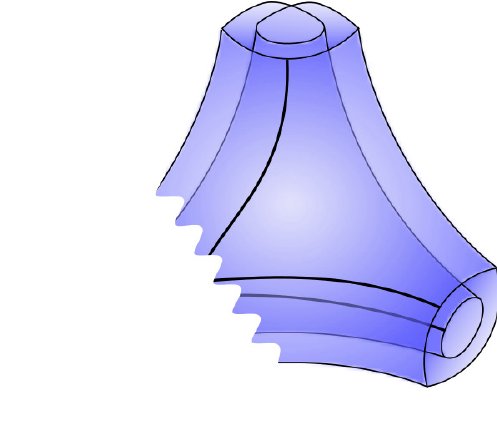}
\put(-65,135) {$M$}
\put(0,30){$A_1\times A_2$}
\ep
}
\ee
We fix a standard marking on the standard sphere $S_n$
by taking the $n$ straight lines that relate the marked
point $k-\frac i3$ to the point $-2i$ for
$k=1,2\ldots n$.

Diffeomorphisms  map marking graphs to marking graphs. Different marking graphs are connected by finite sequences of LTG moves.
This sequence is unique \cite[Theorem 4.24]{BKlego} up to a known set of relations.
The LTG gives rules to translate these moves into natural 
isomorphisms between the corresponding vector spaces. In most cases,
we will suppress the LTG-move $\Z$ (the rotation move) in manipulations of the marking graphs; when translating the LTG-moves into 
morphisms in \C, we will point out at what point one has to 
insert $\Z$-moves.

Similarly, if a surface has been obtained by gluing two boundary circles,
the corresponding circle will be drawn on the surface; such circles are
called \emph{cuts}. A single base point of the sewn surface is obtained
by contracting the link crossing the cut.

\subsection{Dual objects}

We start by finding for every object $V\in\CZ$ a candidate
for the dual object. This will be the object $V^\ast$
representing the functor
\be
\bearlll
\CZ&\to\vect_\Bbbk\\
X&\mapsto \langle V,X\rag.
\eear
\ee
We consider two cases separately.
\begin{itemize}
\item
For $V$ in the neutral component,  $V=A_1\times A_2\in \CC$, we 
consider the standard block $S_2(1,1;1,1)$ whose cover
spaces is  the disjoint union of two copies of $S_2$ and find the object representing the functor
\be
\bearlll
\CC&\to\vect_\Bbbk\\
X_1\times X_2&\mapsto \la A_1\times A_2,X_1\times X_2\rag.
\eear
\ee
Here, we restricted  our attention to the value of the functor on
the neutral component,
since the grading implies that the functor is zero on the 
twisted component. We compute
\be\label{eq:neutraldual}
\begin{split}
\tauzz(S_2(1,1;1,1);V,X)&=\tau(S_2\sqcup S_2;A_1, A_2,X_1, X_2)\\
&\cong\tau(S_2;A_1,X_1)\otimes_\Bbbk\tau(S_2;A_2,X_2)\\
&\defn\homc(\unit,A_1 X_1)\otik\homc(\unit,A_2X_2)\\
&\cong\homc(A_1^\vee,X_1)\otik\homc(A_2^\vee,X_2)\\
&\defn\homcc(A_1^\vee\times A_2^\vee,X_1\times X_2)\\
&\defn\homcz(A_1^\vee\times A_2^\vee,X_1\times X_2).
\end{split}
\ee
Let us explain in this example the various steps in detail; in subsequent
calculations, we will only explain additional new features.
The first equality is by definition of the equivariant
modular functor $\tauzz$ via the cover functor $\cover_\X$; the second isomorphy is the
tensoriality of $\tau$. The third equality follows from the
definition of the modular functor $\tau$ from the modular
category $\C$. The next isomorphism is a consequence of
the duality in the category $\C$, while the last identities
follow from the definition of the Deligne product
$\C\boxtimes \C$ and the definition of $\CX$.
 
Hence we find that $(A_1\times A_2)^\ast\cong A_1^\vee\times A_2^\vee$. 
A choice of the representing object and of a diffeomorphism
is involved in equation (\ref{eq:neutraldual}); 
different choices ultimately lead to equivalent dualities on \CZ.

\item
Let $M$ be an object in the twisted component of \CZ, hence $M\in\C$. The total space of the cover $(S_2(g,g;1,1)\to S_2)$ has only one connected component and is in fact diffeomorphic to $S_2$ as a  smooth manifold, so we find
\be
\begin{split}
\tauzz(S_2(g,g;1,1);M,X)&=\tau(S_2;M,X)\\
&\defn\homc(\unit,MX)\\
&\cong\homc(M^\vee,X)
\end{split}
\ee
We thus find that $M^\ast\cong M^\vee$. Again, a choice of
a representing object and of a diffeomorphism 
$S_2(g,g;1,1)\stackrel{\cong}{\to} S_2$ are involved.
\end{itemize}

So far we only found dual objects, for the further discussion on duality in \CZ we need a tensor product on \CZ. For this reason, 
we return to more aspects of a ribbon structure in section
\ref{sec.4.7}.

\subsection{The tensor product}
The tensor product of two objects $V,W\in\CZ$ is defined 
as the object representing the functor $\la ?,V,W\rag$,
i.e.\
by $\la T,V\otimes W\rag=\la T,V,W\rag$. This again involves a
choice of the representing object. The total space of the standard cover
$S_3((g_1g_2)^{-1},g_1,g_2;1,1,1)$ is isomorphic to the $n$-punctured
sphere $S_n$ for a value of $n$ that depends on the group
elements $g_1$ and $g_2$. We will have to choose such
a diffeomorphism as well.
Different choices of this diffeomorphism lead to isomorphic tensor 
products, but the
choice will enter in the associativity constraints;
hence we have to keep track of this choice. This is done 
by considering the image of the standard marking graph
on $S_n$ that has been introduced at the end of section
\ref{ssec:notconv}. We call this image the 
standard marking graph on $S_3((g_1g_2)^{-1},g_1,g_2;1,1,1)$. 

\begin{itemize}
\item
We first consider two objects $V\equiv A_1\times A_2$ and $W\equiv B_1\times B_2$ in the neutral component of $\CZ$. Since total space of the cover $S_3(1,1,1;1,1,1)$ is just the disjoint union of two three-holed spheres, we find 
\be
\begin{split}
\la T,V\otimes W\rag&\equiv\la T_1\times T_2,A_1\times A_2,B_1\times B_2\rag\\
&\defn\tauzz(S_3(1,1,1;1,1,1);T_1\times T_2,A_1\times A_2,B_1\times B_2)\\
&
\defn
\tau(S_3\sqcup S_3;T_1, T_2,A_1, A_2,B_1, B_2)\\
&\cong\tau(S_3;T_1,A_1,B_1)\otik\tau(S_3;T_2,A_2,B_2)\\
&\defn\homc(\unit,T_1A_1B_1)\otik\homc(\unit,T_2A_2B_2)\\
&\cong\homc(T_1^\vee,A_1B_1)\otik\homc(T_2^\vee,A_2B_2)\\
&\defn\homcc(T_1^\vee\times T_2^\vee,A_1B_1\times A_2B_2)
\end{split}
\ee
Our definition of the tensor product on objects
of the untwisted component thus yields $(A_1\times A_2)\otimes (B_1\times B_2):=A_1B_1\times A_2B_2$, i.e.\  the usual tensor product on $\CC$. The standard marking graph on  the cover $S_3(1,1,1;1,1,1)$ is then

\be
\raisebox{-50pt}{
\bp(200,120)
\pic{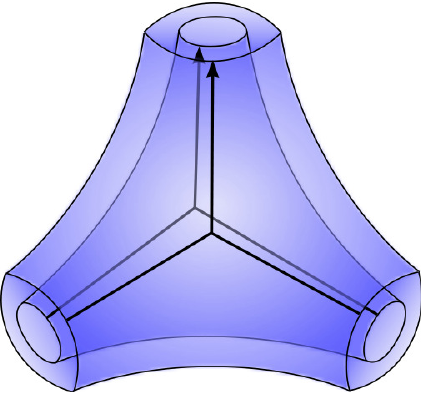}
\put(-78,120) {$T_1\times T_2$}
\put(5,10){$A_1\times A_2$}
\put(-165,10){$B_1\times B_2$}
\ep
}
\ee
\\[.5em]
\item
Next we consider the tensor product of an object in the
untwisted component $A_1\times A_2$ with an object $M$ in the
twisted component.\ The total space $S_3(g,1,g;1,1,1)$
of the relevant cover 
is diffeomorphic to $S_4$ as can be seen by the following 
chain of diffeomorphisms:
  \be
\begin{split}
\raisebox{-40pt}{
  \bp(350,150)
  \put(0,0){\pic{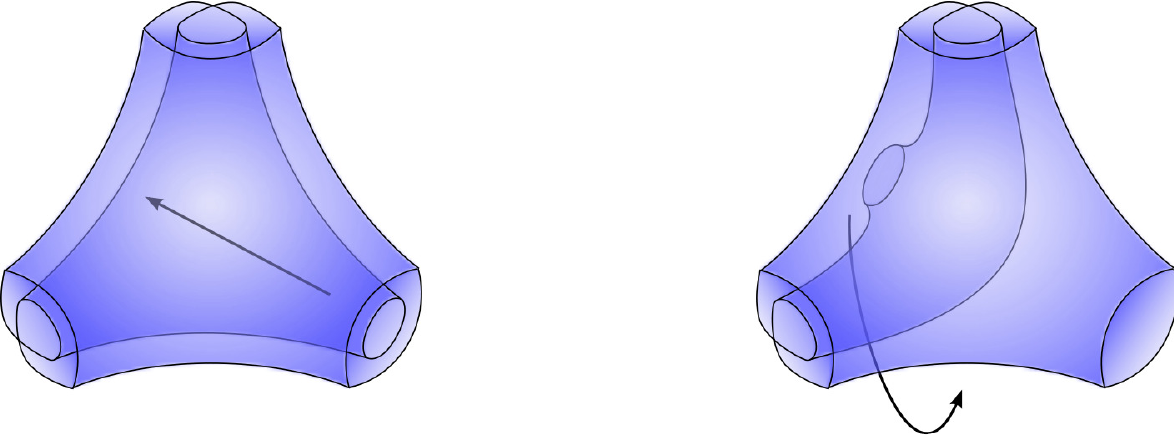}}
  \put(59,130){$T$}
  \put(276,130){$T$}
  \put(160,60){$\congto$}
  \put(125,25){$A_1\times A_2$}
  \put(-16,25){$M$}
  \put(202,25){$M$}
  \put(342,25){$A_1$}
  \put(262,70){$A_2$}
  \ep}\\[1em]
  \raisebox{20pt}{
  \bp(350,120)
  \put(0,0){\pic{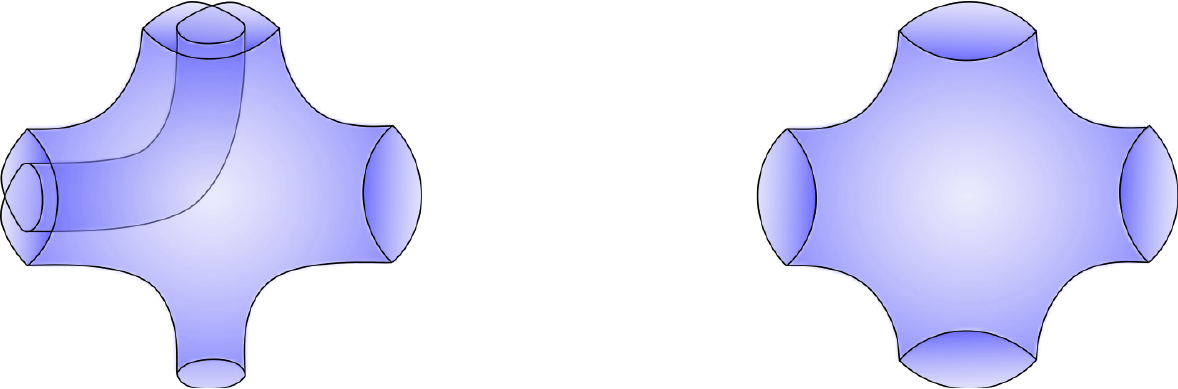}}
  \put(59,115){$T$}
  \put(276,115){$T$}
  \put(160,48){$\congto$}
  \put(-50,48){$\congto$}
  \put(128,50){$A_1$}
  \put(54,-15){$A_2$}
  \put(-18,50){$M$}
  \put(202,50){$M$}
  \put(344,50){$A_1$}
  \put(274,-15){$A_2$}
  \ep}
\end{split}
\label{46}
\ee

In the first and second step we move the inner hole 
labelled by $A_2$ around the component of the boundary with
non-trivial monodromy, labelled by $M$. 
The last step is also an isomorphism of manifolds;\ the reader
should not be confused by the fact that we have to draw
two-dimensional manifolds immersed into the three-dimensional
space ${\mathbb R}^3$. Hence we find that
\be
\begin{split}
\la T,(A_1\times A_2)\otimes M\rag&=\la T,A_1\times A_2,M\rag\\
&\defn\tauzz(S_3(g,1,g;1,1,1);T,A_1\times A_2,M)\\
&\defn\tau(S_3(g,1,g;1,1,1);T,A_1,A_2,M)\\
&\cong\tau(S_4;T,A_1,A_2,M)\\
&\defn\homc(\unit,TA_1A_2M)\\
&\cong\homc(T^\vee,A_1A_2M).
\end{split}
\ee
We are thus lead to the definition
$(A_1\times A_2)\otimes M:=A_1A_2M$. The image of the standard 
marking graph on $S_4$ under the diffeomorphism described
in equation (\ref{46}) gives the following marking on 
$S_3(g,1,g;1,1,1)$:

\be
\raisebox{-50pt}{
\bp(200,120)
\pic{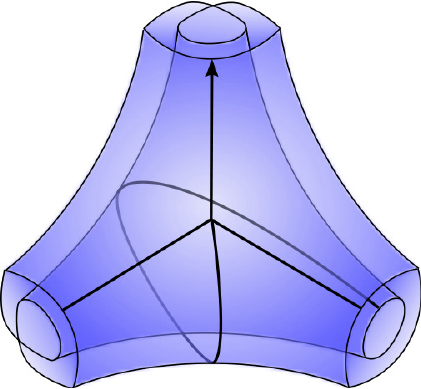}
\put(-65,120) {$T$}
\put(5,10){$A_1\times A_2$}
\put(-140,10){$M$}
\ep
}
\ee
\\[.5em]

\item
The discussion of the tensor product $M\oti (A_1\times A_2)$ 
closely parallels the preceding discussion. Again the manifold $S_3(g,g,1;1,1,1)$ is 
diffeomorphic to the four-punctured sphere $S_4$. 
We introduce a diffeomorphism similar to the one defined
in equation (\ref{46}):

  \be
\begin{split}
\raisebox{-40pt}{
  \bp(370,120)
  \put(0,0){\pic{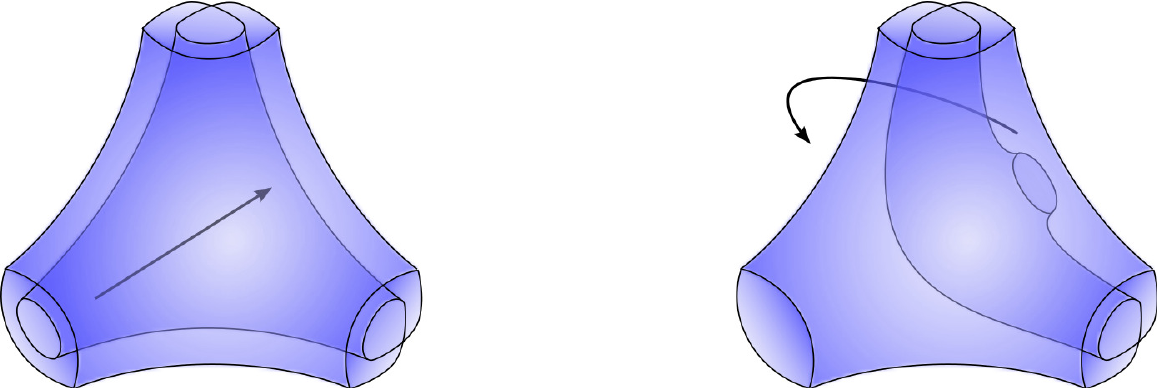}}
  \put(59,118){$T$}
  \put(270,118){$T$}
  \put(160,60){$\congto$}
  \put(-48,13){$A_1\times A_2$}
  \put(126,13){$M$}
  \put(340,13){$M$}
  \put(196,13){$A_1$}
  \put(277,52){$A_2$}
  \ep}\\[2em]
  \raisebox{-45pt}{
  \bp(370,130)
  \put(0,0){\pic{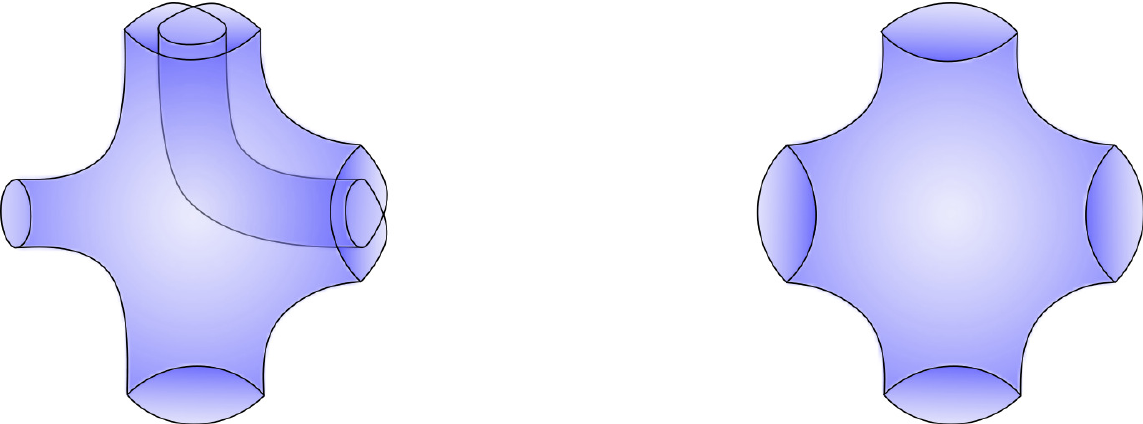}}
  \put(52,125){$T$}
  \put(270,125){$T$}
  \put(160,58){$\congto$}
  \put(-50,58){$\congto$}
  \put(48,-16){$A_1$}
  \put(-22,56){$A_2$}
  \put(118,56){$M$}
  \put(340,56){$M$}
  \put(270,-16){$A_1$}
  \put(198,56){$A_2$}
  \ep}
\end{split}
\ee

We find:
\be
\begin{split}
\la T,M\oti(A_1\times A_2)\rag&\defn\la T,M,A_1\times A_2\rag\\
&\defn\tauzz(S_3(g,g,1;1,1,1);T,M,A_1\times A_2)\\
&\defn\tau(S_3(g,g,1;1,1,1);T,M,A_1,A_2)\\
&\cong\tau(S_4;T,M,A_1,A_2)\\
&\defn\homc(\unit,TMA_1A_2)\\
&\cong\homc(T^\vee,MA_1A_2)
\end{split}
\ee
so that we are lead to define 
$M\oti(A_1\times A_2):=MA_1A_2$ and have the standard marking 
graph on $S_3(g,g,1;1,1,1)$ 

\be
\raisebox{-50pt}{
\bp(150,120)
\pic{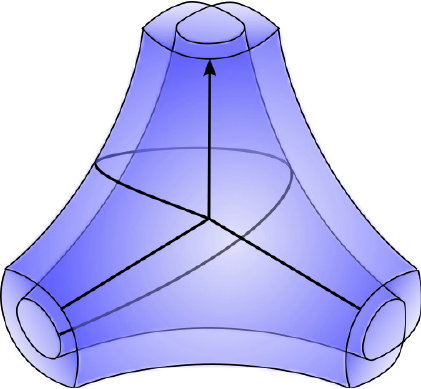}
\put(-65,120) {$T$}
\put(5,10){$M$}
\put(-165,10){$A_1\times A_2$}
\ep
}
\ee
\\[.5em]

These definitions coincide with the ad hoc definitions made
in \cite{bfrs}.

\item
We now turn to a new result about the \ZZ-equivariant tensor product:
the tensor product of two objects in the twisted component 
of \CZ. To this end, we first note that the tensor product functor
\be
\begin{split}
\CC&\to\C\\
\bigoplus_l V_l\times W_l&\mapsto \bigoplus_l V_lW_l
\end{split}
\ee
has the right adjoint \cite[Thm. 2.21]{KR} 
\be
\begin{split}
R:\ \C&\to\CC\\
V&\mapsto\bigoplus_{i\in \I}VU_i^\vee\times U_i \,\, ,
\end{split}
\label{55}
\ee
where the sum is over representatives of isomorphism classes
of simple objects of $\C$.

We have to consider the  manifold $S_3(1,g,g;1,1,1)$ together
with the following diffeomorphisms

  \be
\begin{split}
\raisebox{-45pt}{
  \bp(360,120)
  \put(0,0){\pic{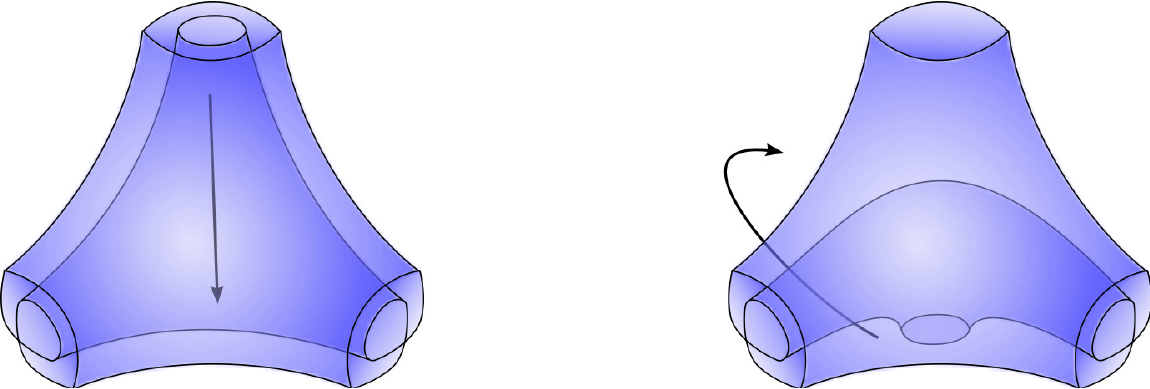}}
  \put(42,118){$T_1\times T_2$}
  \put(267,118){$T_1$}
  \put(160,60){$\congto$}
  \put(-17,13){$N$}
  \put(126,13){$M$}
  \put(340,13){$M$}
  \put(196,13){$N$}
  \put(266,30){$T_2$}
  \ep}\\[2em]
  \raisebox{15pt}{
  \bp(360,140)
  \put(0,0){\pic{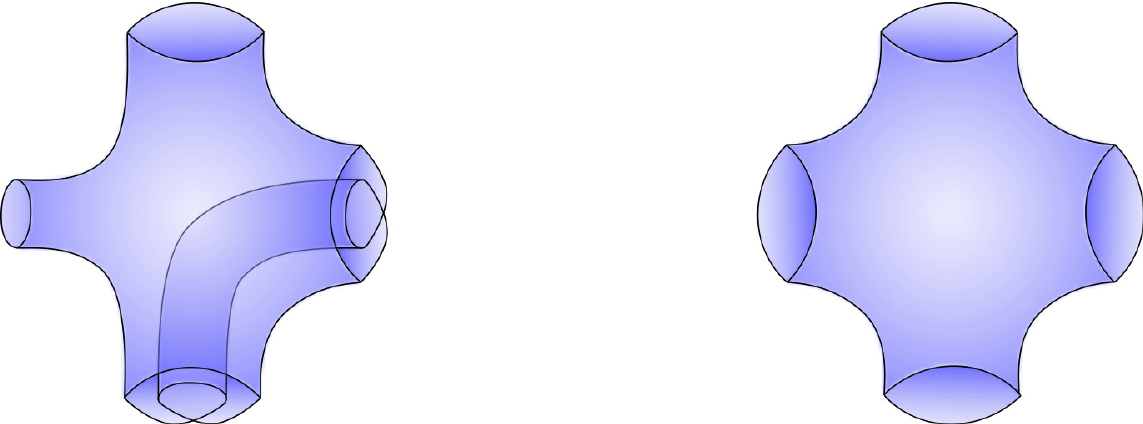}}
  \put(52,128){$T_1$}
  \put(269,128){$T_1$}
  \put(160,55){$\congto$}
  \put(-50,55){$\congto$}
  \put(-17,57){$T_2$}
  \put(120,57){$M$}
  \put(49,-17){$N$}
  \put(340,57){$M$}
  \put(199,57){$T_2$}
  \put(270,-17){$N$}
  \ep}
\end{split}
\ee

We thus find
\be
\begin{split}
\la T_1\times T_2,M\oti N\rag&\defn\la T_1\times T_2,M,N\rag\\
&\defn\tauzz(S_3(1,g,g;1,1,1);T_1\times T_2,M,N)\\
&\defn\tau(S_3(1,g,g;1,1,1);T_1,T_2,M,N)\\
&\cong\tau(S_4;T_2,T_1,M,N)\\
&\defn\homc(\unit,T_2T_1MN)\\
&\cong\homc(T_1^\vee T_2^\vee,MN)\\
&\cong\homcc(T_1^\vee\times T_2^\vee,R(MN))
\end{split}
\ee
where the last isomorphism is given by the adjunction between $R$\ and the tensor product functor. Hence we set
\be
M\oti N=R(MN)=\bigoplus_{i\in\I}MNU_i^\vee\times U_i.
\ee
We note that the adjunction morphism
\be\label{eqn:Radj}
\begin{split}
\homc(T_1^\vee T_2^\vee,MN)&\cong \bigoplus_{i\in \I}\homcc(T_1^\vee\times T_2^\vee,MNU_i^\vee\times U_i)\\
&\cong \bigoplus_{i\in\I}\homc(T_1^\vee,MNU_i^\vee)\otik\homc(T_2^\vee,U_i)
\end{split}
\ee
coincides with the gluing isomorphism

\be
\raisebox{20pt}{
\bp(370,150)
\put(0,0){\pic{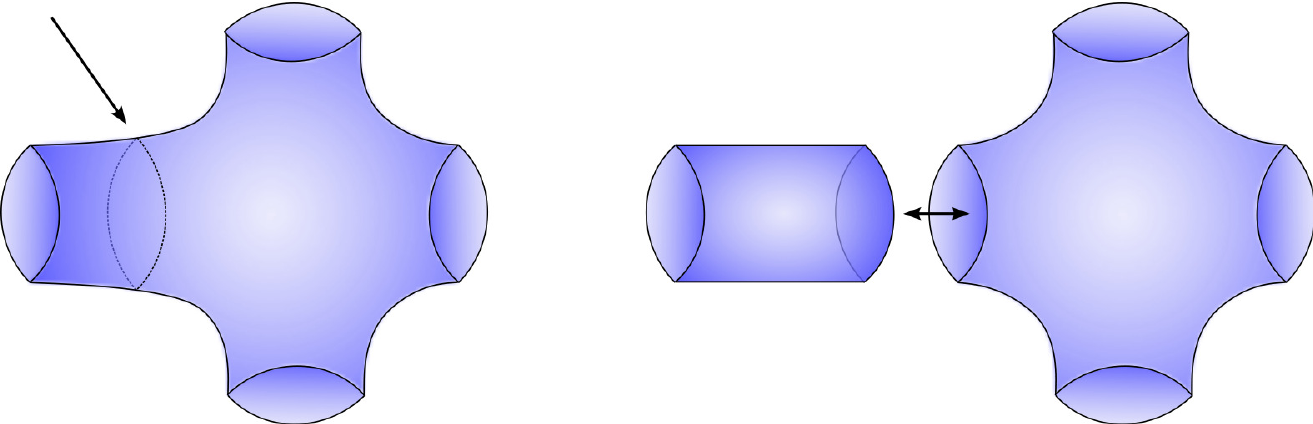}}
  \put(80,128){$T_1$}
  \put(-22,56){$T_2$}
  \put(-30,120){cut for gluing}
  \put(317,128){$T_1$}
  \put(142,56){$\congto\bigoplus_{i\in\I}$}
  \put(79,-16){$N$}
  \put(126,83){$M$}
  \put(385,56){$M$}
  \put(319,-16){$N$}
  \put(186,83){$T_2$}
  \put(243,83){$U_i$}
  \put(271,84){$U_i^\vee$}
\ep
}
\ee

so that on the total space $S_3(1,g,g;1,1,1)$ of the cover
we have the standard marking graph

\be
\raisebox{-50pt}{
\bp(160,120)
\pic{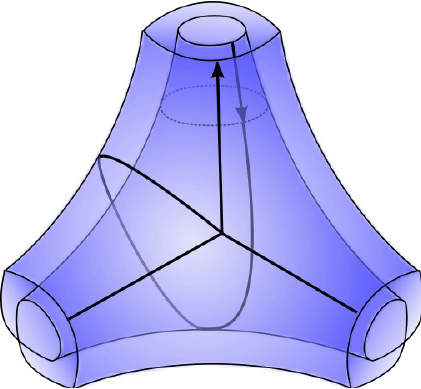}
\put(-77,118) {$T_1\times T_2$}
\put(5,10){$M$}
\put(-137,10){$N$}
\ep
}
\ee\\
which has a cut drawn on the manifold close to the insertion of $T_2$.
The second arrow in the marking on $S_3(1,g,g;1,1,1)$ comes from the marking on $S_2(U_i, T_2)$ and
fixes the order of the objects in $\homc(\unit,U_iT_2)$.
\end{itemize}

Formulae for the dimensions of spaces of conformal blocks in permutation 
orbifolds have been given in \cite{bohs,Ba2}. It is easy to see that
these formulae imply that dimension of the spaces of conformal three-point 
blocks of the equivariant theory on the sphere involving two objects 
in the twisted sector are given by the dimension of the spaces of
four-point blocks on the sphere for $\C$. This nicely follows
from the geometry of the cover functor: the total space of the relevant cover 
is isomorphic to the four-punctured sphere. The analysis can be extended
to other conformal blocks as well.

We finally determine the tensor unit $\unit^{\ZZ}$ of
$\CZ$ that is  defined by 
\be
\la\unit^{\ZZ},U\rag:=\la U\rag
\ee
for all $U$ in \CZ. The classification of covers of
the one-holed sphere implies that
$\la U\rag$ is only defined when $U$ is in the neutral component, 
$U\in\CZ_1$. With $U=A_1\times A_2\in \CC$, we find $\la A_1\times A_2\rag\cong \la A_1\ra\otik\la A_2\ra$ and hence 
\be
\unit^{\ZZ}=\unit\times\unit.
\ee
Since the modular tensor category $\C$ was supposed to be 
strict, the unit constraints of $\CX$ are the identity morphisms.

\subsection{The associativity constraints}

The next step is to derive the associativity constraints for the tensor product given in the previous section. As special cases, we will find
the mixed associativity constraints for which
ans\"atze were proposed in \cite{bfrs}.

For any choice of three elements $p,q,r\in\ZZ$ we need to consider the 
\ZZ-cover 
$$S_4((pqr)^{-1},p,q,r;1,1,1,1)$$ 
of the four-holed sphere $S_4$. This 
cover will be cut in two different ways, representing the tensor 
product $(A\otimes B)\otimes C$ and $A\otimes (B\otimes C)$,
respectively. The two ways of cutting give two different 
markings on $S_4((pqr)^{-1},p,q,r;1,1,1,1)$ which in turn 
represent two different diffeomorphisms from the total
space of the cover
\be
f,g:S_4((pqr)^{-1},p,q,r;1,1,1,1)\congto S
\ee
to the appropriate standard block $S$, which is either a
punctured sphere, a disjoint union of  punctured spheres 
or a surface of genus one.
We then consider for all objects $T\in\CX$ the following chain of natural transformations:
\be\label{eq:66}
\begin{split}
&\homcz(T^\ast,(A\otimes B)\otimes C)\congto \tau(S;T,A,B,C)\\
&\stackrel{f_\ast^{-1}}{\to}\tau(S_4((pqr)^{-1},p,q,r;1,1,1,1);T,A,B,C)
\stackrel{g_\ast}{\to}\tau(S;T,A,B,C)\\
&\congto\homcz(T^\ast,A\otimes (B\otimes C)).
\end{split}
\ee
The first and the last isomorphism in \erf{eq:66} are determined by the definition of the tensor products $(A\oti B)\oti C$ and $A\oti(B\oti C)$ as objects in \C or \CC respectively.
The Yoneda lemma implies that this natural transformation comes
from a  natural isomorphism $\alpha_{A,B,C}:(A\otimes B)\otimes C\congto A\otimes (B\otimes C)$. The arguments of \cite{kpGMF,BK} then imply
that these isomorphisms satisfy the pentagon axiom. 

We summarize our strategy:
\begin{itemize}

\item Determine the two marking graphs on 
$S_4((pqr)^{-1},p,q,r;1,1,1,1)$ induced on the cover by 
cutting $S_4$ in the two ways determined by associativity and by our definition of the tensor product. 
Denote the marking graph representing $(A\otimes B)\otimes C$ by $M_1$ and the marking graph representing $A\otimes (B\otimes C)$ by $M_2$.

\item Determine the standard manifold $S$ such that $S\cong S_4((pqr)^{-1},p,q,r;1,1,1,1)$.
Transform the surface $S_4((pqr)^{-1},p,q,r;1,1,1,1)$ 
with the marking graph $M_1$ to the standard manifold $S$
in the way prescribed by the marking $M_2$.

\item This yields $S$, together with a marking graph.  Determine the LTG-moves that transform this graph into the standard marking graph on $S$.
\item Translate these LTG-moves into morphisms in \C or \CC.
\end{itemize}
Since we need to do this for any choice of $p,q,r\in\ZZ$, we will get eight different associativity constraints $\alpha_{A,B,C}$, 
depending on the sector of the objects $A,B,C$.

\subsubsection{Three objects from the untwisted component}
To illustrate the method,
we start with the easiest case, three objects $A_1\times A_2$, 
$B_1\times B_2$ and $C_1\times C_2$ in the untwisted component.
The total space $S_4(1,1,1,1;1,1,1,1)$ is just a disjoint union of two four-holed spheres. Both cutting procedures yield the same marking graph on $S_4\sqcup S_4$:

\be
\raisebox{0pt}{
\bp(130,130)
\put(0,0) {\pic{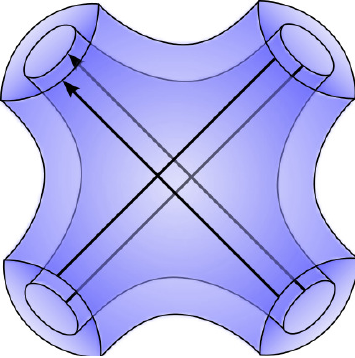}}
\put(95,100){$A_1\times A_2$}
  \put(-30,-5){$C_1\times C_2$}
  \put(-30,100){$T_1\times T_2$}
  \put(95,-5){$B_1\times B_2$}
\ep
}
\ee
\\

Hence we arrive at the trivial LTG-move. Taking into
account that $\C$ is supposed to be strict, we find that 
$\alpha_{A_1\times A_2,B_1\times B_2,C_1\times C_2}$ is 
the identity on $A_1B_1C_1\times A_2B_2C_2$.

\subsubsection{One object from the twisted component}

Next we derive the associativity isomorphisms involving one 
object in the twisted component and two objects in the untwisted 
component. To explain our prescription in detail, we first 
discuss the associativity constraint $\alpha_{A_1\times A_2,B_1\times B_2,M}$
in greater detail.

The first cutting amounts on
the total space of $S_4(g,1,1,g;1,1,1,1)$ to the gluing isomorphism
in $\C\boxtimes \C$:

\be
\begin{split}
&\bigoplus_{i,j\in\I} \la T,U_i^\vee\times U_j^\vee,M\rag\otik \la U_i\times U_j,A_1\times A_2,B_1\times B_2\rag\\
\congto&\la T,A_1\times A_2,B_1\times B_2,M\rag
\end{split}
\ee
With the standard markings on $S_3$-covers introduced in
section \ref{ssec:notconv}, we arrive at the following picture:

\be
\raisebox{0pt}{
\bp(250,130)
\put(0,0){\pic{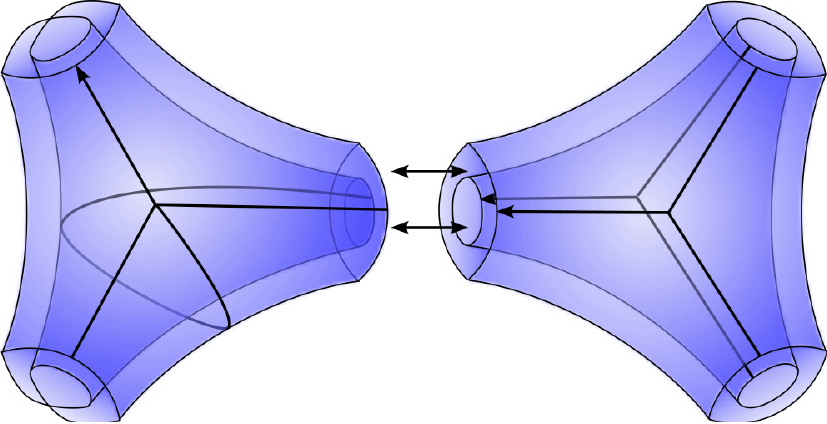}}
\put(80,90){$U_i^\vee\times U_j^\vee$}
\put(0,-5){$M$}
\put(0,120){$T$}
\put(120,25){$U_i\times U_j$}
\put(235,-5){$B_1\times B_2$}
\put(235,120){$A_1\times A_2$}
\ep
}
\ee
By contracting the marking along the factorizing link, we get on
$S_4(g,1,1,g;1,1,1,1)$ the marking
\be
\raisebox{10pt}{
\bp(130,120)
\put(0,0){\pic{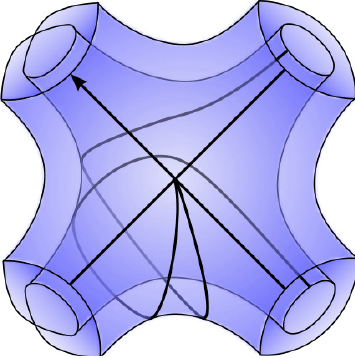}}
\put(90,100){$A_1\times A_2$}
\put(-5,-5){$M$}
\put(95,-5){$B_1\times B_2$}
\put(-5,100){$T$}
\ep
}
\label{67}
\ee
The second gluing procedure is the isomorphism
\be
\begin{split}
&\bigoplus_{i\in \I}\la T,A_1\times A_2,U_i^\vee\rag\otik\la U_i,B_1\times B_2,M\rag\\
\congto&\la T,A_1\times A_2,B_1\times B_2,M\rag
\end{split}
\ee
with the pictorial description
\be
\raisebox{20pt}{
\bp(250,140)
\put(0,0){\pic{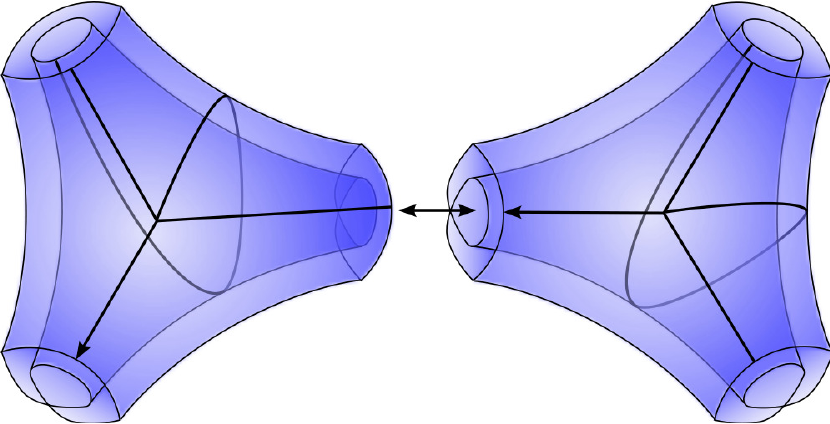}}
\put(99,86){$U_i^\vee$}
\put(0,-5){$T$}
\put(-32,120){$A_1\times A_2$}
\put(130,28){$U_i$}
\put(235,-5){$M$}
\put(235,120){$B_1\times B_2$}
\ep
}
\ee
Again contracting the marking along the factorizing link, we get on $S_4(g,1,1,g;1,1,1,1)$ the
second marking
\be
\raisebox{10pt}{
\bp(130,120)
\put(0,0){\pic{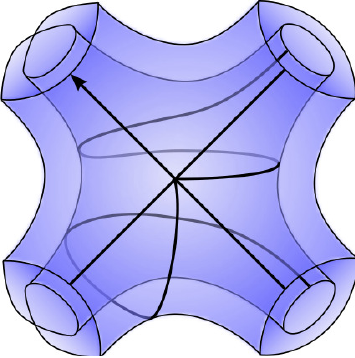}}
\put(90,100){$A_1\times A_2$}
\put(-5,-5){$M$}
\put(95,-5){$B_1\times B_2$}
\put(-5,100){$T$}
\ep
}
\label{70}
\ee

The surface $S_4(g,1,1,g;1,1,1,1)$ is isomorphic to a sphere $S_6$
with $6$ punctures. We draw the graph (\ref{67}) obtained
from the first gluing procedure on $S_4(g,1,1,g;1,1,1,1)$ 
and use the isomorphism to $S_6$ encoded in the graph
(\ref{70}) obtained from the second gluing procedure: 

\be
\begin{split}
\raisebox{20pt}{
\bp(280,145)
\put(-5,0){\raisebox{6pt}{\pic{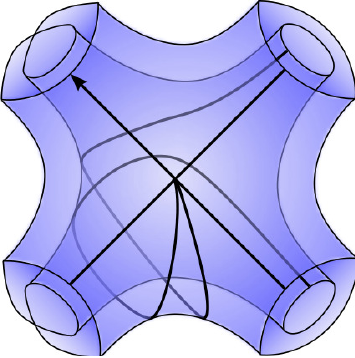}}\qquad\qquad\qquad \pic{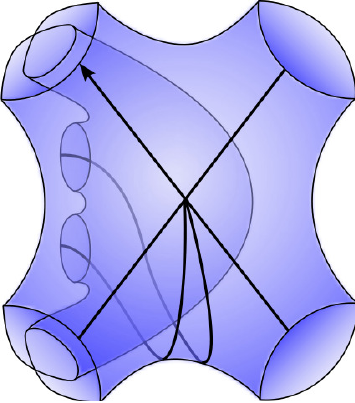}}
  \put(82,110){$A_1\times A_2$}
  \put(262,110){$A_1$}
  \put(125,55){$\congto$}
  \put(0,110){$T$}
  \put(-12,3){$M$}
  \put(90,3){$B_1\times B_2$}
  \put(167,110){$T$}
  \put(158,3){$M$}
  \put(266,3){$B_1$}
  \put(168,38){$B_2$}
  \put(170,65){$A_2$}
\ep
}\\
\raisebox{20pt}{
\bp(295,145)
\put(0,0){\raisebox{6pt}{\pic{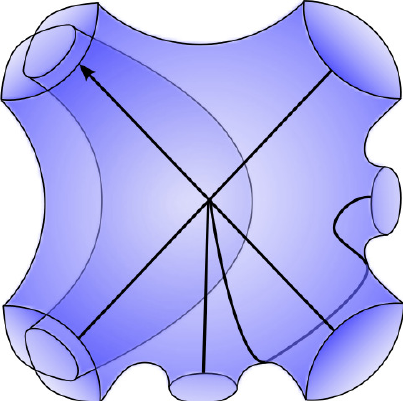}}\qquad\qquad\qquad \pic{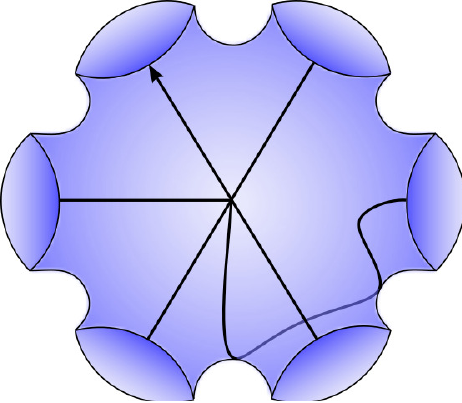}}
  \put(-32,55){$\congto$}
  \put(110,120){$A_1$}
  \put(291,110){$A_1$}
  \put(142,55){$\congto$}
  \put(0,120){$T$}
  \put(-12,3){$M$}
  \put(110,3){$B_1$}
  \put(207,110){$T$}
  \put(168,53){$M$}
  \put(292,-3){$B_1$}
  \put(207,-3){$B_2$}
  \put(325,53){$A_2$}
  \put(53,-8){$B_2$}
  \put(119,60){$A_2$}
\ep
}
\end{split}
\ee

The LTG-moves that transform this marking into the standard marking on $S_6$ are easily read off:
\be
\raisebox{20pt}{
\bp(290,145)
\put(-40,0){\pic{g11g-transform-4}\qquad\qquad\qquad\qquad \pic{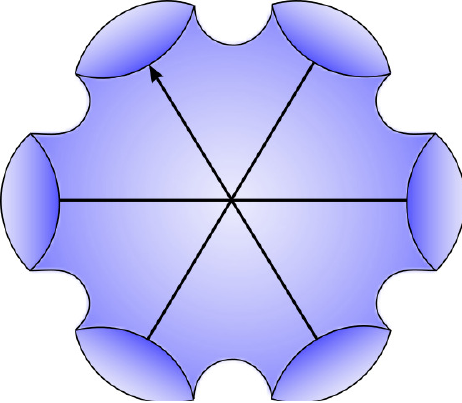}}
  \put(70,110){$A_1$}
  \put(291,110){$A_1$}
  \put(130,55){$\stackrel{\B_{B_1,A_2}}{\longrightarrow}$}
  \put(-20,110){$T$}
  \put(-58,53){$M$}
  \put(-20,-3){$B_2$}
  \put(207,110){$T$}
  \put(168,53){$M$}
  \put(292,-3){$B_1$}
  \put(207,-3){$B_2$}
  \put(325,53){$A_2$}
  \put(70,-3){$B_1$}
  \put(100,53){$A_2$}
\ep
}
\ee

The LTG-move $\B_{B_1,A_2}$ corresponds to the natural transformation 
on \\$\hom_\C(\unit,TA_1B_1A_2B_2M)$ that is induced by the 
braiding in \C, pictorially,
\be
\raisebox{10pt}{
\bp(200,60)
\put(0,0){\rib{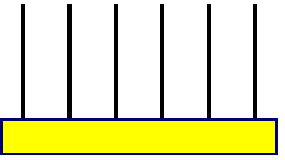}}
\put(140,0){\rib{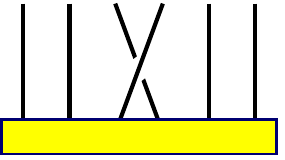}}
  \put(4,47){\scriptsize $T$}
  \put(14,47){\scriptsize $A_1$}
  \put(30,47){\scriptsize $B_1$}
  \put(43,47){\scriptsize $A_2$}
  \put(55,47){\scriptsize $B_2$}
  \put(69,47){\scriptsize $M$}
\put(95,20){$\stackrel{\B_{B_1,A_2}}{\mapsto}$}
\put(140,0){
  \put(4,47){\scriptsize $T$}
  \put(14,47){\scriptsize $A_1$}
  \put(30,47){\scriptsize $A_2$}
  \put(43,47){\scriptsize $B_1$}
  \put(55,47){\scriptsize $B_2$}
  \put(69,47){\scriptsize $M$}
}
\ep
}
\label{76}
\ee
We now apply the isomorphism $\hom_\C(\unit,TA_1B_1A_2B_2M)\cong\hom_\C(T^\vee,A_1B_1A_2B_2M)$ induced by the ribbon structure on \C. Setting $T^\vee:=A_1B_1A_2B_2M$ 
and evaluating the natural transformation (\ref{76}) on the 
identity, we obtain the morphism for the associativity constraint
as  $\alpha_{A_1\times A_2,B_1\times B_2,M}=\id_{A_1}\otimes c_{B_1,A_2}\otimes \id_{B_2M}$.

These are precisely the mixed associativity constraints proposed
in \cite{bfrs}; in that paper, a whole family of 
mixed associativity constraints was given. Here we read off
the morphisms from marking graphs on surfaces. 
These graphs depend on the choice of diffeomorphisms made 
in subsection \ref{ssec:notconv}. Different choices of diffeomorphisms lead to different associativity constraints, of which some were already 
proposed in \cite{bfrs}. Different choices of diffeomorphisms are related by elements of the mapping class group of the relevant 
surface. These group elements can be translated into morphisms
in $\C$ which can be used to endow
the identity functor on \CZ with the structure of an 
monoidal equivalence of monoidal categories.

For the associativity constraint 
$\alpha_{M,A_1\times A_2,B_1\times B_2}$, we get a similar picture. The first gluing procedure on $S_4(g,g,1,1;1,1,1,1)$ gives the marking
\be
\raisebox{20pt}{
\bp(130,130)
\put(0,0){\pic{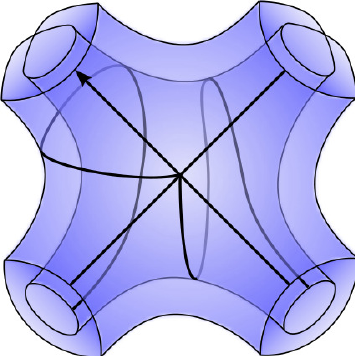}}
\put(95,100){$M$}
\put(-35,-5){$B_1\times B_2$}
\put(95,-5){$A_1\times A_2$}
\put(-5,100){$T$}
\ep
}
\ee
while the second procedure gives the marking
\be
\raisebox{20pt}{
\bp(130,130)
\put(0,0){\pic{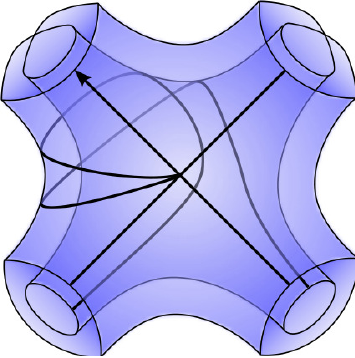}}
\put(95,100){$M$}
\put(-35,-5){$B_1\times B_2$}
\put(95,-5){$A_1\times A_2$}
\put(-5,100){$T$}
\ep
}
\ee

On the sphere $S_6$ with $6$ punctures, this gives the LTG-move

\be
\raisebox{20pt}{
\bp(280,145)
\put(-40,0){\pic{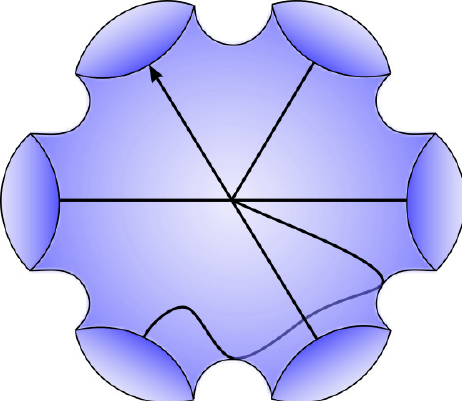}\qquad\qquad\qquad\qquad \pic{S6-paths}}
  \put(130,55){$\stackrel{\B_{B_1,A_2}^{-1}}{\longrightarrow}$}
  \put(70,110){$M$}
  \put(-20,110){$T$}
  \put(-58,53){$B_2$}
  \put(70,-3){$B_1$}
  \put(100,53){$A_1$}
  \put(-20,-3){$A_2$}
\put(225,0){
  \put(70,110){$M$}
  \put(-20,110){$T$}
  \put(-58,53){$B_2$}
  \put(70,-3){$B_1$}
  \put(100,53){$A_1$}
  \put(-20,-3){$A_2$}
}
\ep
}
\ee

By the same reasoning as after equation (\ref{76}),
we conclude that $\alpha_{M,A_1\times A_2,B_1\times B_2}=\id_{MA_1}\otimes c_{B_1,A_2}^{-1}\otimes \id_{B_2}$. This is one of the
constraints in  \cite[Corollary 3]{bfrs}.

For the associativity constraint $\alpha_{A_1\times A_2,M,B_1\times B_2}$ we consider $S_4(g,1,g,1;1,1,1,1)\cong S_6$. The two gluing procedures give the two markings

\be
\raisebox{20pt}{
\bp(280,130)
\put(0,0){\pic{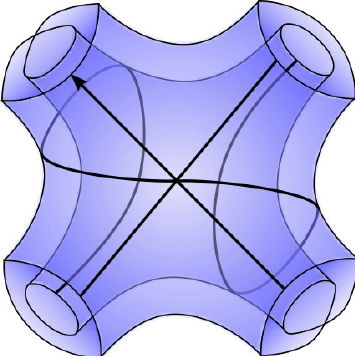}\qquad\qquad\qquad \pic{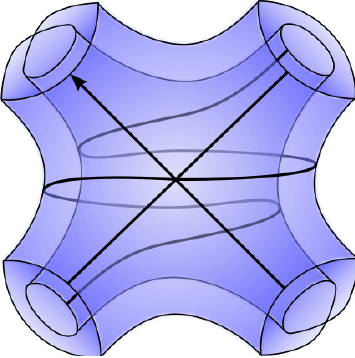}}
\put(95,100){$A_1\times A_2$}
\put(-35,-5){$B_1\times B_2$}
\put(95,-5){$M$}
\put(-5,100){$T$}
\put(180,0){
\put(90,100){$A_1\times A_2$}
\put(-35,-5){$B_1\times B_2$}
\put(90,-5){$M$}
\put(-5,100){$T$}
}
\ep
}
\ee

Our general prescription gives the following marking graph on $S_6$:

\be
\raisebox{20pt}{
\bp(87,135)
\put(-40,0){\pic{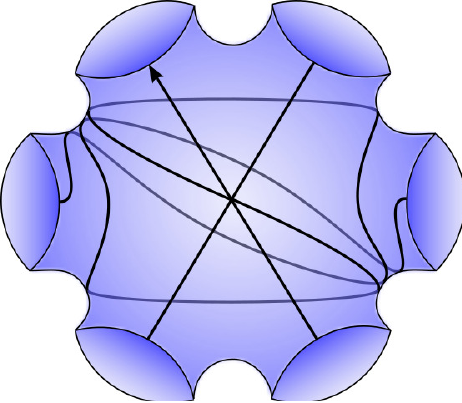}}
  \put(70,110){$A_1$}
  \put(-20,110){$T$}
  \put(-58,53){$B_2$}
  \put(70,-3){$M$}
  \put(100,53){$A_2$}
  \put(-20,-3){$B_1$}
\ep
}
\ee

The following sequence of four LTG-moves transforms this marking
graph into the standard marking graph on 
$S_6$:
\be
\begin{split}
\raisebox{20pt}{
\bp(280,145)
\put(-30,0){\pic{g1g1-LTG-move}}
\put(180,0){\pic{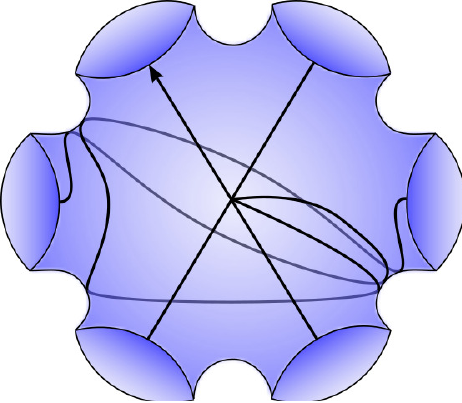}}
  \put(80,110){$A_1$}
  \put(-10,110){$T$}
  \put(-48,53){$B_2$}
  \put(80,-3){$M$}
  \put(105,53){$A_2$}
  \put(-10,-3){$B_1$}
  \put(125,53){$\stackrel{\B_{TA_1,B_2}^{-1}}{\rightarrow}$}
\put(210,0){
  \put(80,110){$A_1$}
  \put(-10,110){$T$}
  \put(-48,53){$B_2$}
  \put(80,-3){$M$}
  \put(110,53){$A_2$}
  \put(-10,-3){$B_1$}
  }
\ep
}\\
\raisebox{20pt}{
\bp(280,145)
\put(-30,0){\pic{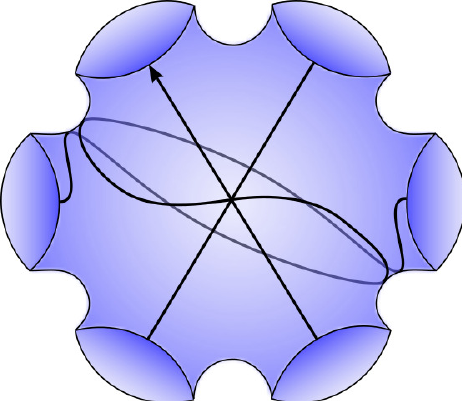}}
\put(180,0){\pic{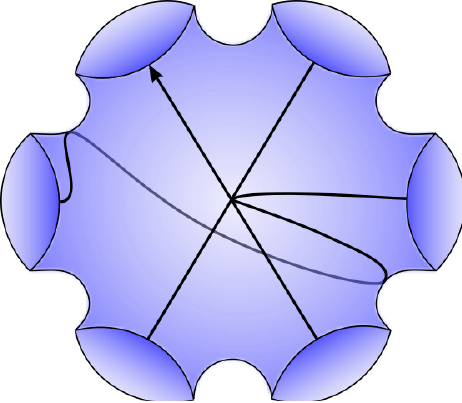}}
\put(120,53){$\stackrel{\B_{B_2MB_1,A_2}}{\rightarrow}$}
\put(-80,53){$\stackrel{\B_{MB_1,A_2}^{-1}}{\rightarrow}$}
  \put(80,110){$A_1$}
  \put(-10,110){$T$}
  \put(-45,53){$B_2$}
  \put(80,-3){$M$}
  \put(105,53){$A_2$}
  \put(-10,-3){$B_1$}
\put(210,0){
  \put(80,110){$A_1$}
  \put(-10,110){$T$}
  \put(-45,53){$B_2$}
  \put(80,-3){$M$}
  \put(110,53){$A_2$}
  \put(-10,-3){$B_1$}
  }
\ep
}\\
\raisebox{20pt}{
\bp(280,145)
\put(0,0){\pic{S6-paths}}
\put(-80,53){$\stackrel{\B_{TA_1A_2,B_2}}{\rightarrow}$}
\put(30,0){
  \put(80,110){$A_1$}
  \put(-10,110){$T$}
  \put(-48,53){$B_2$}
  \put(80,-3){$M$}
  \put(110,53){$A_2$}
  \put(-10,-3){$B_1$}
}
\ep
}
\end{split}
\ee
The sequence 
\be
\B_{TA_1A_2,B_2}\circ \B_{B_2MB_1,A_2}\circ \B_{MB_1,A_2}^{-1}\circ \B_{TA_1,B_2}^{-1}
\ee
of LTG-moves is translated into the natural transformation 
\be
\homc(\unit,TA_1A_2MB_1B_2)\congto\homc(\unit,TA_1A_2MB_1B_2),
\ee
where we will need to insert the appropriate $\Z$-moves implementing
cyclicity:
\be
\begin{split}
\raisebox{40pt}{
\bp(390,100)
\put(0,0){\rib{6ribbons}}
\put(120,-20){\rib{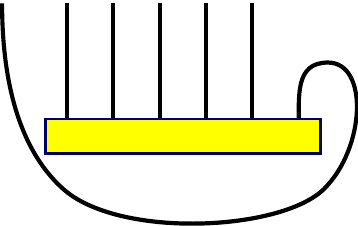}}
\put(270,-20){\rib{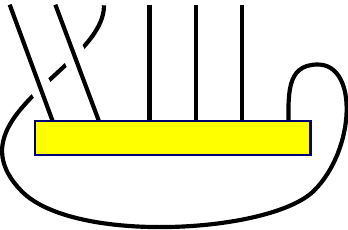}}
  \put(4,47){\scriptsize $T$}
  \put(14,47){\scriptsize $A_1$}
  \put(30,47){\scriptsize $A_2$}
  \put(43,47){\scriptsize $M$}
  \put(55,47){\scriptsize $B_1$}
  \put(69,47){\scriptsize $B_2$}
\put(95,20){$\stackrel{\Z}{\mapsto}$}
\put(232,20){$\stackrel{\B_{TA_1,B_2}^{-1}}{\mapsto}$}
\put(140,0){
  \put(-4,47){\scriptsize $T$}
  \put(7,47){\scriptsize $A_1$}
  \put(22,47){\scriptsize $A_2$}
  \put(48,47){\scriptsize $B_1$}
  \put(-23,47){\scriptsize $B_2$}
  \put(36,47){\scriptsize $M$}
}
\put(280,0){
  \put(-10,47){\scriptsize $T$}
  \put(2,47){\scriptsize $A_1$}
  \put(29,47){\scriptsize $A_2$}
  \put(55,47){\scriptsize $B_1$}
  \put(15,47){\scriptsize $B_2$}
  \put(42,47){\scriptsize $M$}
}
\ep
}\\
\raisebox{0pt}{
\bp(390,150)
\put(0,40){\rib{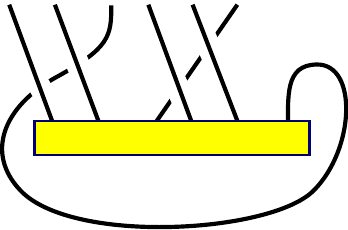}}
\put(150,20){\rib{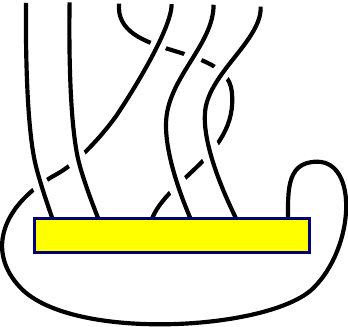}}
\put(290,0){\rib{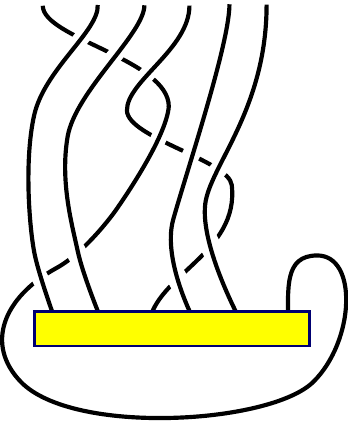}}
\put(-38,73){$\stackrel{\B_{MB_1,A_2}^{-1}}{\rightarrow}$}
\put(105,73){$\stackrel{\B_{B_2MB_1,A_2}}{\rightarrow}$}
\put(250,73){$\stackrel{\B_{TA_1A_2,B_2}}{\rightarrow}$}
  \put(0,109){\scriptsize $T$}
  \put(13,109){\scriptsize $A_1$}
  \put(65,109){\scriptsize $A_2$}
  \put(40,109){\scriptsize $M$}
  \put(52,109){\scriptsize $B_1$}
  \put(27,109){\scriptsize $B_2$}
\put(155,47){
  \put(0,69){\scriptsize $T$}
  \put(12,69){\scriptsize $A_1$}
  \put(26,69){\scriptsize $A_2$}
  \put(52,69){\scriptsize $M$}
  \put(66,69){\scriptsize $B_1$}
  \put(40,69){\scriptsize $B_2$}
}
\put(300,55){
  \put(14,69){\scriptsize $T$}
  \put(27,69){\scriptsize $A_1$}
  \put(40,69){\scriptsize $A_2$}
  \put(52,69){\scriptsize $M$}
  \put(63,69){\scriptsize $B_1$}
  \put(-3,69){\scriptsize $B_2$}
}
\ep
}\\
\raisebox{0pt}{
\bp(390,150)
\put(0,0){\rib{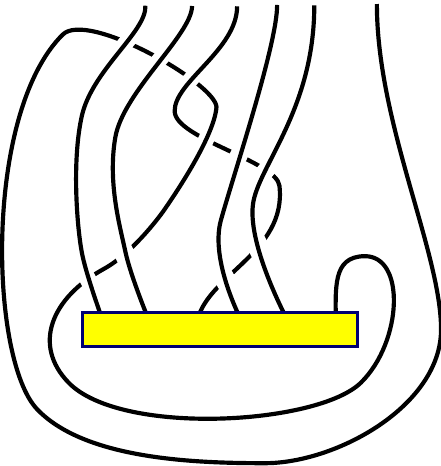}}
\put(180,15){\rib{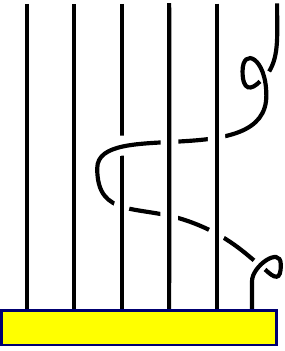}}
\put(310,15){\rib{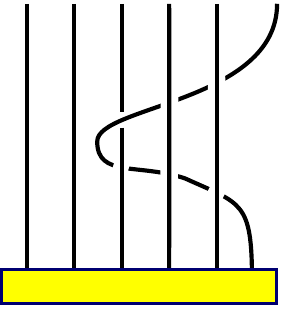}}
\put(-20,60){$\stackrel{\Z^{-1}}{\mapsto}$}
\put(150,60){$=$}
\put(280,60){$=$}
  \put(39,137){\scriptsize $T$}
  \put(52,137){\scriptsize $A_1$}
  \put(65,137){\scriptsize $A_2$}
  \put(77,137){\scriptsize $M$}
  \put(87,137){\scriptsize $B_1$}
  \put(106,137){\scriptsize $B_2$}
\put(145,-18){
  \put(39,137){\scriptsize $T$}
  \put(52,137){\scriptsize $A_1$}
  \put(65,137){\scriptsize $A_2$}
  \put(80,137){\scriptsize $M$}
  \put(94,137){\scriptsize $B_1$}
  \put(111,137){\scriptsize $B_2$}
}
\put(275,-30){
  \put(39,137){\scriptsize $T$}
  \put(52,137){\scriptsize $A_1$}
  \put(65,137){\scriptsize $A_2$}
  \put(80,137){\scriptsize $M$}
  \put(94,137){\scriptsize $B_1$}
  \put(111,137){\scriptsize $B_2$}
}
\ep
}
\end{split}
\ee
Using again the isomorphy $\homc(\unit,TA_1A_2MB_1B_2)\cong\homc(T^\vee,A_1A_2MB_1B_2)$
implied by duality and evaluating on the identity morphism gives
the associativity constraint

\be
\raisebox{0pt}{
\bp(200,100)
\put(100,0){\rib{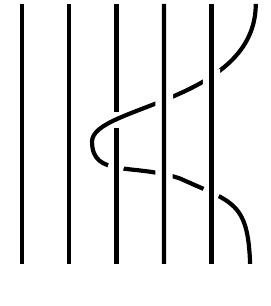}}
\put(0,50){$\alpha_{A_1\times A_2,M,B_1\times B_2}=$}
\put(64,-45){
  \put(39,137){\scriptsize $T$}
  \put(52,137){\scriptsize $A_1$}
  \put(65,137){\scriptsize $A_2$}
  \put(80,137){\scriptsize $M$}
  \put(94,137){\scriptsize $B_1$}
  \put(108,137){\scriptsize $B_2$}
}
\put(64,-135){
  \put(39,137){\scriptsize $T$}
  \put(52,137){\scriptsize $A_1$}
  \put(65,137){\scriptsize $A_2$}
  \put(80,137){\scriptsize $M$}
  \put(94,137){\scriptsize $B_1$}
  \put(108,137){\scriptsize $B_2$}
}
\ep
}
\ee

\subsubsection{Two  objects in the twisted component}

We next discuss associativity constraints involving two 
objects in the twisted and one in the untwisted component of \CZ. The tensor product of three such objects lies in the untwisted component. Again the
relevant marking graphs on covers of the four-punctured
sphere contain cuts. 

We start with the associativity constraint $\alpha_{M,N,A_1\times A_2}$.
The two gluing procedures over $S_4$ yield the following two markings on $S_4(1,g,g,1;1,1,1,1)$, where we 
already removed the cuts and contracted the factorizing links:

\be
\raisebox{20pt}{
\bp(280,130)
\put(0,0){\pic{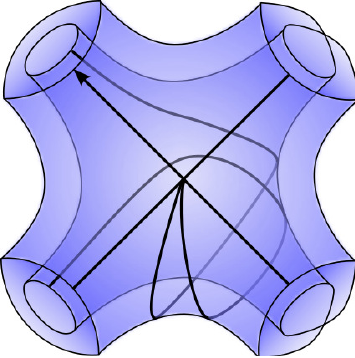}\qquad\qquad\qquad\qquad \pic{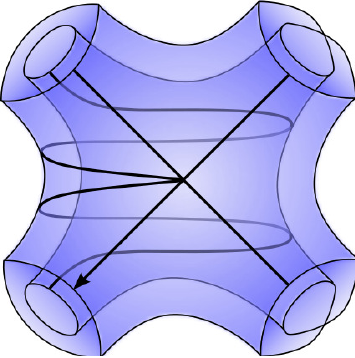}}
\put(95,100){$M$}
\put(-35,-5){$A_1\times A_2$}
\put(95,-5){$N$}
\put(-29,100){$T_1\times T_2$}
\put(195,0){
\put(95,100){$M$}
\put(-35,-5){$A_1\times A_2$}
\put(95,-5){$N$}
\put(-29,100){$T_1\times T_2$}
}
\ep
}
\ee

This gives the following marking on the six-punctured sphere
$S_6$:

\be
\raisebox{20pt}{
\bp(77,135)
\put(-40,0){\pic{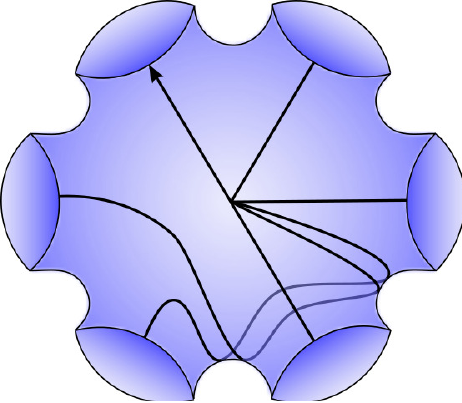}}
  \put(70,110){$M$}
  \put(-22,110){$T_1$}
  \put(-56,53){$T_2$}
  \put(68,-3){$A_1$}
  \put(100,53){$N$}
  \put(-22,-3){$A_2$}
\ep
}
\ee

It is transformed into the standard marking on $S_6$ by the LTG-moves $\B_{A_1,A_2}^{-1}\circ \B_{A_1,T_2}^{-1}$. Now we need to translate this into a natural isomorphism
\be
\begin{split}
&\bigoplus_{i\in\I}\homc(\unit,T_1MNU_i^\vee A_1)\otik\homc(\unit,T_2U_iA_2)\\
\congto&\bigoplus_{j\in\I}\homc(\unit,T_1MNA_1A_2U_j^\vee)\otik\homc(\unit,T_2U_j)
\label{86}
\end{split}
\ee
Removing the cuts in the markings on $S_4(1,g,g,1;1,1,1,1)$ 
amounts to the isomorphism
\be
\bigoplus_{i\in\I}\homc(\unit,T_1MNU_i^\vee A_1)\otik\homc(\unit,T_2U_iA_2)\congto\homc(\unit,T_1MNA_2T_2A_1)
\ee
which is given by the generalized $\F$-move:

\be
\raisebox{25pt}{
\bp(307,80)
\put(10,0){\rib{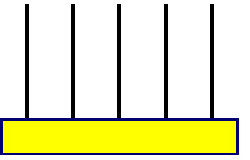}}
\put(110,0){\rib{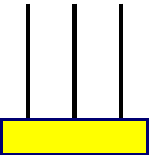}}
\put(210,-20){\rib{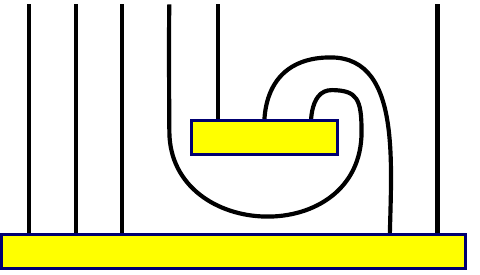}}
  \put(-25,13){$\bigoplus_{i\in\I}$}
  \put(165,13){$\mapsto\bigoplus_{i\in\I}$}  
  \put(89,13){$\otik$}
  \put(27,50){\scriptsize $M$}
  \put(53,50){\scriptsize $U_i^\vee$}
  \put(13,50){\scriptsize $T_1$}
  \put(114,50){\scriptsize $T_2$}
  \put(127,50){\scriptsize $U_i$}
  \put(68,50){\scriptsize $A_1$}
  \put(41,50){\scriptsize $N$}
  \put(141,50){\scriptsize $A_2$}
  \put(227,62){\scriptsize $M$}
  \put(214,62){\scriptsize $T_1$}
  \put(271,62){\scriptsize $T_2$}
  \put(311,45){\scriptsize $U_i$}
  \put(331,62){\scriptsize $A_1$}
  \put(242,62){\scriptsize $N$}
  \put(254,62){\scriptsize $A_2$}
\ep
}
\ee

Then applying the LTG-moves $\B_{A_1,A_2}^{-1}\circ \B_{A_1,T_2}^{-1}$ gives

\be
\raisebox{25pt}{
\bp(175,111)
\put(0,0){\rib{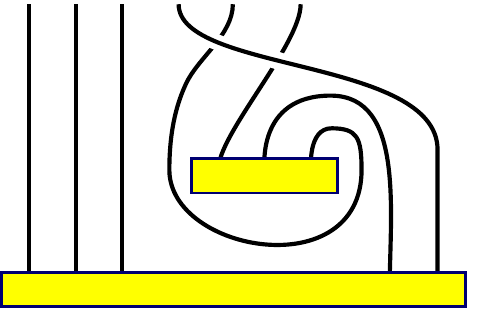}}
\put(-40,40){$\bigoplus_{i\in\I}$}
  \put(17,93){\scriptsize $M$}
  \put(4,93){\scriptsize $T_1$}
  \put(84,93){\scriptsize $T_2$}
  \put(103,15){\scriptsize $U_i$}
  \put(47,93){\scriptsize $A_1$}
  \put(31,93){\scriptsize $N$}
  \put(62,93){\scriptsize $A_2$}
\ep
}
\ee
to which we now apply the adjunction \erf{eqn:Radj}.
After applying duality morphisms,
we obtain the following ribbon graph for the morphism on the right
hand side of (\ref{86}):

\be
\raisebox{25pt}{
\bp(240,140)
\put(20,0){\rib{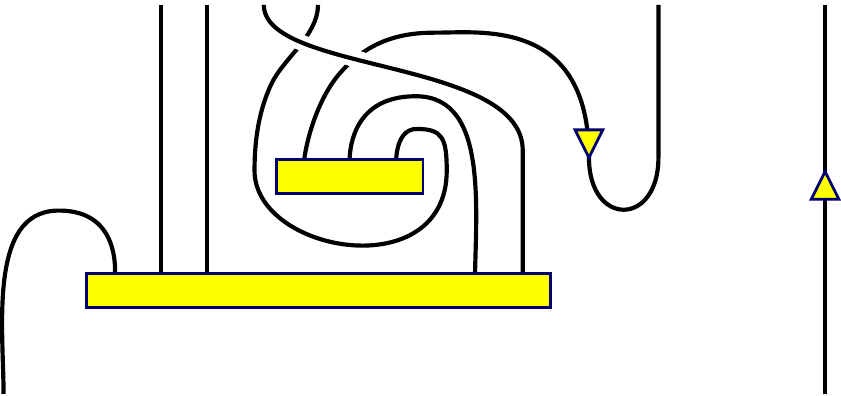}}
\put(-40,60){$\bigoplus_{i,j\in\I}\sum_\alpha$}
\put(230,60){$\otik$}
\put(195,70){\scriptsize $\alpha$}
\put(264,58){\scriptsize $\bar\alpha$}
  \put(63,117){\scriptsize $M$}
  \put(17,-13){\scriptsize $T_1^\vee$}
  \put(256,-13){\scriptsize $T_2^\vee$}
  \put(176,103){\scriptsize $T_2$}
  \put(148,39){\scriptsize $U_i$}
  \put(92,117){\scriptsize $A_1$}
  \put(77,117){\scriptsize $N$}
  \put(108,117){\scriptsize $A_2$}
  \put(255,117){\scriptsize $U_j$}
  \put(206,117){\scriptsize $U_j^\vee$}
\ep
}
\ee

The sum is taken over a basis of $\homc(U_j,T_2^\vee)$ and the corresponding dual basis of $\homc(T_2^\vee,U_j)$. It is obvious that the
morphism does not depend on the choice of basis.
To use again the Yoneda lemma, we evaluate the corresponding natural transformation on 
$T_1^\vee=MNU_i^\vee A_1$ and $T_2^\vee=U_iA_2$
for the identity. We get the following explicit formula for the
associativity constraint:

\be
\raisebox{25pt}{
\bp(60,90)
\put(40,0){\rib{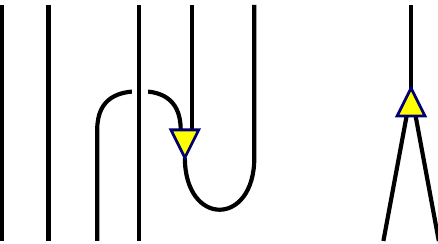}}
\put(-110,30){$\alpha_{M,N,A_1\times A_2}=\bigoplus_{i,j\in\I}\sum_\alpha$}
\put(127,30){$\otik$}
\put(99,27){\scriptsize $\alpha$}
\put(164,38){\scriptsize $\bar\alpha$}
  \put(36,72){\scriptsize $M$}
  \put(36,-9){\scriptsize $M$}
  \put(65,-9){\scriptsize $U_i^\vee$}
  \put(148,-9){\scriptsize $U_i$}
  \put(77,72){\scriptsize $A_1$}
  \put(77,-9){\scriptsize $A_1$}
  \put(50,72){\scriptsize $N$}
  \put(50,-9){\scriptsize $N$}
  \put(90,72){\scriptsize $A_2$}
  \put(165,-9){\scriptsize $A_2$}
  \put(155,72){\scriptsize $U_j$}
  \put(111,72){\scriptsize $U_j^\vee$}
\ep
}
\ee

For the associativity constraint $\alpha_{M,A_1\times A_2,N}$ gluing $S_4(1,g,1,g;1,1,1,1)$ over the four-punctured sphere $S_4$ gives, after removing the
cuts, the graphs

\be
\raisebox{20pt}{
\bp(300,130)
\put(0,0){\pic{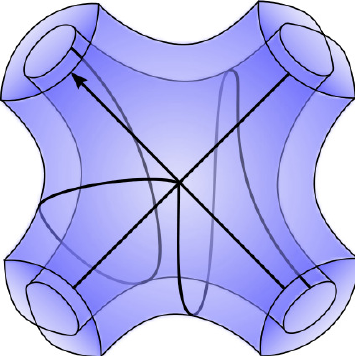}\qquad\qquad\qquad\qquad \pic{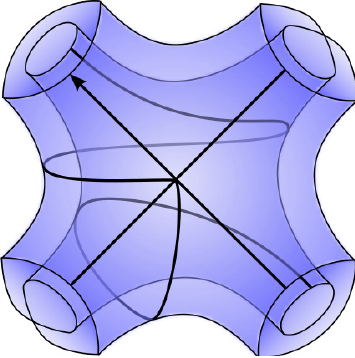}}
\put(95,100){$M$}
\put(-5,-5){$N$}
\put(95,-5){$A_1\times A_2$}
\put(-29,100){$T_1\times T_2$}
\put(195,0){
\put(95,100){$M$}
\put(-5,-5){$N$}
\put(95,-5){$A_1\times A_2$}
\put(-29,100){$T_1\times T_2$}
}
\ep
}
\ee

This results in the following marking on the 6-punctured sphere
$S_6$:

\be
\raisebox{20pt}{
\bp(77,135)
\put(-37,0){\pic{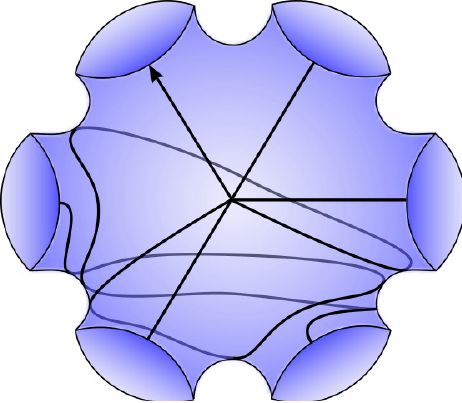}}
  \put(70,110){$M$}
  \put(-22,110){$T_1$}
  \put(-56,53){$T_2$}
  \put(68,-3){$A_2$}
  \put(100,53){$A_1$}
  \put(-22,-3){$N$}
\ep
}
\ee

This marking is transformed into the standard graph on $S_6$ by the following chain of LTG-moves:
\be
\B_{T_2T_1MA_1,A_2}^{-1}\circ\B_{NA_2,T_2}^{-1}\circ\B_{N,T_2}\circ\B_{T_1MA_1,A_2}
\ee
They eventually lead to the associativity isomorphism
\be
\raisebox{25pt}{
\bp(160,90)
\put(40,0){\rib{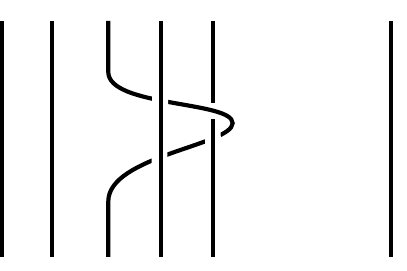}}
\put(-70,35){$\alpha_{M, A_1\times A_2,N}=\bigoplus_{i\in\I}$}
\put(125,35){$\otik$}
  \put(36,77){\scriptsize $M$}
  \put(36,-5){\scriptsize $M$}
  \put(83,77){\scriptsize $N$}
  \put(83,-5){\scriptsize $N$}
  \put(99,-5){\scriptsize $U_i^\vee$}
  \put(151,-5){\scriptsize $U_i$}
  \put(99,77){\scriptsize $U_i^\vee$}
  \put(151,77){\scriptsize $U_i$}
  \put(51,77){\scriptsize $A_1$}
  \put(51,-5){\scriptsize $A_1$}
  \put(66,77){\scriptsize $A_2$}
  \put(66,-5){\scriptsize $A_2$}
\ep
}
\ee

For the associativity constraints $\alpha_{A_1\times A_2,M,N}$ gluing over $S_4$ gives the graphs 

\be
\raisebox{20pt}{
\bp(300,130)
\put(0,0){\pic{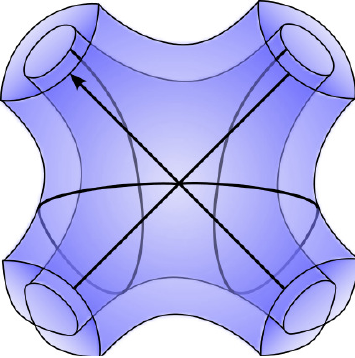}\qquad\qquad\qquad\qquad \pic{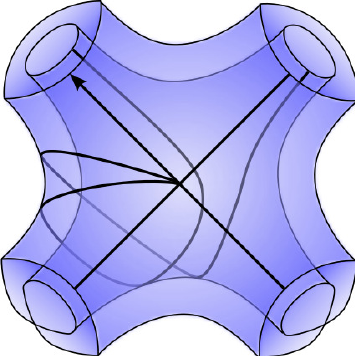}}
\put(95,100){$A_1\times A_2$}
\put(-5,-5){$N$}
\put(95,-5){$M$}
\put(-29,100){$T_1\times T_2$}
\put(195,0){
\put(95,100){$A_1\times A_2$}
\put(-5,-5){$N$}
\put(95,-5){$M$}
\put(-29,100){$T_1\times T_2$}
}
\ep
}
\ee

Transforming $S_4(1,1,g,g;1,1,1,1)$ to $S_6$ gives the following marking on $S_6$:

\be
\raisebox{20pt}{
\bp(77,125)
\put(-37,0){\pic{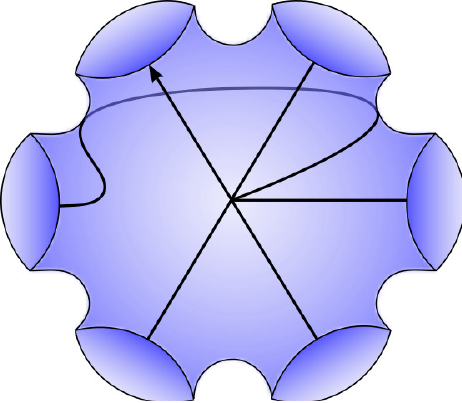}}
  \put(70,110){$A_1$}
  \put(-22,110){$T_1$}
  \put(-56,53){$A_2$}
  \put(68,-3){$N$}
  \put(100,53){$M$}
  \put(-22,-3){$T_2$}
\ep
}
\ee

The transformation into the standard graph on $S_6$ is just given by the LTG-move $\B_{T_1A_1,A_2}$. The procedure outlined before then gives
the constraint
\be
\raisebox{15pt}{
\bp(120,100)
\put(40,4){\rib{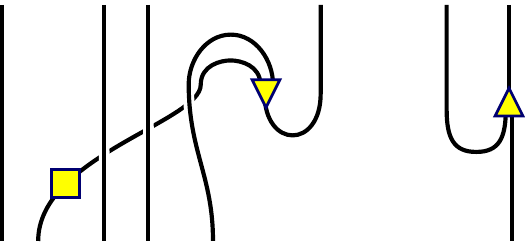}}
\put(-100,35){$\alpha_{A_1\times A_2,M,N}=\bigoplus_{i\in\I}\sum_\alpha$}
\put(145,35){$\otik$}
  \put(66,77){\scriptsize $M$}
  \put(66,-5){\scriptsize $M$}
  \put(80,77){\scriptsize $N$}
  \put(80,-5){\scriptsize $N$}
  \put(97,-5){\scriptsize $U_i^\vee$}
  \put(185,-5){\scriptsize $U_i$}
  \put(129,77){\scriptsize $U_j^\vee$}
  \put(184,77){\scriptsize $U_j$}
  \put(36,77){\scriptsize $A_1$}
  \put(36,-5){\scriptsize $A_1$}
  \put(166,77){\scriptsize $A_2$}
  \put(48,-5){\scriptsize $A_2$}
  \put(122,45){\scriptsize $\alpha$}
  \put(192,42){\scriptsize $\bar\alpha$}
  \put(53,29){\scriptsize $\theta^{-1}$}
\ep
}
\ee
where the $\alpha$-summation is over a basis of $\homc(U_j,A_2^\vee U_i)$ and the corresponding dual basis of $\homc(A_2^\vee U_i,U_j)$.

\subsubsection{Three objects in the twisted component}
The last associativity isomorphism $\alpha_{M,N,O}$ for three objects $M,N,O$ in the twisted component of \CZ is more involved:
the total space $S_4(g,g,g,g;1,1,1,1)$ of the relevant cover
is a surface of genus one. The two gluing procedures over $S_4$ give the following markings:

\be
\raisebox{20pt}{
\bp(280,130)
\put(0,0){\pic{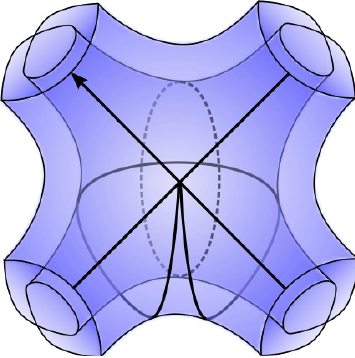}\qquad\qquad\qquad\qquad \pic{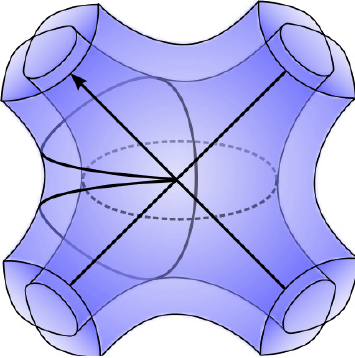}}
\put(95,100){$M$}
\put(-5,-5){$O$}
\put(95,-5){$N$}
\put(0,100){$T$}
\put(195,0){
\put(95,100){$M$}
\put(-5,-5){$O$}
\put(95,-5){$N$}
\put(0,100){$T$}
}
\ep
}
\ee

Since the cover is not a genus zero surface any longer, the
rules of the LTG do not allow us to remove the cuts and to
contract lines of the marking. Hence we keep lines indicating
the cuts which are drawn
with dotted lines. Transforming $S_4(g,g,g,g;1,1,1,1)$ into a 
standard block along the second marking gives the following
surface of genus one with four boundary components:

\be
\begin{split}
\raisebox{20pt}{
\bp(330,150)
\put(0,-40){\raisebox{38pt}{\pic{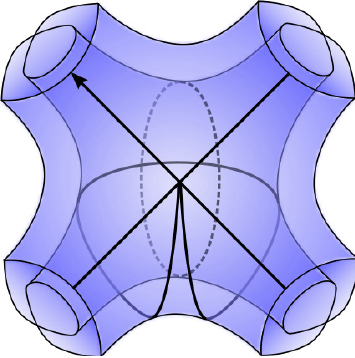}}\qquad\qquad\qquad \pic{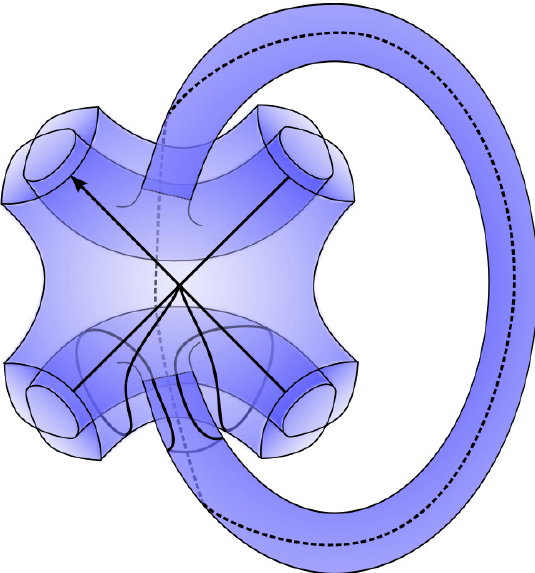}}
\put(95,100){$M$}
\put(-5,-5){$O$}
\put(95,-5){$N$}
\put(0,100){$T$}
\put(278,80){$M$}
\put(165,-5){$O$}
\put(275,-5){$N$}
\put(165,80){$T$}
\put(130,40){$\congto$}
\ep
}\\
\raisebox{30pt}{
\bp(380,150)
\put(0,-20){\pic{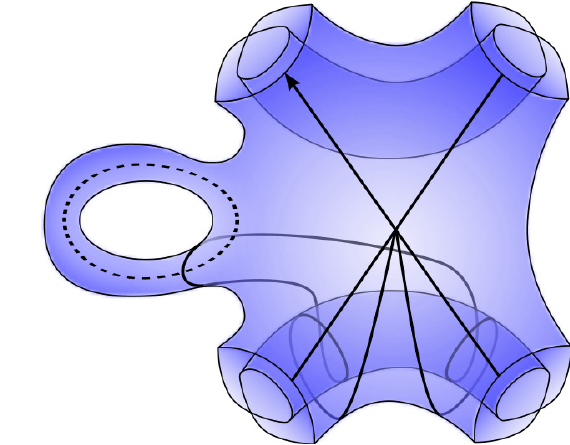}\qquad\qquad \pic{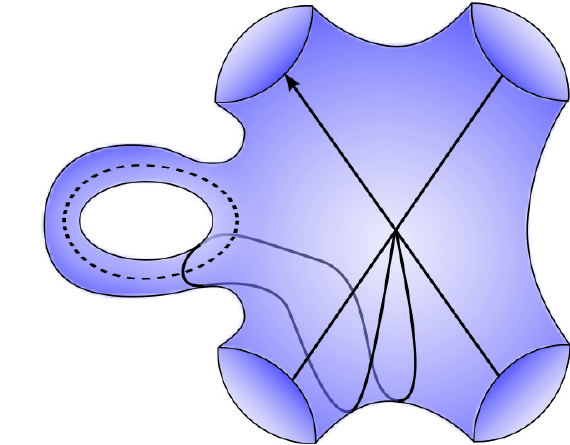}}
\put(-30,40){$\congto$}
\put(165,95){$M$}
\put(50,-15){$O$}
\put(165,-15){$N$}
\put(55,95){$T$}
\put(378,95){$M$}
\put(265,-15){$O$}
\put(378,-15){$N$}
\put(265,95){$T$}
\put(180,40){$\congto$}
\ep
}
\end{split}
\ee

The transformation
\be
\raisebox{30pt}{
\bp(420,140)
\put(0,-20){\pic{gggg-transform-4}\qquad\qquad \pic{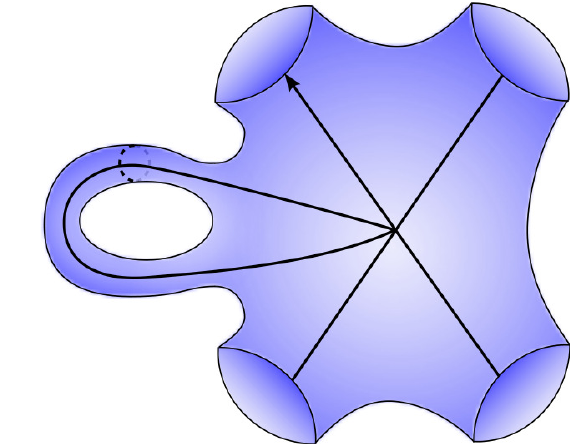}}
\put(165,95){$M$}
\put(50,-15){$O$}
\put(165,-15){$N$}
\put(55,95){$T$}
\put(378,95){$M$}
\put(265,-15){$O$}
\put(378,-15){$N$}
\put(265,95){$T$}
\put(180,40){$\to$}
\ep
}
\ee
of the resulting marking into the standard marking on the four-holed torus is given by the following sequence of LTG-moves:
\be
\S^{-1}\circ\B_{TMN,\R^{(1)}}\circ\B_{O,\R^{(2)}}^{-1}
\ee
Translating these moves into a morphism in the modular
tensor category \C gives
\be
\raisebox{15pt}{
\bp(160,100)
\put(40,7){\rib{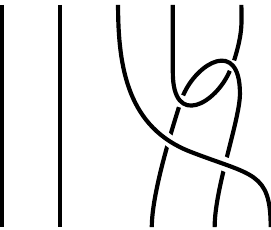}}
\put(-90,35){$\alpha_{M,N,O}=\quad\bigoplus_{i,j\in\I}\frac{d_i}{\catdim}$}
  \put(36,77){\scriptsize $M$}
  \put(36,-5){\scriptsize $M$}
  \put(54,77){\scriptsize $N$}
  \put(54,-5){\scriptsize $N$}
  \put(72,77){\scriptsize $O$}
  \put(115,-5){\scriptsize $O$}
  \put(79,-5){\scriptsize $U_j^\vee$}
  \put(99,-5){\scriptsize $U_j$}
  \put(86,77){\scriptsize $U_i^\vee$}
  \put(106,77){\scriptsize $U_i$}
\ep
}
\label{eq:103}
\ee
with $\catdim=\sqrt{p^+p^-}$ and $p^\pm$ defined as in section \ref{ssect:Gmodcov}.
(Note that the morphism corresponding to the $\S$-move in \cite{BK}
is defined by a direct sum of morphisms $U_l U_l^\vee
\to  U_k U_k^\vee$; we have inserted two pairs of 
isomorphisms identifying for a simple object 
$U_k\cong U_{k^\ast}^\vee$  which cancel pairwise.)

These associativity constraints have to satisfy mixed pentagon 
axioms for any choice of four objects in the two sectors 
of \CZ, yielding in total 16 different types of pentagon
diagrams.
Theorem \ref{thm:GMF} asserts that our construction yields a $G$-equivariant
functor; the general results of \cite[Section 7.2]{kpGMF} then ensure that
all associativity constraints obtained in this section 
satisfy the pentagon axiom.
We have checked this by hand as well; 
only the pentagon with four objects in the twisted component is more involved.

\subsection{Tensoriality of the \ZZ-action}

We next derive the isomorphisms 
$\phi_{A,B}:\,\,  ^h(A\otimes B)\to {}^h\!A\otimes {}^h\!B$ that turn
the equivariance functor $R_h$ into a tensor functor.
They have been described in general before 
lemma \ref{prop:tensorf}. The functor $R_1$ is the 
identity functor and the tensoriality constraints $\phi_{A,B}^1$ 
are identity morphisms. The non-trivial element $g\in\ZZ$
acts by permutation of factors on the untwisted component
$\C\boxtimes\C$ and as the identity on the twisted component
$\C$. Hence we only compute the tensoriality constraint
$\phi_{A,B}^g$. 

Before we proceed, we will have a look at the $S_2$ cover of the form 
$(S_2(g,g;1,g)\to S_2)$, together with its marking. As a smooth manifold the total 
space $S_2(g,g;1,g)$ is diffeomorphic to $S_2(g,g;1,1)$ 
by a half turn around the second hole. However this is not a map of marked covers over $S_2$. We choose a half turn diffeomorphism
(say by rotating the second circle clockwise by $\pi$)
to identify the total spaces $S_2(g,g;1,g)$ and 
$S_2(g,g;1,1)$. Both spaces are diffeomorphic to the
two-punctured sphere $S_2$ and the corresponding
diffeomorphism of $S_2$ induces a natural isomorphism 
$\homc(\unit,UV)\cong\homc(\unit,U{}^gV)$. Using the duality,
we get an isomorphism $\sigma_V:V\to{}^gV=V$. We use this isomorphism to identify the respective hom-spaces.
On marking graphs this introduces an additional move which we call the $\sigma$-move,

\be
\raisebox{20pt}{
\bp(250,60)
\put(0,0){\pic{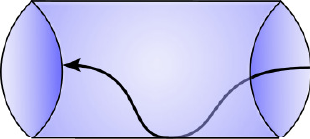}\qquad\qquad \pic{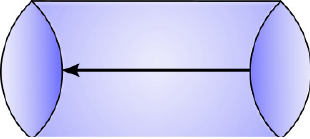}}
\put(-12,35){$A$}
\put(90,35){$B$}
\put(105,15){$\stackrel{\sigma_B}{\longrightarrow}$}
\put(125,35){$A$}
\put(225,35){$B$}
\ep
}
\ee
on the two-punctured sphere $S_2$ and
\be
\raisebox{20pt}{
\bp(250,60)
\put(0,0){\pic{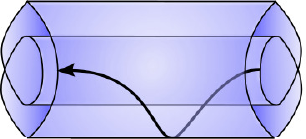}\qquad\qquad \pic{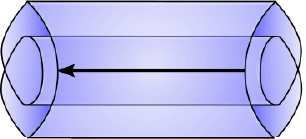}}
\put(-12,35){$A$}
\put(90,35){$B$}
\put(105,15){$\stackrel{\sigma_B}{\longrightarrow}$}
\put(125,35){$A$}
\put(225,35){$B$}
\ep
}
\label{105}
\ee
on the total space $S_2(g,g;1,g)$ of the cover.
Applying the half-turn diffeomorphism twice is the Dehn-twist, 
hence we have the identity $\sigma_V^2=\theta_V \in\hom_\C(V,V)$. 

From now on we will draw the standard spheres $S_n$ as 
the one-point compactification
$\overline{\mathbb{C}}$ of the complex plane with $n$ discs 
of radius $1/3$ centered at $0,1,2,\ldots n-1$
removed. As in subsection \ref{sssect:covcirc}, the standard blocks 
are obtained by identifying the trivial cover of the cut sphere 
along the cuts, i.e.\
as $(S_n\!\setminus\!\mbox{cuts}\times G)/\!\sim$. For instance the marking graphs on covers of $S_3$ that represent the tensor products are shown in figure \ref{fig:tensorp}. 
\begin{figure}[tb]
\raisebox{20pt}{
\bp(400,146)
\put(0,0){\pic{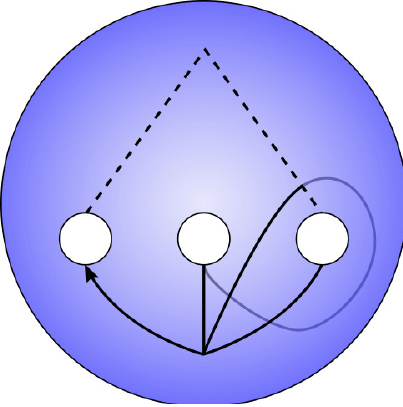}}
\put(20,-15){$(A_1\times A_2)\otimes M$}
\put(140,0){\pic{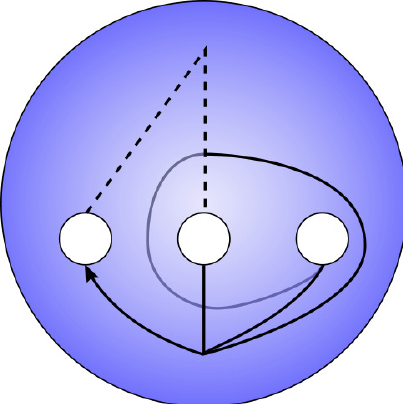}}
\put(160,-15){$M\otimes (A_1\times A_2)$}
\put(280,0){\pic{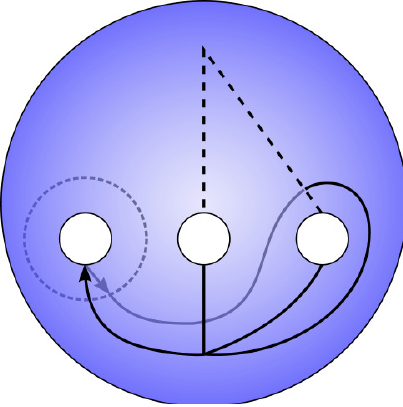}}
\put(322,-15){$M\otimes N$}
\ep
}
\caption{The marking graphs on $S_3$-covers that represent the non-trivial tensor products.
We show three Riemann spheres $\overline{\mathbb C}$ from above.
The arrow points to the
disc where the test object $T$ is inserted. All covers are twofold; their
total space has the topology of a four-holed sphere.
The dashed line is a branch cut linking the two insertions with non-trivial monodromy
and indicates a self-intersection in the immersion of the total space of
the cover into three-dimensional space. 
 }\label{fig:tensorp}
\end{figure}
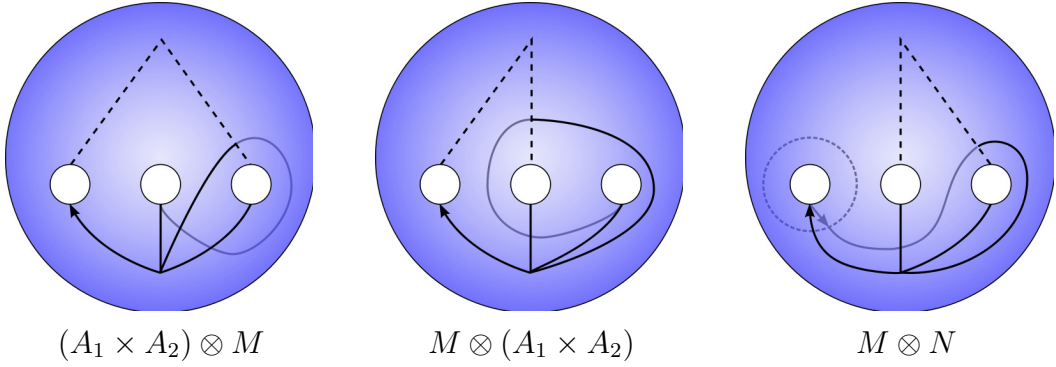

We are now ready to compute the tensoriality constraint
$\phi_{A,B}^g$. We apply the $\sigma$-move  introduced in
picture (\ref{105}) to the relevant cover of $S_2$ obtain the marking graph 

\be
\raisebox{20pt}{
\bp(150,140)
\put(0,0){\pic{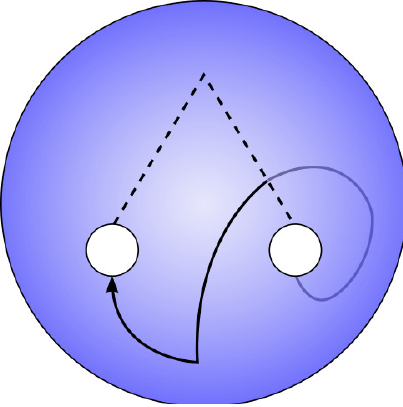}}
\put(28,42){\scriptsize $A$}
\put(81,42){\scriptsize $B$}
\ep
}
\ee 
on $S_2(g,g;1,g)$.
Then we apply the sequence of transformations in \erf{eq:tensorf}
on $S_2(g,g;1,g)$ and translate them into morphisms, which
have to be applied after the morphism $\sigma$.

We outline our procedure to determine the tensoriality 
constraint $\phi_{A,B}^g$:

\begin{enumerate}
\item Determine the cover of the two-punctured sphere
that is appropriate for the pair of objects $(A,B)$.

\item If the tensor product takes its value in the twisted component,
$A\otimes B\in \CZ_g$, use the half-twist $\sigma$ to 
identify the covers $S_2(g,g;1,1)$ and $S_2(g,g;1,g)$.
\item Apply the diffeomorphism $\tilde g$ from equation (\ref{18} )
to the cover.
\item Now we can glue in the cover of the three-punctured sphere $S_3$ 
that implements the tensor product $A\otimes B$.
\item Apply the diffeomorphism $\widetilde{g^{-1}}=\tilde g$
again.
\item Transform the resulting cover of $S_3$ back to a standard block using the marking from the tensor product ${}^gA\otimes{}^gB$.
\item Read off the LTG-moves that transform the resulting marking into the standard marking. If $A\otimes B$ is in the untwisted
component, this is the tensoriality constraint. Otherwise, if
$A\otimes B\in \CZ_g$, the tensoriality constraint is obtained by
first applying  the morphism $\sigma_{A\otimes B}$ to 
${}^g(A\otimes B)$ to take into account step 2, and then the LTG-moves.
\end{enumerate}

The choice of the diffeomorphism $\sigma$ to identify 
the total spaces $S_2(g,g;1,1)$ and $S_2(g,g;1,g)$ is  
non-canonical.  Since the mapping class group of the cylinder is generated by the Dehn-twist, different choices of identifications differ by powers of the Dehn-twist, which in our conventions is the square
of $\sigma$. A different choice of the identification
of $S_2(g,g;1,1)$ and $S_2(g,g;1,g)$ gives different
markings on the cover of $S_3$ in step 6 of our procedure.
As a consequence, the LTG-moves will differ by powers of the 
twist on ${}^g(A\otimes B)$. The difference then
cancels in step 7, so that the tensoriality constraint is
independent of the choice of identifications of
$S_2(g,g;1,1)$\ and $S_2(g,g;1,g)$. 

We will start by deriving the tensoriality constraint
$\phi_{A_1\times A_2,B_1\times B_2}^g$
for $g$ on the tensor product of two objects in the untwisted
component. Our procedure gives the standard marking

\be
\raisebox{20pt}{
\bp(150,140)
\put(0,0){\pic{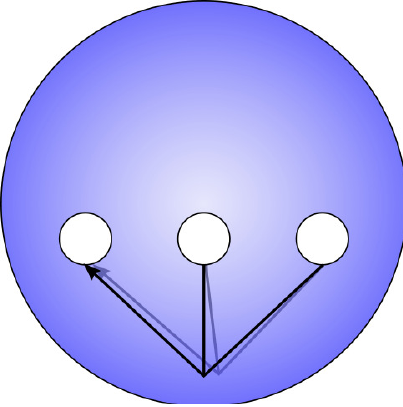}}
\put(5,60){\scriptsize $A_1\times A_2$}
\put(42,60){\scriptsize $B_1\times B_2$}
\put(80,60){\scriptsize $C_1\times C_2$}
\ep
}
\ee
hence we deduce that the constraint is the identity,
$\phi_{A_1^\times A_2,B_1\times B_2}^g=\id_{A_2B_2\times A_1B_1}$.

For the constraint $\phi_{M,A_1\times A_2}^g$, we first have
to use $\sigma$  to identify $S_2(g,g;1,1)$ and 
$S_2(g,g;1,g)$; then we apply the sequence of operations described
in \erf{eq:tensorf}:

\be
\begin{split}
\raisebox{20pt}{
\bp(400,146)
\put(0,0){\pic{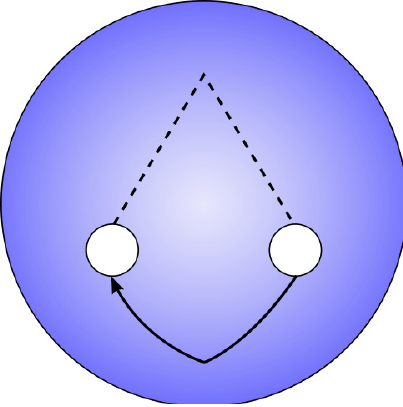}}
\put(140,0){\pic{S2-circle-sigma}}
\put(280,0){\pic{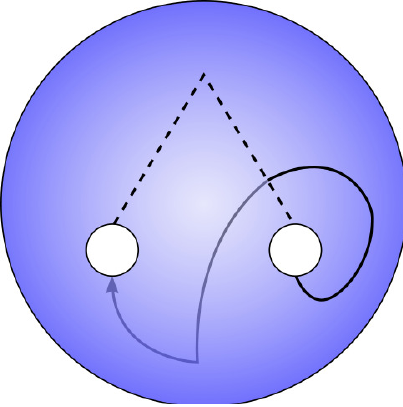}}
\put(122,50){$\stackrel{\sigma}{\to}$}
\put(262,50){$\stackrel{(\tilde g)_\ast}{\to}$}
\put(29,43){\scriptsize $T$}
\put(82,43){\scriptsize $MA_1A_2$}
\put(169,43){\scriptsize $T$}
\put(222,43){\scriptsize $MA_1A_2$}
\put(308,43){\scriptsize $T$}
\put(362,43){\scriptsize $MA_1A_2$}
\ep
}\\
\raisebox{20pt}{
\bp(400,146)
\put(0,0){\pic{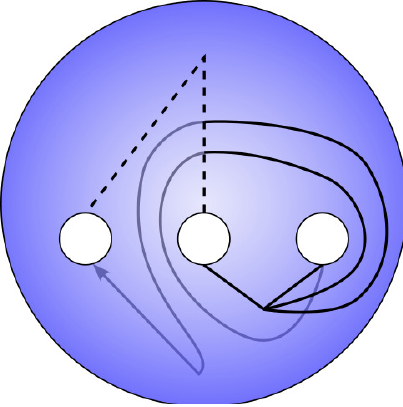}}
\put(140,0){\pic{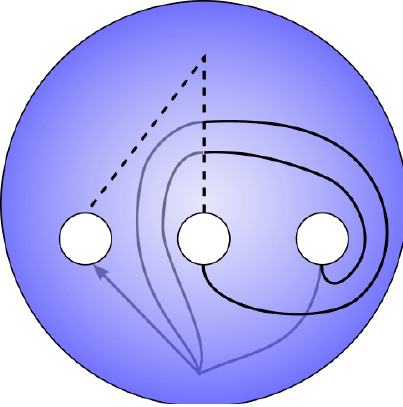}}
\put(280,0){\pic{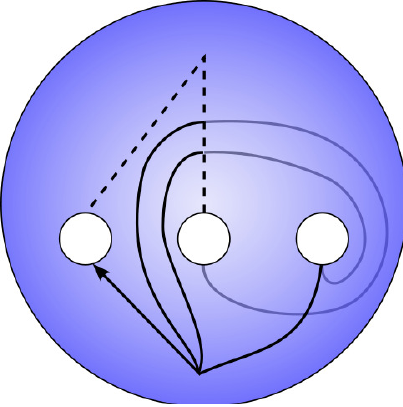}}
\put(-22,50){$\to$}
\put(122,50){$\to$}
\put(262,50){$\stackrel{(\tilde g)_\ast}{\to}$}
\put(22,45){\scriptsize $T$}
\put(55,45){\scriptsize $M$}
\put(73,45){\scriptsize $A_1\times A_2$}
\put(162,45){\scriptsize $T$}
\put(194,45){\scriptsize $M$}
\put(213,45){\scriptsize $A_1\times A_2$}
\put(301,45){\scriptsize $T$}
\put(333,45){\scriptsize $M$}
\put(353,45){\scriptsize $A_2\times A_1$}
\ep
}\\
\raisebox{20pt}{
\bp(400,146)
\put(0,0){\pic{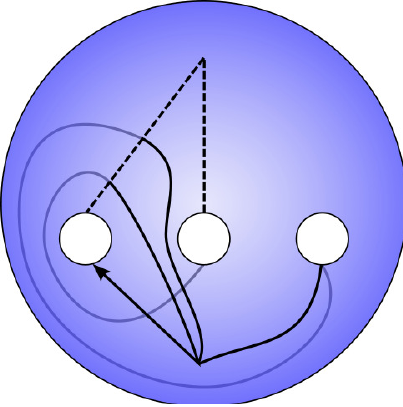}}
\put(-22,50){$\to$}
\put(21,45){\scriptsize $T$}
\put(55,45){\scriptsize $M$}
\put(77,45){\scriptsize $A_2\times A_1$}
\ep
}
\end{split}
\ee
In the first step we apply the half-twist $\sigma_{MA_1A_2}$, in the second step the application of $\tilde g$ as defined in
(\ref{18}) exchanges the two sheets
of the cover. In the third step we glue in the three-puncture
sphere $S_3(g,g,1;1,1,1)$ with the marking representing the 
tensor product $M\otimes (A_1\times A_2)$, see figure 1. In the fourth step we perform the gluing on the marking graph. In the fifth step we apply $\tilde g$ again, 
which in particular exchanges the holes labelled by $A_1$ and $A_2$. 
The last step is merely a simplification of the 
graph: we move the lines around the back side of the sphere. 

We use the marking on $S_3(g,g,1;1,1,1)$ in figure 1 that represents the tensor 
product $M\otimes (A_2\times A_1)$ to get an isomorphism to $S_4$. This
marking instructs us to move the disc labelled by $A_1$ around the
disc labelled by $M$. This yields the left figure in the following line:
\be
\raisebox{20pt}{
\bp(400,146)
\put(0,0){\pic{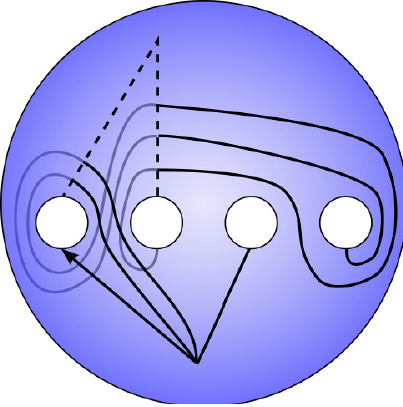}}
\put(140,0){\pic{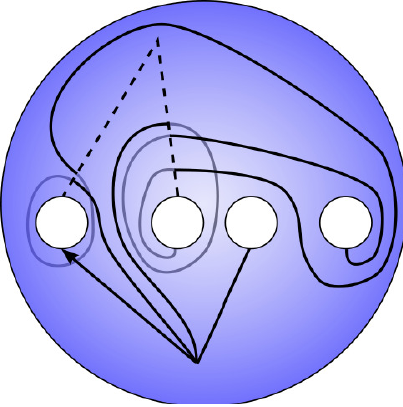}}
\put(280,0){\pic{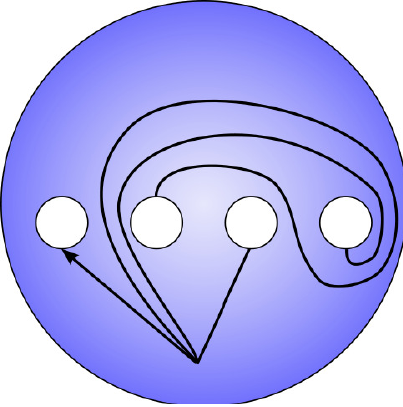}}
\put(15,50){\scriptsize $T$}
\put(41,50){\scriptsize $M$}
\put(67,50){\scriptsize $A_2$}
\put(95,50){\scriptsize $A_1$}
\put(140,0){
\put(15,50){\scriptsize $T$}
\put(46,50){\scriptsize $M$}
\put(67,50){\scriptsize $A_2$}
\put(95,50){\scriptsize $A_1$}
}
\put(280,0){
\put(14,50){\scriptsize $T$}
\put(40,50){\scriptsize $M$}
\put(67,50){\scriptsize $A_2$}
\put(95,50){\scriptsize $A_1$}
}
\ep
}
\label{eq:109}
\ee
In the second picture we redraw the marking graph in a more convenient shape. 
The last picture is obtained by a diffeomorphism of non-embedded manifolds.

The marking graph in the third picture of (\ref{eq:109})
is transformed into the standard marking graph on $S_4$ 
by the following sequence of LTG-moves:
\be
\theta_{A_1}^{-1}\circ\sigma_M\circ \B_{MA_2,A_1}^{-1}\circ\B_{A_2,M}\circ\B_{T,M}
\ee
where $\theta$ is an abbreviation for the Dehn twist move 
$\Z\circ\B^{-1}$. 
When we translate the LTG moves into morphisms, several
Dehn twists
occur in manipulations of the ribbon graphs.
They have to
be combined with the morphism $\sigma$ which 
squares to the twist. This leads to powers of $\sigma$ that
differ from the naive expectations, and 
we arrive at the tensoriality constraint
\be
\raisebox{20pt}{
\bp(110,110)
\put(40,0){\rib{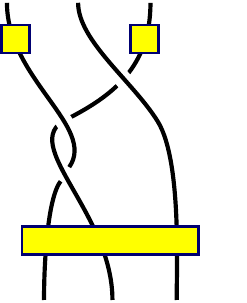}}
\put(-40,40){$\phi^g_{M,A_1\times A_2}=$}
\put(48,-9){\scriptsize $M$}
\put(69,-9){\scriptsize $A_1$}
\put(88,-9){\scriptsize $A_2$}
\put(38,90){\scriptsize $M$}
\put(59,90){\scriptsize $A_2$}
\put(79,90){\scriptsize $A_1$}
\put(36,16){\scriptsize $\sigma$}
\put(24,75){\scriptsize $\sigma^{-1}$}
\put(90,75){\scriptsize $\theta^{-1}$}
\ep
}
\ee
where we have included $\sigma_{MA_1A_2}$ according to our
general prescription, step 7.

For tensoriality constraint $\phi_{A_1\times A_2,M}^g$, 
we proceed in a similar way. Again we identify the
total spaces of the covers $S_2(g,g;1,1)$ and $S_2(g,g;1,g)$ 
with the diffeomorphism $\sigma$ and apply \erf{eq:tensorf}. This gives:

\be
\raisebox{20pt}{
\bp(400,136)
\put(30,0){\pic{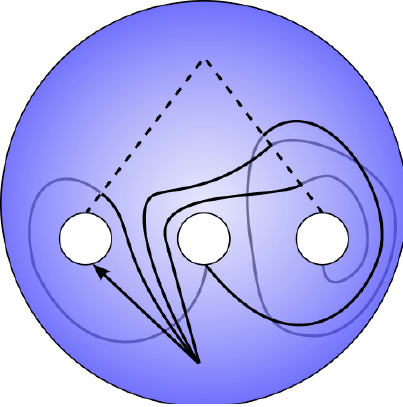}}
\put(250,0){\pic{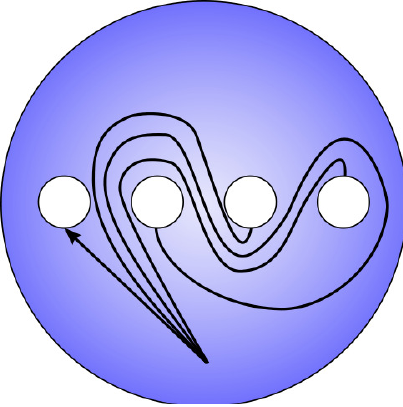}}
\put(51,45){\scriptsize $T$}
\put(119,45){\scriptsize $M$}
\put(79,45){\scriptsize $A_2\times A_1$}
\put(265,56){\scriptsize $T$}
\put(345,56){\scriptsize $M$}
\put(318,56){\scriptsize $A_1$}
\put(290,56){\scriptsize $A_2$}
\ep
}
\ee
The first picture is the result of \erf{eq:tensorf}. The second picture is the result of transforming $S_3(g,1,g;1,1,1)$ to $S_4$ using the marking on $S_3(g,1,g;1,1,1)$ that represents the tensor product ${}^g(A_1\times A_2)\otimes {}^gM=(A_2\times A_1)\otimes M$.

The resulting marking graph is transformed into the standard graph by the sequence
\be
\B_{A_2,A_1}^{-1}\circ\B_{A_2,M}^{-1}\circ\theta_{A_2}^{-1}\circ\B_{M,A_2}^{-1}\circ\sigma_M^{-1}
\ee
of LTG-moves. This gives the morphism
\be
\raisebox{25pt}{
\bp(160,110)
\put(40,0){\rib{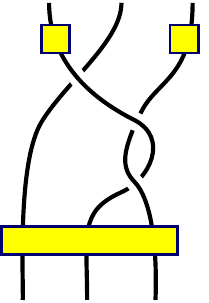}}
\put(-40,45){$\phi^g_{A_1\times A_2,M}=$}
\put(82,-9){\scriptsize $M$}
\put(43,-9){\scriptsize $A_1$}
\put(61,-9){\scriptsize $A_2$}
\put(91,90){\scriptsize $M$}
\put(71,90){\scriptsize $A_1$}
\put(49,90){\scriptsize $A_2$}
\put(34,16){\scriptsize $\sigma$}
\put(101,75){\scriptsize $\sigma^{-1}$}
\put(35,75){\scriptsize $\theta^{-1}$}
\ep
}
\ee

The last tensoriality constraint to be determined is $\phi_{M,N}^g$. 
We apply \erf{eq:tensorf} to $S_2(1,1;1,1)$ and get
\be
\raisebox{20pt}{
\bp(400,146)
\put(30,0){\pic{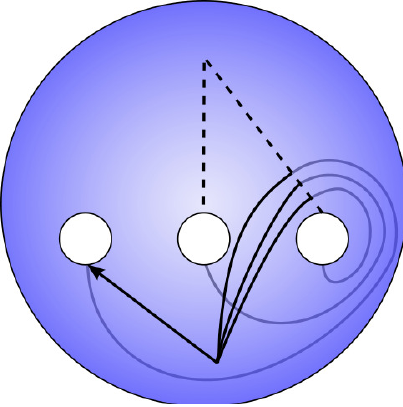}}
\put(250,0){\pic{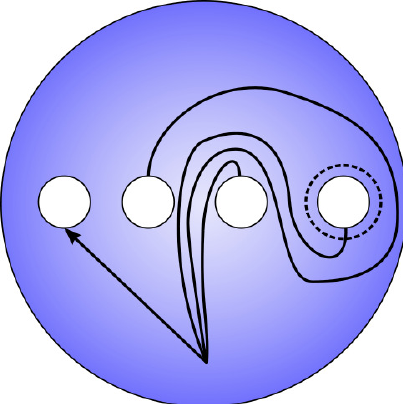}}
\put(39,45){\scriptsize $T_1\times T_2$}
\put(119,45){\scriptsize $N$}
\put(84,45){\scriptsize $M$}
\put(265,56){\scriptsize $T_1$}
\put(345,56){\scriptsize $T_2$}
\put(316,56){\scriptsize $N$}
\put(288,56){\scriptsize $M$}
\ep
}
\ee
The first picture is again the result of \erf{eq:tensorf}, the second picture is the marking graph we obtain on $S_4$. Now this marking is transformed into the standard marking by the LTG-moves
\be
\B_{NT_2,M}\circ\B_{N,T_2}^{-1}\circ\B_{N,M}^{-1}\circ\sigma_N^{-1}\circ\sigma_M
\ee
Recall that ${}^g(M\otimes N)=\bigoplus U_i\times MNU_i^\vee$. 
Hence, we have to apply the corresponding transformations to $\hom_\C(\unit, T_1U_i)\otik\hom_\C(\unit,T_2MNU_i^\vee)\cong\hom_\C(\unit,T_1T_2MN)$ and finally use the adjunction \erf{eqn:Radj} to obtain
\be
\raisebox{20pt}{
\bp(160,90)
\put(40,0){\rib{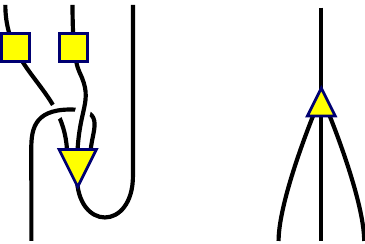}}
\put(-80,35){$\phi^g_{M,N}=\quad\bigoplus_{i,j\in\I}\sum_\alpha$}
\put(97,35){$\otik$}
\put(38,73){\scriptsize $M$}
\put(45,-9){\scriptsize $U_i$}
\put(142,-9){\scriptsize $U_i^\vee$}
\put(116,-9){\scriptsize $M$}
\put(129,-9){\scriptsize $N$}
\put(58,73){\scriptsize $N$}
\put(75,73){\scriptsize $U_j^\vee$}
\put(129,73){\scriptsize $U_j$}
\put(70,20){\scriptsize $\alpha$}
\put(138,38){\scriptsize $\bar\alpha$}
\put(68,55){\scriptsize $\sigma$}
\put(25,55){\scriptsize $\sigma^{-1}$}
\ep
}
\ee
where the summation over $\alpha$ runs over a basis of $\homc(U_j,MNU_i^\vee)$ and the corresponding dual basis. 
Again, we had to take into account Dehn twists which shifted
the powers of $\sigma$.

We have checked directly that all morphisms derived 
indeed satisfy all identities needed to endow 
the functor $R_g$ with the structure of a tensor functor.

\subsection{The braiding}

We finally derive the braiding on the $\ZZ$-equivariant 
category $\CX$
which consists of isomorphisms 
$C_{U,V}:U\otimes V\congto {}^pV\otimes U$ with $U\in\CZ_p$ and $V\in\CZ_q$. To this end, we lift the braiding diffeomorphism
\be
\raisebox{20pt}{
\bp(146,146)
\put(0,0){\pic{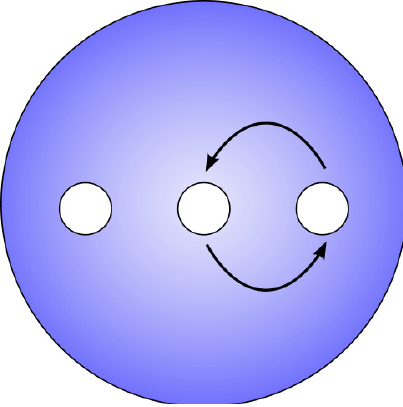}}
\ep
}
\ee
of the three-holed sphere $S_3$ to appropriate covers and obtain
a diffeomorphism
\be
\tilde\phi_\B:S_3((pq)^{-1},p,q;1,1,1)\congto S_3((pq)^{-1},pqp^{-1},p;1,p^{-1},1) \,\, .
\label{eq:119}
\ee
It induces for any object $T\in\CZ$ a natural isomorphism
\be
\begin{split}
&\la T,U\otimes V\rag\defn\la T,U,V\rag\defn\tauzz(S_3((pq)^{-1},p,q;1,1,1);T,U,V)\stackrel{(\tilde\phi_\B)_\ast}{\to}\\
&\tauzz(S_3((pq)^{-1},pqp^{-1},p;1,p^{-1},1);T,V,U)=\\
&\tauzz(S_3((pq)^{-1},pqp^{-1},p;1,1,1);T,{}^pV,U)\defn
\la T,{}^pV,U\rag\defn\la T,{}^pV\otimes U\rag.
\end{split}
\ee
Thus the procedure to determine braidings is analogous to
the one to determine the associativity constraints:
\begin{enumerate}
\item Start with the standard marking on the cover $S_3((pq)^{-1},p,q;1,1,1)$
that represents the tensor product $U\otimes V$.
\item Apply the diffeomorphism $\tilde\phi_\B$ defined in 
(\ref{eq:119}).
\item The result is the cover $S_3((pq)^{-1},pqp^{-1},p;1,p^{-1},1)$ 
and has to be transformed into a standard block, using the 
marking representing the tensor product ${}^pV\otimes U$.
\item Next use the LTG-moves to transform the resulting 
marking graph on the standard block into the standard marking 
graph.  
\item Finally translate the LTG-moves into morphisms in \C 
or \CC, respectively.
\end{enumerate}

Not surprisingly, the braiding of two objects $A_1\times A_2$ and $B_1\times B_2$ in the neutral component of \CZ turns out to be the braiding on \CC. To see this, we lift the braiding diffeomorphism $\phi_\B$ to $S_3(1,1,1;1,1,1)$.
As a smooth manifold, this is just isomorphic to the disjoint
union $S_3\sqcup S_3$ of two three-holed spheres.
The lift $\tilde\phi_\B$ of the braiding isomorphism 
is just $\phi_\B$ applied to both components. Hence
\be
C_{A_1\times A_2,B_1\times B_2}=c_{A_1,B_1}\otik c_{A_2,B_2}
\,\,\, .
\ee
We now discuss the more complicated situations. For the braiding $C_{A_1\times A_2,M}:(A_1\times A_2)\otimes M\to M\otimes (A_1\times A_2)$, we
have to consider the cover $S_3(g,1,g;1,1,1)$ of $S_3$ with its standard marking graph. Lifting the braiding $\phi_\B$ gives the marking

\be
\raisebox{20pt}{
\bp(400,126)
\put(30,0){\pic{g1g-circle}}
\put(240,0){\pic{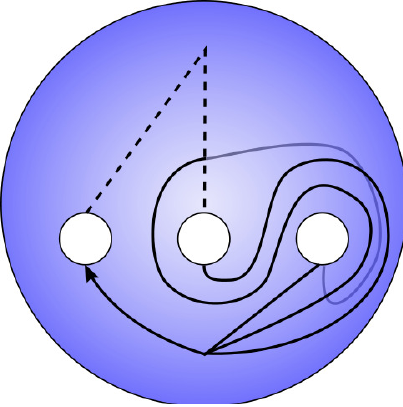}}
\put(185,50){$\stackrel{\tilde \phi_\B}{\rightarrow}$}
\put(51,45){\scriptsize $T$}
\put(119,45){\scriptsize $M$}
\put(70,45){\scriptsize $A_1\times A_2$}
\put(261,45){\scriptsize $T$}
\put(294,45){\scriptsize $M$}
\put(318,45){\scriptsize $A_1\times A_2$}
\ep
}
\ee
on $S_3(g,g,1;1,1,1)$. Applying the diffeomorphism to $S_4$
that represents
the tensor product 
${}^1M\otimes (A_1\times A_2)=M\otimes (A_1\times A_2)$  gives the marking

\be
\raisebox{20pt}{
\bp(146,146)
\put(0,0){\pic{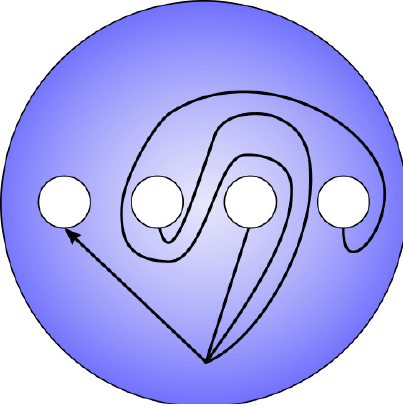}}
\put(67,56){\scriptsize $A_1$}
\put(15,56){\scriptsize $T$}
\put(40,56){\scriptsize $M$}
\put(94,56){\scriptsize $A_2$}
\ep
}
\ee 
on $S_4$. It is connected to the standard marking of the
four-punctured sphere $S_4$ by the following sequence of LTG-moves:
\be
\B_{A_1,M}\circ\B_{T,A_2}\circ\B_{A_1,A_2}
\ee
Applied to $\homc(\unit,TA_1A_2M)$, this induces the 
following braiding isomorphism on $\CX$
\be
\raisebox{25pt}{
\bp(160,90)
\put(40,0){\rib{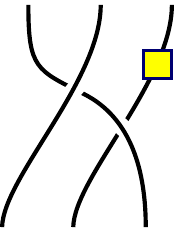}}
\put(-40,35){$C_{A_1\times A_2,M}=$}
\put(44,69){\scriptsize $M$}
\put(65,69){\scriptsize $A_1$}
\put(86,69){\scriptsize $A_2$}
\put(36,-9){\scriptsize $A_1$}
\put(57,-9){\scriptsize $A_2$}
\put(79,-9){\scriptsize $M$}
\put(94,45){\scriptsize $\theta^{-1}$}
\ep
}
\ee

For the braiding $C_{M,A_1\times A_2}:M\otimes (A_1\times A_2)\to {}^g(A_1\times A_2)\otimes M=(A_2\times A_1)\otimes M$ we consider the cover $S_3(g,g,1;1,1,1)$ of $S_3$ with its standard marking graph. We lift the braiding  $\phi_\B$ and get the marking

\be
\raisebox{20pt}{
\bp(146,146)
\put(0,0){\pic{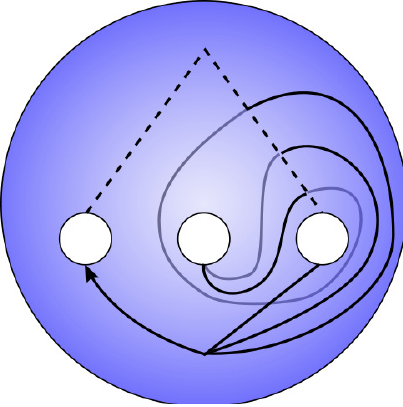}}
\put(21,45){\scriptsize $T$}
\put(88,45){\scriptsize $M$}
\put(40,45){\scriptsize $A_2\times A_1$}
\ep
}
\ee
on $S_3(g,1,g;1,g,1)$. 
Note that the insertions of $A_1$ and $A_2$ have exchanged the sheet, when they were moved pass the self-intersection.
We apply the diffeomorphism to $S_4$ given by the marking on $S_3(g,1,g;1,g,1)$ that represents the tensor product ${}^g(A_1\times A_2)\otimes M=(A_2\times A_1)\otimes M$. The result is the following
marking on $S_4$:

\be
\raisebox{20pt}{
\bp(146,146)
\put(0,0){\pic{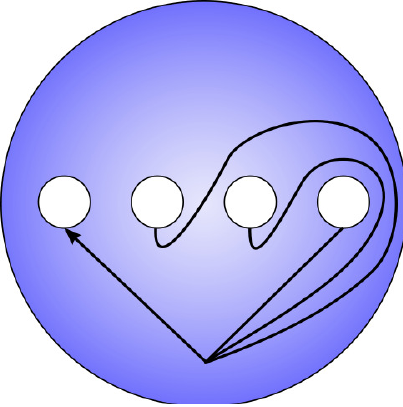}}
\put(67,56){\scriptsize $A_1$}
\put(15,56){\scriptsize $T$}
\put(40,56){\scriptsize $A_2$}
\put(94,56){\scriptsize $M$}
\ep
}
\ee

This marking on $S_4$ is transformed into the standard marking on $S_4$ by the LTG-moves
\be
\B_{A_1M,A_2}\circ\B_{M,A_1}.
\ee
We get the braiding isomorphism on $\CX$:
\be
\raisebox{25pt}{
\bp(160,90)
\put(40,0){\rib{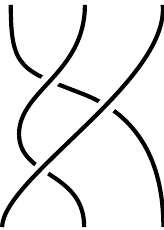}}
\put(-40,35){$C_{M,A_1\times A_2}=$}
\put(39,69){\scriptsize $A_2$}
\put(60,69){\scriptsize $A_1$}
\put(83,69){\scriptsize $M$}
\put(36,-9){\scriptsize $M$}
\put(60,-9){\scriptsize $A_1$}
\put(83,-9){\scriptsize $A_2$}
\ep
}
\ee
Finally we describe the braiding isomorphism $C_{M,N}:M\otimes N\to {}^gN\otimes M=N\otimes M$. We lift the braiding $\phi_\B$ to $S_3(1,g,g;1,1,1)$. The standard marking on $S_3(1,g,g;1,1,1)$ involves a cut: 
we first remove this cut 
and then apply $\tilde\phi_\B$:

\be
\raisebox{20pt}{
\bp(400,146)
\put(30,0){\pic{1gg-circle}}
\put(240,0){\pic{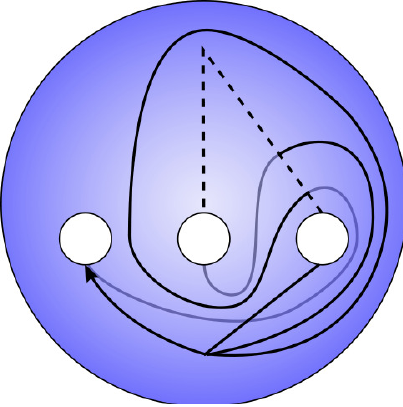}}
\put(185,50){$\stackrel{\tilde \phi_\B}{\rightarrow}$}
\put(40,45){\scriptsize $T_1\times T_2$}
\put(84,45){\scriptsize $M$}
\put(119,45){\scriptsize $N$}
\put(249,45){\scriptsize $T_1\times T_2$}
\put(328,45){\scriptsize $M$}
\put(294,45){\scriptsize $N$}
\ep
}
\ee
We transform the resulting manifold into $S_4$ and arrive at

\be
\raisebox{20pt}{
\bp(146,146)
\put(0,0){\pic{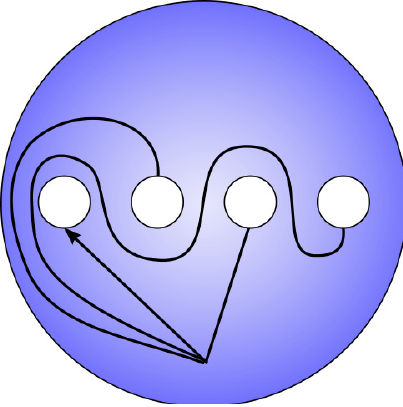}}
\put(67,56){\scriptsize $M$}
\put(15,56){\scriptsize $T_1$}
\put(41,56){\scriptsize $N$}
\put(94,56){\scriptsize $T_2$}
\ep
}
\ee 
This marking is transformed into the standard marking on $S_4$ by the following sequence of moves:
\be
\B_{T_1,N}^{-1}\circ\B_{M,T_2}^{-1}\circ\B_{T_1,T_2}^{-1}\circ\sigma_N^{-1}
\ee
When we apply this to $\homc(\unit,T_1MNT_2)$ and perform the gluing 
process in the adjunction (\ref{55}), we arrive at
\be
\raisebox{25pt}{
\bp(160,90)
\put(40,0){\rib{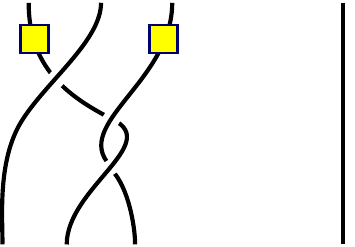}}
\put(-50,30){$C_{M,N}=\bigoplus_{i\in\I}$}
\put(105,30){$\otik$}
\put(44,74){\scriptsize $N$}
\put(65,74){\scriptsize $M$}
\put(85,74){\scriptsize $U_i^\vee$}
\put(135,74){\scriptsize $U_i$}
\put(36,-9){\scriptsize $M$}
\put(56,-9){\scriptsize $N$}
\put(76,-9){\scriptsize $U_i^\vee$}
\put(135,-9){\scriptsize $U_i$}
\put(32,58){\scriptsize $\sigma_N$}
\put(95,58){\scriptsize $\theta_{U_i^\vee}$}
\ep
}
\ee
Again powers of $\sigma$ are changed by taking into account
Dehn twists.

By theorem \ref{thm:GMF}, our construction yields a $G$-equivariant
modular functor, and by the general results of \cite{kpGMF}, 
braiding morphisms obtained from a 
$G$-equivariant modular functor satisfy $\ZZ$-equivariant 
generalizations of the hexagon axioms. We have also directly
verified that the morphisms presented in this subsection
satisfy the hexagon axioms.

\subsection{Equivariant ribbon structure}\label{sec.4.7}

We now return to study the existence of a
\ZZ-equivariant ribbon structure on \CZ. The general results
of \cite{kpGMF} ensure that \CZ has a twist but do not
guarantee the existence of duality morphisms like 
the evaluation $D_U:U\otimes U^\ast\to\unit$. 
However, in the case of $\ZZ$-permutation equivariant categories, there are indeed 
compatible duality morphisms that endow $\CZ$ with 
a ribbon structure.

We obtain \cite[Section 7.2]{kpGMF} the twist morphism $\Theta_U:U\to {}^pU$ for $U\in\CZ_p$ 
by the Yoneda lemma from a natural transformation of functors
\be
\begin{split}
\tauzz(S_2(p,p^{-1};1,1);U,T)&\stackrel{(\tilde\phi_\B^{-1})_\ast}{\to}\tauzz(S_2(p^{-1},p;1,p^{-1});T,U)\\
&\stackrel{T_{(1,p)}}{=}\tauzz(S_2(p^{-1},p;1,1);T,{}^pU)\\
&\stackrel{(\tilde\phi_\Z)_\ast}{\to}\tauzz(S_2(p,p^{-1};1,1);{}^pU,T) \,\, .
\end{split}
\ee
Here  $\phi_\B$ is the braiding of the two holes of $S_2$ and 
$\tilde\phi_\B$ its lift to the cover $S_2(p^{-1},p;1,p^{-1})$. 
The equivariance morphism $T_{(1,p)}$ is, in our case, the identity.
Finally, $\phi_\Z$ is the diffeomorphism of the standard sphere inducing
a cyclic move of the distinguished edge, and $\tilde\phi_\Z$ is its
lift.

For an object $U=A_1\times A_2$ in the untwisted component of $\CZ$, 
the total space $S_2(1,1;1,1)$ of the cover is again just a disjoint union 
of two copies of $S_2$ with the lifts of $\phi_\B^{-1}$ and $\phi_\Z$ being 
the application of the diffeomorphisms to both components separately. Hence the twist on $\CZ_1$ is 
just the usual twist on the category \CC:
\be
\Theta_{A_1\times A_2}=\theta_{A_1}\otik\theta_{A_2}
\ee

To calculate the twist morphism for an object $U=M\in\CZ_g$ in the twisted component,
our moves amount to
\be
\raisebox{20pt}{
\bp(400,146)
\put(0,0){\pic{S2-circle-paths}}
\put(140,0){\pic{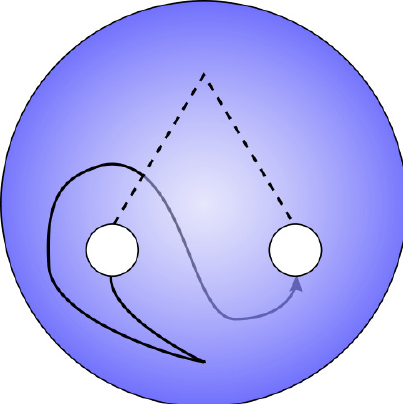}}
\put(280,0){\pic{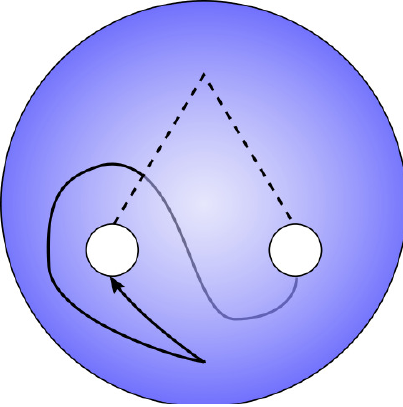}}
\put(123,52){$\stackrel{\tilde\phi_\B^{-1}}{\to}$}
\put(263,52){$\stackrel{\tilde\phi_\Z}{\to}$}
\put(83,42){\scriptsize $T$}
\put(28,42){\scriptsize $M$}
\put(221,42){\scriptsize $M$}
\put(169,42){\scriptsize $T$}
\put(360,42){\scriptsize $M$}
\put(308,42){\scriptsize $T$}
\ep
}
\ee

Transforming the total space of this cover into the standard sphere $S_2$ 
gives the following marking:

\be
\raisebox{20pt}{
\bp(400,146)
\put(60,0){\pic{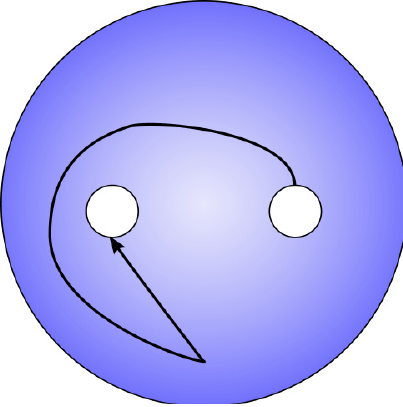}}
\put(220,0){\pic{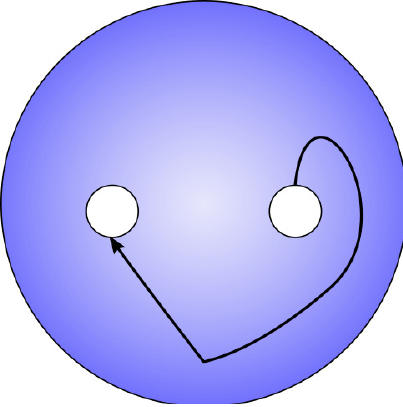}}
\put(193,52){$=$}
\put(87,53){\scriptsize $T$}
\put(142,53){\scriptsize $M$}
\put(300,53){\scriptsize $M$}
\put(248,53){\scriptsize $T$}
\ep
}
\ee
Here, we redrew the figure by pulling the line connecting to $M$ along the
backside of the sphere. The last figure is transformed into the standard 
marking on $S_2$ by applying the half turn move $\sigma_M$. 
Hence we find
\be
\Theta_M=\sigma_M \,\, \, .
\label{141}
\ee
Permutation orbifolds of rational conformal field theories have been
analyzed with representation theoretic tools in \cite{bohs}. The formula
(\ref{141}) for the twist in the twisted component is in full agreement with formula
\cite[(4.21)]{bohs} for the conformal weights.

Up to this point, we have derived all structure on \CZ from 
the approach of \ZZ-equivariant modular functors \cite{kpGMF}. In fact, we have
fully exploited this ansatz and obtained the structure of a 
weakly rigid \ZZ-equivariant monoidal category.
We next check that the condition in proposition \ref{prop:conditionrigid} 
that ensures that the tensor category \CZ is even rigid is satisfied.
The criterion is easy to check for simple objects $U_i\times U_j$ in the 
untwisted component of \CZ. In this case the morphism
\be
i_{U_i\times U_j}:\unit\times\unit\to U_iU_i^\vee\times U_jU_j^\vee
\ee
is just given by the tensor product of two coevaluations, 
\be
i_{U_i\times U_j}=b_{U_i}\otik b_{U_j} \,\, . 
\ee
The morphisms in proposition \ref{prop:conditionrigid} $a_{i\times j}$ are non-zero by rigidity of \CC. The ribbon 
structure on the neutral component $\CZ_1$ is then just the usual ribbon 
structure on \CC.

Now we look at a simple object $U_i$ in $\CZ_g$. As we have seen before, the dual object in \CX of $U_i$ is $U_i^\vee$. The adjunction \erf{eqn:Radj} with $T_1=T_2=\unit$ and $N=U_i^\vee$ gives isomorphisms
\be
\homc(U_i^\vee,U_i^\vee)\cong\homc(\unit,U_iU_i^\vee)\cong\homcc(\unit\times\unit,R(U_iU_i^\vee))
\ee
Evaluating the isomorphism on $\id_{U_i^\vee}$ gives a morphism in 
$\bigoplus_{j\in\I}\homcc(\unit\times\unit,U_iU_i^\vee U_j^\vee\times U_j)$
whose only non-vanishing component appears for $j=0$. It reads
\be
\raisebox{15pt}{
\bp(160,90)
\put(40,0){\rib{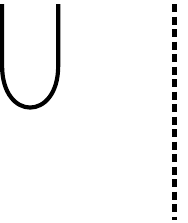}}
\put(65,30){$\otik$}
\put(-15,30){$i_{U_i}:=$}
\put(88,-11){\scriptsize $\unit$}
\put(88,69){\scriptsize $\unit$}
\put(36,69){\scriptsize $U_i$}
\put(52,69){\scriptsize $U_i^\vee$}
\ep
}
\ee

Now compute 
$\alpha_{U_i^\vee,U_i,U_i^\vee}^{-1}\circ (\id_{U_i^\vee}\otimes i_{U_i})$
using equation (\ref{eq:103}) for the associativity 
constraint:
\be
\raisebox{25pt}{
\bp(160,90)
\put(130,0){\rib{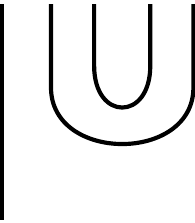}}
\put(-70,35){$\alpha_{U_i^\vee,U_i,U_i^\vee}^{-1}\circ (\id_{U_i^\vee}\otimes i_{U_i})=\quad\bigoplus_{k\in\I}\frac{d_k}{\catdim}$}
\put(168,68){\scriptsize $U_k$}
\put(153,68){\scriptsize $U_k^\vee$}
\put(141,68){\scriptsize $U_i$}
\put(126,68){\scriptsize $U_i^\vee$}
\put(182,68){\scriptsize $U_i^\vee$}
\put(126,-10){\scriptsize $U_i^\vee$}
\ep
}
\ee
Here $d_k$ is the dimension of the simple object $U_k$ and
$\catdim$ is the dimension of the category $\C$ introduced
in \erf{catdim}.
The morphism
$$ a_i: \unit\times\unit \to \bigoplus_k U_i^\vee U_i U_k^\vee \times U_k $$
introduced in proposition \ref{prop:conditionrigid} can only have a 
non-vanishing component for $k=0$,
$$ a^{(0)}_i:  \unit\times\unit \to  U_i^\vee U_i\times \unit \,\, . $$
Since the tensor unit $\unit$ is absolutely simple, i.e.\ $\mathrm{End}(\unit)
= \Bbbk \id_\unit$, this component $a^{(0)}_i$ is of the form
$$ a^{(0)}_i=\overline{a_i}\otik \id_\unit \,\, $$
with $\overline{a_i}\in \homc(\unit,U_i^\vee U_i)$. We then have
\be
\raisebox{25pt}{
\bp(160,90)
\put(70,0){\rib{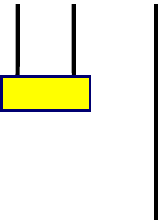}}
\put(-80,30){$\alpha_{U_i^\vee,U_i,U_i^\vee}^{-1}\circ (\id_{U_i^\vee}\otimes i_{U_i})=$}
\put(130,30){$+\sum_{k\neq 0} \,\, \dots$}
\put(98,36){\scriptsize $\overline{a_i}$}
\put(71,68){\scriptsize $U_i^\vee$}
\put(112,68){\scriptsize $U_i^\vee$}
\put(112,-10){\scriptsize $U_i^\vee$}
\put(88,68){\scriptsize $U_i$}
\ep
}
\ee
Hence
\be
\raisebox{25pt}{
\bp(160,90)
\put(0,0){\rib{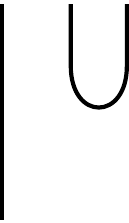}}
\put(80,0){\rib{duality-1}}
\put(-22,31){$\frac{1}{\catdim}$}
\put(55,31){$=$}
\put(109,36){\scriptsize $\overline{a_i}$}
\put(-3,68){\scriptsize $U_i^\vee$}
\put(-3,-10){\scriptsize $U_i^\vee$}
\put(121,-10){\scriptsize $U_i^\vee$}
\put(17,68){\scriptsize $U_i$}
\put(33,68){\scriptsize $U_i^\vee$}
\put(121,68){\scriptsize $U_i^\vee$}
\put(98,68){\scriptsize $U_i$}
\put(80,68){\scriptsize $U_i^\vee$}
\ep
}
\ee
To determine $\overline{a_i}$, 
we take a partial trace on both sides an arrive at

\be
\raisebox{25pt}{
\bp(160,90)
\put(0,0){\rib{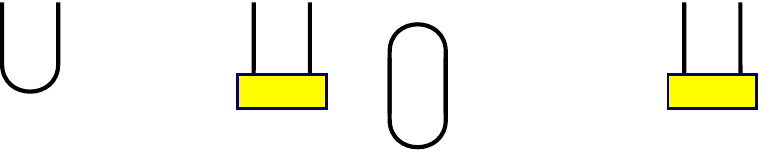}}
\put(-22,16){$\frac{1}{\catdim}$}
\put(41,16){$=$}
\put(141,16){$=d_i$}
\put(96,14){\scriptsize $\overline{a_i}$}
\put(222,14){\scriptsize $\overline{a_i}$}
\put(-5,48){\scriptsize $U_i^\vee$}
\put(15,48){\scriptsize $U_i$}
\put(68,48){\scriptsize $U_i^\vee$}
\put(193,48){\scriptsize $U_i^\vee$}
\put(210,48){\scriptsize $U_i$}
\put(86,48){\scriptsize $U_i$}
\put(117,18){\scriptsize $U_i$}
\ep
}
\ee
so that
\be
\raisebox{15pt}{
\bp(160,90)
\put(-3,12){\rib{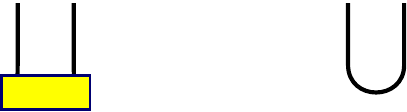}}
\put(37,26){$=\frac{1}{\catdim d_i}$}
\put(-12,14){\scriptsize $\overline{a_i}$}
\put(-5,48){\scriptsize $U_i^\vee$}
\put(15,48){\scriptsize $U_i$}
\put(92,48){\scriptsize $U_i^\vee$}
\put(110,48){\scriptsize $U_i$}
\ep
}
\ee
and hence
\be
\raisebox{25pt}{
\bp(160,90)
\put(40,0){\rib{duality-adjunction}}
\put(-30,30){$a_i=\quad\frac{1}{\catdim d_i}$}
\put(65,30){$\otik$}
\put(37,68){\scriptsize $U_i^\vee$}
\put(53,68){\scriptsize $U_i$}
\put(88,68){\scriptsize $\unit$}
\put(88,-9){\scriptsize $\unit$}
\ep
}
\ee
This morphism is non-zero since the coevaluation of $\C$ is non-zero.
Hence the $\ZZ$-permutation equivariant tensor category $\CZ$ is rigid.

The evaluation morphisms $D_{U_i}:U_i^\vee\otimes U_i\to\unit\times\unit$ 
are fixed by the condition $D_{U_i}\circ a_i=\id_{\unit\times\unit}$. 
It is easy to see that the right evaluation reads 

\be
\raisebox{25pt}{
\bp(160,90)
\put(30,0){\rib{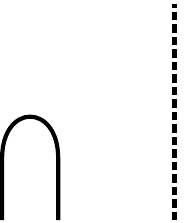}}
\put(-40,25){$D_{U_i}=\quad \catdim $}
\put(58,25){$\otik$}
\put(100,25){$\in\homcz(U_i^\vee\otimes U_i,\unit\times\unit)$}
\put(26,-10){\scriptsize $U_i^\vee$}
\put(45,-10){\scriptsize $U_i$}
\put(79,68){\scriptsize $\unit$}
\put(79,-9){\scriptsize $\unit$}
\ep
}
\ee
Similarly we find for the left evaluation
\be
\raisebox{25pt}{
\bp(160,90)
\put(31,0){\rib{duality-evaluation}}
\put(-40,25){$\tilde D_{U_i}=\quad \catdim$}
\put(58,25){$\otik$}
\put(100,25){$\in\homcz(U_i\otimes U_i^\vee,\unit\times\unit)$}
\put(26,-10){\scriptsize $U_i$}
\put(45,-10){\scriptsize $U_i^\vee$}
\put(79,68){\scriptsize $\unit$}
\put(79,-9){\scriptsize $\unit$}
\ep
}
\ee

\newpage
\bibliography{bib}{}
\bibliographystyle{alpha}
\end{document}